\documentclass[10pt,reqno]{amsart}
\usepackage{amssymb,mathrsfs,graphicx,extpfeil,amsmath,bbm}
\usepackage{ifthen,latexsym,float,colortbl}
\usepackage[margin=1in]{geometry}
\usepackage{caption}
\usepackage{dsfont}
\usepackage{arydshln}
\usepackage{sidecap}
\usepackage{rotating,dsfont,stackengine}
\usepackage{hyperref}
\usepackage{enumerate}
\usepackage[scr=boondox]{mathalfa}
\usepackage{xcolor}
\usepackage{lineno}

\usepackage[normalem]{ulem}
\usepackage{soul}
\usepackage{cancel}

\usepackage{colortbl}
\definecolor{black}{rgb}{0.0, 0.0, 0.0}
\definecolor{red}{rgb}{1.0, 0.5, 0.5}
\provideboolean{shownotes} 
\setboolean{shownotes}{true} 
\newcommand{\margnote}[1]{
	\ifthenelse{\boolean{shownotes}}%
	{\marginpar{\raggedright\tiny\texttt{#1}}}%
	{}%
}
\newcommand{\hole}[1]{
	\ifthenelse{\boolean{shownotes}}%
	{\begin{center} \fbox{ \rule {.25cm}{0cm} \rule[-.1cm]{0cm}{.4cm}
				\parbox{.85\textwidth}{\begin{center} \texttt{#1}\end{center}} \rule
				{.25cm}{0cm}}\end{center}} {} }

\graphicspath{{pics/}}

\title[Nonlinear quantum Fokker--Planck equation near equilibrium]
{Nonlinear quantum Fokker--Planck equation near equilibrium}

 \author[Choi]{Young-Pil Choi}
 \address[Young-Pil Choi]{\newline Department of Mathematics,\newline
 Yonsei University, 50 Yonsei-Ro, Seodaemun-Gu, Seoul 03722, Republic of Korea}
 \email{ypchoi@yonsei.ac.kr}

\author[Hwang]{Byung-Hoon Hwang}
\address[Byung-Hoon Hwang]{\newline Department of Mathematics Education\newline
	Sangmyung University, 20 Hongjimun 2-gil, Jongno-Gu, Seoul 03016, Republic of Korea}
\email{bhhwang@smu.ac.kr}

 \author[Hyun]{Ju-Hwan Hyun}
 \address[Ju-Hwan Hyun]{\newline Department of Mathematics\newline
 Yonsei University, 50 Yonsei-Ro, Seodaemun-Gu, Seoul 03722, Republic of Korea}
 \email{jhhyun@yonsei.ac.kr}

\makeatother

\numberwithin{equation}{section}

\newtheorem{theorem}{Theorem}[section]
\newtheorem{lemma}{Lemma}[section]

\newtheorem{proposition}{Proposition}[section]
\newtheorem{remark}{Remark}[section]

\makeatletter
\def\@subsubsection#1{\par
  \addpenalty\@secpenalty
  \vskip 1.5ex plus 1ex minus .2ex
  \refstepcounter{subsubsection}
  \addcontentsline{toc}{subsubsection}{\protect\numberline{\thesubsubsection}#1}
  \noindent{\normalfont\bfseries\thesubsubsection\quad#1\par\nobreak}
  \vskip 0.5ex \@afterheading}
\makeatother

\newcommand{\lt}{\left}
\newcommand{\rt}{\right}
\newcommand{\lal}{\langle}
\newcommand{\ral}{\rangle}

\newcommand{\mI}{\mathbb{I}}
\newcommand{\mN}{\mathbb{N}}

\newcommand{\mQ}{\mathbb{Q}}
\newcommand{\mR}{\mathbb{R}}

\newcommand{\mcC}{\mathcal{C}}
\newcommand{\mcD}{\mathcal{D}}
\newcommand{\mcE}{\mathcal{E}}
\newcommand{\mcF}{\mathcal{F}}

\newcommand{\mcI}{\mathcal{I}}
\newcommand{\mcJ}{\mathcal{J}}
\newcommand{\mcN}{\mathcal{N}}
\newcommand{\mcP}{\mathcal{P}}

\newcommand{\ph}{\psi_\hbar}

\newcommand{\mfe}{\mathfrak{e}}

\newcommand{\intr}{\int_{\mR^3}}
\newcommand{\inttr}{\iint_{\mR^3 \times \mR^3}}
\newcommand{\Ls}{\Lambda^{-\tilde{s}}}

\newcommand{\Q}{\mathbb Q}
\newcommand{\R}{\mathbb R}

\newcommand{\bq}{\begin{equation}}
\newcommand{\eq}{\end{equation}}
\newcommand{\e}{\varepsilon}
\newcommand{\pa}{\partial}

\newcommand{\calA}{\mathcal A}

\newcommand{\calC}{\mathcal C}
\newcommand{\calD}{\mathcal D}
\newcommand{\calE}{\mathcal E}

\newcommand{\calI}{\mathcal I}
\newcommand{\calJ}{\mathcal J}

\newcommand{\calL}{\mathcal L}
\newcommand{\calM}{\mathcal M}

\newcommand{\calR}{\mathcal R}

\newcommand{\calT}{\mathcal T}

\newcommand{\rd}{\textnormal{d}}

\newcommand{\rdx}{\textnormal{d}x}

\newcommand{\rdt}{\textnormal{d}t}

\newcommand{\rdta}{\textnormal{d}\tau}
\newcommand{\rds}{\textnormal{d}s}
\newcommand{\rdp}{\textnormal{d}p}

\begin{document}

\allowdisplaybreaks

\date{\today}


\begin{abstract}
We investigate a nonlinear quantum Fokker--Planck equation with self-consistent collision frequency, bulk velocity, and temperature. In contrast to quantum Fokker--Planck equations with prescribed diffusion and friction coefficients, the macroscopic quantities are nonlinear functionals of the distribution function. The equation preserves mass, momentum, and kinetic energy, admits a quantum entropy dissipation structure, and propagates the Pauli admissible range in the fermionic case. Its collision operator is also formally connected to the quantum Landau equation. For the Cauchy problem in the three-dimensional whole space, we prove the global-in-time existence and uniqueness of strong solutions near a global quantum equilibrium. The proof is based on a perturbative macro--micro energy method that combines microscopic coercivity, estimates for nonlinear velocity moments, and a macroscopic dissipation argument. We further establish the propagation of nonnegativity and the fermionic Pauli upper bound. Under an additional negative Sobolev assumption on the initial perturbation, we obtain algebraic decay rates toward equilibrium.
\end{abstract}

\maketitle \centerline{\date}

\tableofcontents

%
%
%
%
%
%
%
%
%
%

\section{Introduction} 

The Fokker--Planck equation is one of the fundamental equations in statistical physics and kinetic theory. It describes the evolution of a distribution function under the combined effects of transport, diffusion, and friction. It is also closely related to the Landau equation, which arises as the grazing-collision limit of the Boltzmann equation; see, for instance, \cite{CDW22,DR01,Gou97}. Owing to this connection, Fokker--Planck-type equations retain important structural features of collisional kinetic models while remaining more tractable from the analytical point of view. We refer to \cite{Cha03,Cha04_2,Cha04,Kan01} for derivations and structural aspects of kinetic equations in the physics literature.

The Cauchy problem for kinetic equations has been extensively studied. For general background on the mathematical theory of kinetic equations and their asymptotic analysis, we refer to \cite{BGP00,Gla96,Vil02}. In particular, the global existence of perturbative solutions to the Boltzmann equation near equilibrium has been investigated in \cite{Dua08,Guo04,LYY04} and related works. For the Landau equation, we refer to \cite{Vil96,Vil02} and the references therein. We also mention the recent breakthrough result \cite{GS25}, which rules out finite-time blow-up for the spatially homogeneous Landau equation, including the Coulomb case. For the spatially inhomogeneous Landau and Boltzmann equations, the recent work \cite{GHS26} establishes new weighted pointwise estimates and continuation criteria without assuming bounds on the hydrodynamic quantities. The trend to equilibrium for the linear kinetic Fokker--Planck equation has been studied by various methods, including entropy and hypocoercivity arguments; see, for instance, \cite{DV01,DMS15}. Nonlinear kinetic Fokker--Planck equations with self-consistent macroscopic fields have also been considered. In \cite{Cho16}, global classical solutions near equilibrium and their large-time behavior were established for a nonlinear Vlasov--Fokker--Planck equation in which the diffusion coefficient is fixed while the drift velocity is determined by a local velocity average of the distribution function. The stability of the global Maxwellian for a fully nonlinear Fokker--Planck equation preserving mass, momentum, and energy was studied in \cite{LY21}, while global weak solutions with large initial data were constructed for a related equation in \cite{CG04, CHY25, CS25}. We also refer to \cite{Bed17,BZZ25} for related results on Vlasov--Poisson equations with nonlinear Fokker--Planck collisions in the weakly collisional regime.
 
Quantum kinetic equations incorporate quantum statistics through inclusion or exclusion factors. Such equations were introduced and studied in the physics literature; see, for instance, \cite{Kan95,KQ93,KQ94}. The sign of the statistical parameter distinguishes bosonic and fermionic statistics. For the quantum Boltzmann equation, spatially homogeneous Fermi--Dirac models were studied in \cite{Lu01,LW03}. Global well-posedness results near equilibrium or vacuum for spatially inhomogeneous quantum Boltzmann equations were obtained in \cite{BJY21,OW22,WXZ23}; see also the references therein. The quantum Landau equation has also been studied extensively, particularly in the fermionic case. For the spatially homogeneous Landau--Fermi--Dirac equation, well-posedness results were established for hard potentials in \cite{Bag04}, for moderately soft potentials in \cite{ABDL22}, and for the Coulomb potential in \cite{GGZ22}. The long-time behavior for hard potentials was investigated in \cite{ABL21}, while entropy production estimates adapted to the Fermi--Dirac structure were developed in \cite{ABDL21}. These estimates are uniform with respect to the quantum parameter and recover corresponding estimates for the classical Landau equation in the classical limit. More recently, global weak solutions to the spatially inhomogeneous Landau--Fermi--Dirac equation with Coulomb potential were constructed in \cite{Sam24}. The semiclassical limit of these solutions as the quantum parameter tends to zero was subsequently studied in \cite{Sam25}. We also refer to the recent work \cite{GPTW25}, where the quantum Landau operator is rigorously derived from the quantum Boltzmann operator through a weak-coupling limit for both Fermi--Dirac and Bose--Einstein statistics.
 
Quantum Fokker--Planck equations posed on a classical kinetic phase space have been studied in several settings. Well-posedness and long-time behavior for spatially homogeneous models were investigated in \cite{CCLR16,CLR09,CRS08}. Linearized quantum and relativistic Fokker--Planck--Landau equations were considered in \cite{Lem00}. For spatially inhomogeneous kinetic models, perturbative global classical solutions and stability of equilibria were established in \cite{LZ15,NS05,NS07}. A recent work \cite{APT26} proves exponential convergence toward equilibrium for a related class of spatially inhomogeneous bosonic and fermionic Fokker--Planck equations, under the assumption that global solutions exist. 

A typical equation considered in this context is
\bq\label{linear}
\pa_tf+p\cdot\nabla_x(f+\sigma\hbar\kappa f^2)=\nabla_p\cdot\lt\{\nabla_p f+pf(1+\hbar\kappa f)\rt\},
\eq
where $\sigma\in\{0,1\}$ distinguishes two related equations studied in the literature. When $\sigma=1$, the quantum correction appears in both the transport term and the collision operator, whereas, when $\sigma=0$, the transport term retains its classical form and the quantum correction appears only in the collision operator. These two equations share the same equilibrium distributions; see \cite{NS07} for further details. Here and below, $\kappa=-1$ corresponds to fermions, $\kappa=1$ to bosons. Formally, the choice $\kappa=0$ recovers the classical case. Throughout the paper, we focus on the quantum cases $\kappa\in\{-1,1\}$. The parameter $\hbar>0$ is a dimensionless quantum parameter associated with the reduced Planck constant after rescaling. More precisely, in the quantum Fokker--Planck--Landau model considered in \cite{HJY12}, the corresponding statistical factor is written as $f(1\pm\vartheta f)$, where $\vartheta=\bar\hbar^{d_v}$, $\bar\hbar$ denotes the rescaled Planck constant, and $d_v$ is the velocity dimension. In the present three-dimensional setting, this corresponds to $\vartheta=\bar\hbar^3$. We use the notation $\hbar$ for the resulting dimensionless coefficient appearing in the quantum statistical factor. The classical regime is formally recovered as $\hbar\to0$. In equations of the form \eqref{linear}, the diffusion coefficient and friction field are prescribed independently of the macroscopic quantities of the distribution. 

Another class of quantum Fokker--Planck equations is formulated in terms of Wigner functions or quantum dynamical semigroups. These models are related to open quantum systems and are structurally different from the nonlinear kinetic equations considered above. We refer to \cite{AFN08,AGGMMS12} and the references therein for the Lindblad and Wigner approaches and for the analysis of stationary states and large-time behavior of Wigner--Fokker--Planck equations.
 
%
%
%
%
%
%
%
%
%
%

\subsection{Nonlinear quantum Fokker--Planck equation}

The purpose of the present paper is to study a nonlinear quantum Fokker--Planck equation in which the collision frequency, bulk velocity, and temperature are determined self-consistently by the distribution function. More precisely, we consider
\bq\label{eq:main}
\pa_tf+p\cdot\nabla_xf=\rho\nabla_p\cdot\lt\{\Theta\nabla_p f+(p-u)f(1+\hbar\kappa f)\rt\},
\eq
where $f=f(t,x,p)$ denotes the one-particle distribution function at time $t\ge0$ and phase point $(x,p)\in\mR^3\times\mR^3$. The macroscopic quantities are defined by
\[
\rho(t,x)=\intr f\,\rdp,\quad u(t,x)=\frac{\intr pf(1+\hbar\kappa f)\,\rdp}{\intr f(1+\hbar\kappa f)\,\rdp},\quad \Theta(t,x)=\frac{\intr(|p|^2-u\cdot p)f(1+\hbar\kappa f)\,\rdp}{3\intr f\,\rdp}.
\]

The main feature of \eqref{eq:main} is the nonlinear feedback between the distribution function and the macroscopic quantities $\rho$, $u$, and $\Theta$. In contrast to the representative quantum Fokker--Planck equations of the form \eqref{linear}, in which the diffusion coefficient and friction field are prescribed, the bulk velocity and diffusion temperature in \eqref{eq:main} depend nonlinearly on quantum-weighted moments of the distribution, while the collision frequency is given by the local density. This self-consistent structure makes \eqref{eq:main} compatible with the conservation of mass, momentum, and kinetic energy. The equation also satisfies a quantum entropy dissipation identity and, in the fermionic case, preserves the Pauli admissible range.

The particular form of the collision operator in \eqref{eq:main} is not introduced merely for analytical convenience. In Section \ref{ssec:f_deri}, we show that it arises naturally from a formal reduction of the quantum Landau operator in the Maxwellian-molecule case under a radial symmetry ansatz around the quantum-weighted bulk velocity. The self-consistent macroscopic fields $u$ and $\Theta$ are selected from the quantum-weighted moments in a manner consistent with the conservation of mass, momentum, and kinetic energy.

As shown in Section \ref{ssec:lqe} below, the local quantum equilibria associated with \eqref{eq:main} are given by
\[
\mcF(\theta,u,\Theta;p)=\lt(e^{\frac{|p-u(t,x)|^2}{2\Theta(t,x)}+\theta(t,x)}-\hbar\kappa\rt)^{-1}.
\]
For a fixed constant $\theta_0>0$, we define the global equilibrium
\bq\label{F_0}
\mcF_\hbar (p):=\mcF(\theta_0,0,1;p)=\lt(e^{\frac{|p|^2}{2}+\theta_0}-\hbar\kappa\rt)^{-1}.
\eq
Throughout the paper, we assume that
\bq\label{eq:ht-cond}
\hbar e^{-\theta_0}<1.
\eq
In the bosonic case, this condition prevents the denominator in \eqref{F_0} from vanishing. In the fermionic case, it also ensures the positivity of a coefficient appearing in the perturbative formulation.

%
%
%
%
%
%
%
%
%
%

\subsection{Main results}

We study perturbations of the global equilibrium $\mcF_\hbar $ of the form
\[
f=\mcF_\hbar +\sqrt{\mu_\hbar }\,g,\quad \mu_\hbar :=\mcF_\hbar +\hbar\kappa\mcF_\hbar ^2.
\]
The weight $\mu_\hbar $ is naturally associated with the quantum statistical factor and reveals the symmetric structure of the linearized collision operator. Substituting this perturbation ansatz into \eqref{eq:main}, we obtain
\bq\label{IVP}
\pa_tg+p\cdot\nabla_xg=L(g)+\Gamma(g),
\eq
where $L$ is the linearized collision operator and $\Gamma$ denotes the nonlinear remainder. Their precise forms are derived in Section \ref{ssec:pre_mac}.

The contributions of this work are both structural and analytical. First, we provide a formal reduction of the quantum Landau operator in the Maxwellian-molecule case that explains the particular self-consistent definitions of the collision frequency, bulk velocity, and diffusion temperature in \eqref{eq:main}. The resulting operator retains the Bose--Einstein or Fermi--Dirac statistical correction while preserving mass, momentum, and kinetic energy and admitting the corresponding quantum entropy structure. Second, we develop a global perturbative well-posedness theory for this fully self-consistent quantum Fokker--Planck equation in the whole space. In contrast to models with prescribed drift and diffusion coefficients, the collision frequency is determined by a velocity moment of the distribution, while the bulk velocity and diffusion temperature depend nonlinearly through rational expressions of quantum-weighted moments. Our analysis also propagates the nonnegativity of the distribution and, in the fermionic case, the Pauli upper bound. These admissibility properties are proved dynamically and are not inferred solely from the smallness of the perturbation. Third, we identify the nonstandard structure of the linearized collision operator. The self-consistent bulk velocity and temperature produce finite-rank correction terms whose generating modes do not directly coincide with the canonical momentum and energy modes associated with the collision invariants. By exploiting exact identities satisfied by the quantum equilibrium weights, we characterize the five-dimensional null space of the linearized operator and establish coercivity on its orthogonal microscopic complement. Finally, we combine the microscopic coercivity with a macroscopic dissipation argument and negative Sobolev estimates to derive a hierarchy of algebraic decay rates toward equilibrium in the whole space.

For the negative Sobolev norm in the spatial variable, we use the notation
\[
\|g\|_{L^2_p\dot H^{-\tilde s}_x}:=\lt(\intr\|\Lambda_x^{-\tilde s}g(\cdot,p)\|_{L^2_x}^2\,\rdp\rt)^{\frac12},
\]
where $\Lambda_x^{-\tilde s}$ denotes the Fourier multiplier with symbol $|\xi|^{-\tilde s}$.

\begin{theorem}\label{ge}
Let $s\ge4$. There exists a sufficiently small constant $\delta_{\rm in}>0$ such that, for any initial datum $g_0\in H^s(\R^3\times\R^3)$ satisfying  $\|g_0\|_{H^s}\le\delta_{\rm in}$ and that  $f_0(x,p):=\mcF_\hbar (p)+\sqrt{\mu_\hbar (p)}\,g_0(x,p)\ge0$ for almost every $(x,p)\in\mR^3\times\mR^3$, the Cauchy problem \eqref{IVP} with initial datum $g_0$ admits a unique global solution $g\in \calC([0,\infty);H^s(\mR^3\times\mR^3))$. Moreover,
\[
\sup_{t\ge0}\|g(t)\|_{H^s}\le C\|g_0\|_{H^s}.
\]
The corresponding distribution $f(t,x,p):=\mcF_\hbar (p)+\sqrt{\mu_\hbar (p)}\,g(t,x,p)$  satisfies $f(t,x,p)\ge0$ for almost every $(x,p)\in\mR^3\times\mR^3$ and every $t\ge0$.

In the fermionic case $\kappa=-1$, if in addition $f_0(x,p)\le\frac1\hbar$ for almost every $(x,p)\in\mR^3\times\mR^3$, then
\[
f(t,x,p)\le\frac1\hbar
\]
for almost every $(x,p)\in\mR^3\times\mR^3$ and every $t\ge0$.

If, in addition, $g_0\in L^2_p\dot H^{-\tilde s}_x$ for some $0<\tilde s<\frac32$, then $g \in L^\infty(0,\infty; L^2_p\dot H^{-\tilde s}_x)$. Furthermore, for every integer $l$ satisfying $0\le l\le s-1$, there exists a constant $C_l>0$ such that
\[
 \sum_{l\le|\alpha|\le s}\|\pa_x^\alpha g(t)\|_{L^2_{x,p}}^2 \le C_l(1+t)^{-(l+\tilde s)}
\]
for all $t\ge0$.
\end{theorem}

\begin{remark}
The restriction $0<\tilde s<\frac32$ in the large-time behavior result is intrinsic to the negative Sobolev estimate used in the proof. Indeed, in order to estimate the nonlinear term after applying $\Lambda_x^{-\tilde s}$, we use the Hardy--Littlewood--Sobolev inequality in the form
\[
\|\Lambda_x^{-\tilde s}F\|_{L^2_x}\le C\|F\|_{L^q_x},\quad \frac1q=\frac12+\frac{\tilde s}{3}.
\]
The strong-type estimate requires $1<q<2$, which is equivalent to $0<\tilde s<\frac32$. At the endpoint $\tilde s=\frac32$, one has $q=1$, and the corresponding strong-type Hardy--Littlewood--Sobolev estimate is no longer available. Consequently, the present $L^2$-based energy method does not extend directly to $\tilde s\ge\frac32$.
\end{remark}

The proof of Theorem \ref{ge} is based on a perturbative energy method in the spirit of the classical macro--micro framework for collisional kinetic equations. We first establish the local well-posedness of the perturbation equation and propagate the admissible bounds on the distribution function.   A central structural difficulty already appears at the linearized level. Since the self-consistent bulk velocity and temperature are defined through quantum-weighted moments, the linearized collision operator takes the form
\[
L=L_0+P_1,
\]
where $L_0$ is the dissipative Fokker--Planck part satisfying
\[
\ker L_0=\mathrm{span}\{\sqrt{\mu_\hbar}\},
\]
whereas $P_1$ is a finite-rank correction generated by the four quantum-weighted modes
\[
p_j\eta_\hbar\sqrt{\mu_\hbar}, \ j=1,2,3, \quad\text{and}\quad (|p|^2\eta_\hbar-3)\sqrt{\mu_\hbar}.
\]
These modes do not coincide directly with the canonical momentum and energy modes generated by the collision invariants. More precisely, the canonical macroscopic space associated with the collision invariants is
\[
\mcN = \mathrm{span}\{\sqrt{\mu_\hbar },\,p_1\sqrt{\mu_\hbar },\,p_2\sqrt{\mu_\hbar },\, p_3\sqrt{\mu_\hbar },\,|p|^2\sqrt{\mu_\hbar }\}.
\]
It is thus not immediate that the finite-rank correction $P_1$ recovers exactly the four additional collision-invariant directions in the null space, without introducing any spurious null modes. By exploiting exact moment identities and integration-by-parts relations satisfied by the quantum equilibrium weights, we prove that
\[
\ker L=\mcN
\]
and establish the coercivity of $L$ on the orthogonal microscopic complement $\mcN^\perp$; see Section \ref{sec:lin-str}. Since this coercivity acts directly only on the microscopic component, we recover the dissipation of the macroscopic variables from the balance relations satisfied by the coefficients of the macroscopic projection and suitable temporal interaction functionals. At the nonlinear level, the self-consistent fields $\rho$, $u$, and $\Theta$ generate quantum-weighted moment terms and rational expressions of these moments. Their estimates require uniform control of the denominators together with high-order Sobolev estimates for the resulting nonlinear composition terms. Finally, when velocity derivatives are included, derivatives fall on the quantum equilibrium weights and on the macroscopic projection, producing lower-order commutator terms. These terms are controlled through a weighted triangular hierarchy in the number of velocity derivatives, which allows us to close the global energy estimate.

The large-time analysis requires a further treatment of the low spatial frequencies. Since the problem is posed in the whole space, no Poincar\'e inequality is available, and the global energy estimate alone does not yield a quantitative decay rate. We thus propagate a negative Sobolev norm in the spatial variable and combine it with the higher-order dissipation estimates through interpolation inequalities, following the general strategy developed in \cite{GW12}. In the present equation, however, the negative Sobolev estimate is complicated by the fully nonlinear dependence of the collision frequency, bulk velocity, and temperature on the distribution function. In particular, after applying $\Lambda_x^{-\tilde s}$, one must control products involving quantum-weighted velocity moments, rational macroscopic coefficients, and the microscopic part of the solution in spatial Lebesgue spaces below $L^2_x$. We establish estimates for these terms by combining the Hardy--Littlewood--Sobolev inequality with the previously obtained high-order energy bounds and the smallness of the perturbation. This yields the uniform propagation of the negative Sobolev norm. Interpolating this negative regularity with the spatial dissipation recovered from the macro--micro system then produces a hierarchy of algebraic decay estimates for successive spatial derivatives. Thus, the decay argument is coupled essentially to the nonstandard coercive structure and the macroscopic dissipation mechanism developed for the self-consistent quantum collision operator, rather than being a direct application of the classical linear Fokker--Planck theory.

%
%
%
%
%
%
%
%
%
%
\subsection{Organization of the paper}

The rest of the paper is organized as follows. In Section \ref{sec:deri}, we formally derive the nonlinear quantum Fokker--Planck equation from the quantum Landau equation and discuss its structural properties. In Section \ref{sec:lin-str}, we derive the perturbative formulation, identify the macroscopic space, and establish the coercivity properties of the linearized operator. In Section \ref{sec:non-lwp}, we estimate the nonlinear terms and prove the local well-posedness of the perturbation equation. In Section \ref{sec:global-energy}, we establish the global energy estimate by combining microscopic coercivity with a macroscopic dissipation argument. In Section \ref{sec:gwp-lt}, we prove the global existence of perturbative solutions and derive their algebraic decay rates. Appendix \ref{app:lwp} contains the details of the local construction and the propagation of the admissible bounds. Appendix \ref{app:mac-diss} provides the coefficient-level calculations used in the macroscopic dissipation estimate.

%
%
%
%
%
%
%
%
%
%

\section{Derivation and structural properties}\label{sec:deri}

In this section, we explain the origin of the nonlinear quantum Fokker--Planck equation and provide its basic structural properties. The purpose is twofold. First, we formally derive the equation from the quantum Landau equation in the Maxwellian-molecule case under a radial symmetry ansatz around the quantum-weighted bulk velocity. Second, we show that the resulting equation preserves the fundamental features expected from quantum kinetic theory, namely conservation of mass, momentum, and energy, the local quantum equilibria, the entropy dissipation structure, and the Pauli exclusion principle in the fermionic case.

Throughout this section, for notational simplicity, we write
\[
F(f):=f(1+\hbar\kappa f).
\]
In the fermionic case $\kappa=-1$, this expression is understood on the Pauli admissible range, which will be discussed in Section \ref{ssec:pauli}.

%
%
%
%
%
%
%
%
%
%
\subsection{Formal derivation from the quantum Landau equation}\label{ssec:f_deri}

We begin with a formal derivation of the nonlinear quantum Fokker--Planck operator from the quantum Landau equation. Although the main analysis of this paper is carried out in dimension three, we present the derivation in general dimension $d\ge2$. This also makes clear the origin of the constants appearing in the reduced equation. The computation below is inspired by Villani's derivation of Fokker--Planck-type models from the Landau collision operator in the Maxwellian-molecule case \cite{Vil98}. In the present setting, we adapt this idea to the quantum Landau operator, keeping the quantum statistical factor $1+\hbar\kappa f$ throughout the derivation.

The quantum Landau operator \cite{ABDL21, ABDL22} is formally given by
\[
Q_{\rm L}(f) = \nabla_p\cdot \int_{\R^d} q(p-p^*) \lt\{ F(f_*)\nabla_p f - F(f)\nabla_{p^*}f_* \rt\} \rdp^*,
\]
where $f=f(p)$, $f_*=f(p^*)$, and
\[
q(z)=|z|^{\gamma+2} \lt(\mI_d-\frac{z\otimes z}{|z|^2}\rt), \quad \gamma=\frac{\nu-(2d-1)}{s-1}.
\]
In the Maxwellian-molecule case $\nu=2d-1$, we have $\gamma=0$, and thus
\[
q(p-p^*) = |p-p^*|^2\mI_d-(p-p^*)\otimes(p-p^*).
\]
In this case, the operator can be written as
\bq\label{eq:lan-AB}
Q_{\rm L}(f) = \sum_{i,j=1}^d\pa_{p_i} \lt\{\bar A_{ij}\pa_{p_j}f\rt\} + \sum_{i=1}^d\pa_{p_i} \lt\{\bar B_i F(f)\rt\},
\eq
where $\bar A_{ij} = \lt(|p|^2\delta_{ij}-p_i p_j\rt)*F(f)$, $\bar B_i = (d-1)p_i*f$. Here and below, the convolution is taken with respect to the velocity variable. More precisely,
\[
\lt(|p|^2\delta_{ij}-p_i p_j\rt)*F(f) := \int_{\R^d} \lt( |p-p^*|^2\delta_{ij} - (p_i-p_i^*)(p_j-p_j^*) \rt)F(f_*)\,\rdp^*
\]
and
\[
p_i*f := \int_{\R^d} (p_i-p_i^*)f_*\,\rdp^*.
\]

We introduce the classical and quantum-weighted macroscopic quantities
\begin{align*}
    \rho&=\int_{\R^d} f\,\rdp, & \rho U&=\int_{\R^d} pf\,\rdp, & d\rho \Theta+\rho|U|^2&=\int_{\R^d} |p|^2f\,\rdp, \cr
    \tilde\rho&=\int_{\R^d} F(f)\,\rdp, & \tilde\rho\tilde U&=\int_{\R^d} pF(f)\,\rdp, & d\rho\tilde \Theta+\tilde\rho|\tilde U|^2&=\int_{\R^d} |p|^2F(f)\,\rdp.
\end{align*}
Notice that $\tilde \Theta$ is normalized with respect to the classical mass $\rho$, not the quantum-weighted mass $\tilde\rho$. This is not the usual temperature associated with the measure $F(f)\,\rdp$; rather, it is the effective diffusion temperature selected so that the reduced nonlinear Fokker--Planck operator preserves the classical kinetic energy. Equivalently,
\[
d\rho\tilde \Theta = \int_{\R^d}|p-\tilde U|^2F(f)\,\rdp.
\] 
This convention is consistent with the nonlinear Fokker--Planck equation considered in this paper.

Using the identity
\[
\int_{\R^d} p_i p_j F(f)\,\rdp = \tilde\rho\,\tilde U_i\tilde U_j + \int_{\R^d} (p_i-\tilde U_i)(p_j-\tilde U_j)F(f)\,\rdp,
\]
we compute
\bq\label{eq:Aij-ex}
    \bar A_{ij} =  \lt(\tilde\rho|p-\tilde U|^2+d\rho\tilde \Theta\rt)\delta_{ij} - \tilde\rho(p_i-\tilde U_i)(p_j-\tilde U_j) - \int_{\R^d} (p_i^*-\tilde U_i)(p_j^*-\tilde U_j)F(f_*)\,\rdp^*.
\eq
Similarly, since
\[
\bar B_i=(d-1)\int_{\R^d} (p_i-p_i^*)f_*\,\rdp^* = (d-1)\rho(p_i-U_i),
\]
the drift part is determined by the classical bulk velocity $U$.

Substituting \eqref{eq:Aij-ex} into \eqref{eq:lan-AB}, we obtain
\begin{align}\label{eq:lan-ex}
\begin{aligned}
    Q_{\rm L}(f) &= \sum_{i,j=1}^d \lt(\tilde\rho|p-\tilde U|^2+d\rho\tilde \Theta\rt)\delta_{ij} \pa^2_{p_ip_j}f  - \sum_{i,j=1}^d \tilde\rho(p_i-\tilde U_i)(p_j-\tilde U_j) \pa^2_{p_ip_j}f \cr
    &\quad - \sum_{i,j=1}^d \int_{\R^d} (p_i^*-\tilde U_i)(p_j^*-\tilde U_j)F(f_*)\,\rdp^* \pa^2_{p_ip_j}f - (d-1)\tilde\rho(p-\tilde U)\cdot\nabla_p f \cr
    &\quad + d(d-1)\rho F(f) + (d-1)\rho(p-U)\cdot\nabla_p F(f).
\end{aligned}
\end{align}
Here the fourth term comes from differentiating the diffusion matrix $\bar A_{ij}$ with respect to $p$. Indeed,
\[
\sum_{i=1}^d\pa_{p_i}\bar A_{ij} = - (d-1)\tilde\rho(p_j-\tilde U_j).
\]
The last two terms come from expanding the drift part $ \nabla_p\cdot\{(d-1)\rho(p-U)F(f)\}$. 

We now impose a radial symmetry assumption around the quantum-weighted bulk velocity $\tilde U$. More precisely, we assume that
\bq\label{eq:rad-ans}
    f=\Phi(v), \quad v=\frac{|p-\tilde U|^2}{2\tilde \Theta}.
\eq
Then $F(f)=f(1+\hbar\kappa f)$ is also radially symmetric around $\tilde U$. Consequently, its second-order central moment is isotropic:
\bq\label{eq:iso-clo}
\int_{\R^d} (p-\tilde U)\otimes(p-\tilde U)F(f)\,\rdp = \rho\tilde \Theta\,\mI_d.
\eq
Equivalently,
\[
    \int_{\R^d} (p_i-\tilde U_i)(p_j-\tilde U_j)F(f)\,\rdp=0 \quad\text{if } i\neq j,
\]
and
\bq\label{eq:iso-clo-dia}
    \int_{\R^d} (p_i-\tilde U_i)^2F(f)\,\rdp=\rho\tilde \Theta \quad\text{for } i=1,\dots,d.
\eq
Indeed, by radial symmetry the matrix on the left-hand side of \eqref{eq:iso-clo} is a scalar multiple of $\mI_d$, and its trace is
\[
\int_{\R^d} |p-\tilde U|^2F(f)\,\rdp=d\rho\tilde \Theta
\]
by the definition of $\tilde \Theta$.

Under \eqref{eq:rad-ans}, we have
\bq\label{eq:Phi-derivatives}
    \pa_{p_i}f = \frac{p_i-\tilde U_i}{\tilde \Theta}\Phi', 
    \quad  
    \pa^2_{p_ip_j}f = \frac{(p_i-\tilde U_i)(p_j-\tilde U_j)}{\tilde \Theta^2}\Phi'' + \frac{\delta_{ij}}{\tilde \Theta}\Phi', 
    \quad
    \Delta_p f = \frac d{\tilde \Theta}\Phi' + \frac{|p-\tilde U|^2}{\tilde \Theta^2}\Phi''.
\eq
 Using \eqref{eq:iso-clo}, \eqref{eq:iso-clo-dia}, and \eqref{eq:Phi-derivatives} in \eqref{eq:lan-ex}, we get
\begin{align*}
Q_{\rm L}(f) &= \lt(\tilde\rho|p-\tilde U|^2+d\rho\tilde \Theta\rt) \lt( \frac d{\tilde \Theta}\Phi' + \frac{|p-\tilde U|^2}{\tilde \Theta^2}\Phi'' \rt) - \tilde\rho \lt( \frac{|p-\tilde U|^4}{\tilde \Theta^2}\Phi'' + \frac{|p-\tilde U|^2}{\tilde \Theta}\Phi' \rt) \cr
&\quad - \rho\tilde \Theta \lt( \frac{|p-\tilde U|^2}{\tilde \Theta^2}\Phi'' + \frac d{\tilde \Theta}\Phi' \rt) - (d-1)\tilde\rho\frac{|p-\tilde U|^2}{\tilde \Theta}\Phi' \cr
&\quad + d(d-1)\rho F(f) + (d-1)\rho(p-U)\cdot\nabla_p F(f).
\end{align*}
The terms involving $\tilde\rho$ cancel. This implies
\[
Q_{\rm L}(f) = d(d-1)\rho\Phi' + (d-1)\rho\frac{|p-\tilde U|^2}{\tilde \Theta}\Phi'' + d(d-1)\rho F(f) + (d-1)\rho(p-U)\cdot\nabla_p F(f).
\]

Since
\[
\tilde \Theta\Delta_p f = d\Phi' + \frac{|p-\tilde U|^2}{\tilde \Theta}\Phi'',
\]
we obtain
\bq\label{eq:for-Q0}
Q_{\rm L}(f) = (d-1)\rho\tilde \Theta\Delta_p f + (d-1)\rho\nabla_p\cdot\{(p-U)F(f)\}.
\eq
Finally, under the radial ansatz \eqref{eq:rad-ans}, the classical and quantum-weighted bulk velocities coincide. Indeed, by the oddness,
\[
\rho U = \int_{\R^d} pf\,\rdp = \int_{\R^d} (q+\tilde U)\Phi\lt(\frac{|q|^2}{2\tilde \Theta}\rt)\,\textnormal{d}q = \rho\tilde U,
\]
and hence $U=\tilde U$. Therefore, \eqref{eq:for-Q0} becomes
\[
Q_{\rm L}(f) = (d-1)\rho\nabla_p\cdot \lt\{ \tilde \Theta\nabla_p f+(p-\tilde U)F(f) \rt\}.
\]
The dimensional constant $d-1$ can be absorbed into the normalization of the collision frequency. Thus, at the level of the collision operator, we are led to the normalized nonlinear quantum Fokker--Planck operator
\[
Q_{\rm L}(f) = \rho\nabla_p\cdot \lt\{ \tilde \Theta\nabla_p f+(p-\tilde U)F(f) \rt\}.
\]
For notational consistency with the rest of the paper, we write $u:=\tilde U$, $\Theta:=\tilde \Theta$, that is,
\[
\rho(t,x)=\int_{\R^d} f\,\rdp, 
\quad 
u(t,x)=\frac{\int_{\R^d} pf(1+\hbar\kappa f)\,\rdp}{\int_{\R^d} f(1+\hbar\kappa f)\,\rdp}, 
\quad 
\Theta(t,x)= \frac{\int_{\R^d} (|p|^2-u\cdot p)f(1+\hbar\kappa f)\,\rdp}{d\int_{\R^d} f\,\rdp}.
\]
Consequently, the corresponding inhomogeneous kinetic equation is
\bq\label{eq:nqfp-s2}
\pa_t f+p\cdot\nabla_x f = \rho\nabla_p\cdot \lt\{ \Theta\nabla_p f+(p-u)f(1+\hbar\kappa f) \rt\}.
\eq
In the rest of the paper, we work in dimension $d=3$, so that the denominator in the definition of $\Theta$ is $3\intr f\,\rdp$.

\begin{remark}
The preceding computation should not be interpreted as a rigorous asymptotic limit from the quantum Landau equation. Rather, it identifies a natural Fokker--Planck-type reduction of the Maxwellian-molecule quantum Landau operator under a radial symmetry ansatz around the quantum-weighted bulk velocity. The resulting operator is distinguished by two structural features:
the diffusion temperature is selected from the quantum-weighted second moment but normalized by the classical density, and the drift contains the quantum statistical factor $f(1+\hbar\kappa f)$. These choices are not arbitrary; they are precisely what makes the reduced operator compatible with the conservation of mass, momentum, and energy, while retaining the Bose--Einstein/Fermi--Dirac correction inherited from the quantum Landau equation. In this sense, the derivation provides a structural justification for the nonlinear quantum Fokker--Planck equation studied in this paper.
\end{remark}

%
%
%
%
%
%
%
%
%
%

\subsection{Pauli exclusion principle}\label{ssec:pauli}

We now discuss the fermionic admissible range. When $\kappa=-1$, the collision operator becomes
\[
Q(f)= \rho\nabla_p\cdot \lt\{ \Theta\nabla_p f+(p-u)f(1-\hbar f) \rt\}.
\]
In this case, the physically admissible interval is
\[
0\le f\le \frac1{\hbar}.
\]
This interval is compatible with the structure of the equation. Indeed, the nonlinear factor $f(1-\hbar f)$ vanishes at both endpoints $f=0$ and $f=\frac1\hbar$, so the nonlinear drift degenerates on the boundary of the admissible interval.

To see the barrier structure more explicitly, suppose formally that $f$ is a smooth solution and that the macroscopic quantities are regular with $\rho \Theta>0$. Since $\rho$, $u$, and $\Theta$ depend only on $(t,x)$, the collision operator can be written in the non-divergence form
\[
Q(f) = \rho \Theta\Delta_p f + \rho(p-u)(1-2\hbar f)\cdot\nabla_p f + 3\rho f(1-\hbar f).
\]
Assume that the solution leaves the admissible interval for the first time at $t=t_0$. If the first exit occurs through the lower barrier, then there exists a contact point $(x_0,p_0)$ such that $f(t_0,x_0,p_0)=0$. At this point, $f(t_0,\cdot,\cdot)$ attains its minimum in the $(x,p)$ variables. Hence, we have
\[
\nabla_x f(t_0,x_0,p_0)=0,
\quad
\nabla_p f(t_0,x_0,p_0)=0,
\quad \text{and} \quad
\Delta_p f(t_0,x_0,p_0)\ge0.
\]
Evaluating the full kinetic equation at $(t_0,x_0,p_0)$, we obtain
\[
\pa_t f(t_0,x_0,p_0) = \rho \Theta\Delta_p f(t_0,x_0,p_0) \ge 0,
\]
which formally rules out crossing the lower barrier. Similarly, if the first exit occurs through the upper barrier, then $f(t_0,x_0,p_0)=\frac1\hbar$ at a maximum point in the $(x,p)$ variables. This implies
\[
\nabla_x f(t_0,x_0,p_0)=0,
\quad
\nabla_p f(t_0,x_0,p_0)=0,
\quad \text{and} \quad
\Delta_p f(t_0,x_0,p_0)\le0,
\]
and hence
\[
\pa_t f(t_0,x_0,p_0) = \rho \Theta\Delta_p f(t_0,x_0,p_0) \le 0.
\]
This formal barrier argument suggests that the interval
\[
0 \le f \le\frac1\hbar
\]
is invariant under the evolution. A rigorous preservation result for the perturbative solutions constructed in this paper is proved later in Lemma \ref{lem:local-pos-pauli} by an $L^2$ truncation argument.

\begin{remark}
In the perturbative theory around the global equilibrium $\mcF_\hbar $, the nonnegativity of the distribution is not inferred solely from the smallness of the perturbation. Instead, assuming that the initial distribution satisfies $f_0\ge0$, we prove that this property is propagated by the evolution; see Lemma \ref{lem:local-pos-pauli}. In the fermionic case, the upper Pauli bound is propagated in the same way, provided that $f_0\le \frac1\hbar$. The perturbative smallness condition ensures that the macroscopic coefficients remain regular and that $\rho \Theta$ stays uniformly positive.
\end{remark}
%
%
%
%
%
%
%
%
%
%

\subsection{Conservation laws}

We next investigate the conservation laws of \eqref{eq:nqfp-s2}. For simplicity, we assume that $f$ decays sufficiently fast in $p$ and $x$ so that all integrations by parts below are justified. The rigorous identities can be obtained by a standard approximation argument for sufficiently regular solutions.

Let
\[
Q(f):= \rho\nabla_p\cdot \lt\{ \Theta\nabla_p f+(p-u)f(1+\hbar\kappa f) \rt\}.
\]
Then \eqref{eq:nqfp-s2} becomes
\[
\pa_t f+p\cdot\nabla_x f=Q(f).
\]

First, integrating in $(x,p)$ gives the conservation of mass:
\[
    \frac {\rd}{\rdt}\inttr f\,\rdx\rdp =0.
\]
Indeed, both the transport term and the collision term vanish after integration by parts.

For the momentum, we compute
\[
\frac {\rd}{\rdt}\inttr p f\,\rdx\rdp  = \inttr p Q(f)\,\rdx\rdp  = -\intr \rho \intr \lt\{ \Theta\nabla_p f+(p-u)f(1+\hbar\kappa f) \rt\} \rdp\,\rdx.
\]
Since $\intr \nabla_p f\,\rdp=0$ and
\[
\intr (p-u)f(1+\hbar\kappa f)\,\rdp = \intr p f(1+\hbar\kappa f)\,\rdp - u\intr f(1+\hbar\kappa f)\,\rdp =0
\]
by the definition of $u$, we obtain
\[
    \frac {\rd}{\rdt}\inttr p f\,\rdx\rdp =0.
\]

For the energy, we find
\[
\frac {\rd}{\rdt}\inttr |p|^2 f\,\rdx\rdp  = \inttr |p|^2Q(f)\,\rdx\rdp   = -2\intr \rho  \intr p\cdot \lt\{ \Theta\nabla_p f+(p-u)f(1+\hbar\kappa f) \rt\} \rdp\,\rdx.
\]
Using
\[
\intr p\cdot\nabla_p f\,\rdp=-3\intr f\,\rdp=-3\rho
\]
and the definition of $\Theta$,
\[
\intr (|p|^2-u\cdot p)f(1+\hbar\kappa f)\,\rdp = 3\rho \Theta,
\]
we get
\[
\intr |p|^2Q(f)\,\rdp = -2\rho \lt\{ -3\rho \Theta + \intr (|p|^2-u\cdot p)f(1+\hbar\kappa f)\,\rdp \rt\} = 0.
\]
Hence, we have
\[
    \frac {\rd}{\rdt}\inttr |p|^2 f\,\rdx\rdp =0.
\]
Combining the above observations concludes
\[
\frac{\rd}{\rdt}\inttr f\,\rdx\rdp =0,
\quad
\frac{\rd}{\rdt}\inttr pf\,\rdx\rdp =0,
\quad
\frac{\rd}{\rdt}\inttr |p|^2f\,\rdx\rdp =0.
\]

%
%
%
%
%
%
%
%
%
%
 
\subsection{Local quantum equilibria and entropy dissipation}\label{ssec:lqe}

We now identify the local quantum equilibria and derive the entropy dissipation structure. For $z>0$ satisfying $1+\hbar\kappa z>0$, we define
\[
\Xi_\kappa(z):=\ln \frac{z}{1+\hbar\kappa z}.
\]
Since
\[
\nabla_p\Xi_\kappa(f)=\frac{\nabla_p f}{f(1+\hbar\kappa f)},
\]
the collision flux can be written as
\[
\Theta\nabla_p f+(p-u)f(1+\hbar\kappa f)=f(1+\hbar\kappa f)\lt\{\Theta\nabla_p\Xi_\kappa(f)+(p-u)\rt\}.
\]

We first characterize the distributions for which the collision flux vanishes. The condition
\bq\label{eq:zero-fe}
\Theta\nabla_p f+(p-u)f(1+\hbar\kappa f)=0
\eq
is equivalent to
\[
\nabla_p\Xi_\kappa(f)=-\frac{p-u}{\Theta}.
\]
Integrating in $p$, we obtain
\[
\ln\frac{f}{1+\hbar\kappa f}=-\frac{|p-u|^2}{2\Theta}-\theta
\]
for some scalar function $\theta=\theta(t,x)$. Solving this relation for $f$ gives the local quantum equilibrium
\bq\label{eq:loc-qua-e}
\mcF(\theta,u,\Theta;p)=\lt(e^{\frac{|p-u(t,x)|^2}{2\Theta(t,x)}+\theta(t,x)}-\hbar\kappa\rt)^{-1}.
\eq
Conversely, a direct computation shows that \eqref{eq:loc-qua-e} satisfies \eqref{eq:zero-fe}. Thus, within the class of positive distributions satisfying $1+\hbar\kappa f>0$, the vanishing of the collision flux is equivalent to $f$ being a local quantum equilibrium of the form \eqref{eq:loc-qua-e}.

Let $\mcF=\mcF(\theta,u,\Theta;p)$ be a local quantum equilibrium with the same velocity $u$ and temperature $\Theta$. Since
\[
\nabla_p\Xi_\kappa(\mcF)=-\frac{p-u}{\Theta},
\]
we have
\[
\Theta\nabla_p f+(p-u)f(1+\hbar\kappa f)=\Theta f(1+\hbar\kappa f)\nabla_p\lt\{\Xi_\kappa(f)-\Xi_\kappa(\mcF)\rt\}.
\]
Equivalently,
\bq\label{eq:Q-log-form}
Q(f)=\rho\nabla_p\cdot\lt\{\Theta f(1+\hbar\kappa f)\nabla_p\ln\lt(\frac{f(1+\hbar\kappa\mcF)}{\mcF(1+\hbar\kappa f)}\rt)\rt\}.
\eq
In the classical case $\hbar=0$, this reduces to the familiar Fokker--Planck form
\[
Q(f)=\rho\nabla_p\cdot\lt\{\Theta f\nabla_p\ln\frac{f}{\mcF}\rt\}.
\]

We define the quantum entropy functional by
\bq\label{eq:ent-fl}
\mathscr{H}[f]=\inttr\lt\{f\ln f-\frac{\kappa}{\hbar}(1+\hbar\kappa f)\ln(1+\hbar\kappa f)\rt\}\,\rdx\rdp .
\eq
In the classical case $\hbar=0$, this should be understood in the limiting sense and reduces, up to an irrelevant affine term, to the classical entropy $\inttr f\ln f\,\rdx\rdp $.

We formally derive the entropy identity for smooth solutions satisfying the admissibility condition $f\ge0$, $1+\hbar\kappa f\ge0$. In the fermionic case, this condition is precisely the Pauli admissible range discussed in Section \ref{ssec:pauli}. For the computation below, one may first assume the strict inequalities $f>0$, $1+\hbar\kappa f>0$, and then pass to the general case by a standard approximation argument.

Differentiating \eqref{eq:ent-fl} along smooth solutions of \eqref{eq:nqfp-s2}, the transport term vanishes after integration by parts, and we obtain
\[
\frac{\rd}{\rdt}\mathscr{H}[f]=\inttr Q(f)\Xi_\kappa(f)\,\rdx\rdp .
\]
On the other hand,
\[
\Xi_\kappa(\mcF)=-\frac{|p-u(t,x)|^2}{2\Theta(t,x)}-\theta(t,x),
\]
which is a linear combination of $1$, $p$, and $|p|^2$. By the conservation properties of the collision operator,
\[
\intr Q(f)\Xi_\kappa(\mcF)\,\rdp=0.
\]
Hence, we have
\[
\frac{\rd}{\rdt}\mathscr{H}[f]=\inttr Q(f)\lt\{\Xi_\kappa(f)-\Xi_\kappa(\mcF)\rt\}\,\rdx\rdp .
\]
Using \eqref{eq:Q-log-form} and integrating by parts in $p$ yield
\[
\frac{\rd}{\rdt}\mathscr{H}[f]+\mathscr{D}[f]=0,
\]
where
\bq\label{eq:ent-dis0}
\mathscr{D}[f]:=\inttr \rho(t,x)\Theta(t,x)f(1+\hbar\kappa f)\lt|\nabla_p\ln\lt(\frac{f(1+\hbar\kappa\mcF)}{\mcF(1+\hbar\kappa f)}\rt)\rt|^2 \rdx\rdp \ge0.
\eq
Consequently,
\[
\frac{\rd}{\rdt}\mathscr{H}[f]\le0.
\]
If $\rho \Theta>0$, equality holds if and only if the collision flux vanishes, which is equivalent to $f$ being a local quantum equilibrium of the form \eqref{eq:loc-qua-e}.

%
%
%
%
%
%
%
%
%
%
\subsection{Summary of the structural properties}

The nonlinear quantum Fokker--Planck equation \eqref{eq:nqfp-s2} therefore satisfies the following formal structural properties:
\begin{enumerate}
\item \emph{Pauli exclusion principle.} For smooth solutions, if $\kappa=-1$ and $0\le f_0\le\frac1\hbar$, then the admissible interval is preserved:
\[
0\le f(t,x,p)\le\frac1{\hbar}.
\]

\item \emph{Conservation laws.} Smooth decaying solutions satisfy
\[
\frac{\rd}{\rdt}\inttr f\,\rdx\rdp =0,
\quad
\frac{\rd}{\rdt}\inttr pf\,\rdx\rdp =0,
\quad
\frac{\rd}{\rdt}\inttr |p|^2f\,\rdx\rdp =0.
\]

 \item \emph{Collision equilibria.} Within the class of positive admissible distributions with $\rho \Theta>0$, the vanishing of the collision flux, equivalently the vanishing of the entropy dissipation, characterizes the local quantum equilibria
\[
\mcF(\theta,u,\Theta;p) = \lt( e^{\frac{|p-u(t,x)|^2}{2\Theta(t,x)}+\theta(t,x)} - \hbar\kappa \rt)^{-1}.
\]

\item \emph{Entropy inequality.} The entropy functional \eqref{eq:ent-fl} satisfies
    \[
\frac{\rd}{\rdt}\mathscr{H}[f]+\mathscr{D}[f]=0,
    \quad
   \mathscr{D}[f]\ge0,
    \]
    where $\mathscr{D}[f]$ is given by \eqref{eq:ent-dis0}. In particular,
    \[
\frac{\rd}{\rdt}\mathscr{H}[f]\le0,
    \]
    and equality holds precisely when the collision flux vanishes, equivalently when $f$ is a local quantum equilibrium.
\end{enumerate}

%
%
%
%
%
%
%
%
%
%

\section{Perturbative formulation and structure of the linearized operator}
\label{sec:lin-str}

In this section, we derive the perturbation equation around the global quantum equilibrium and analyze the basic structure of the linearized collision operator. The goal is to identify the finite-dimensional macroscopic modes, the dissipative microscopic part, and the nonlinear terms generated by the self-consistent dependence of $\rho$, $u$, and $\Theta$ on the distribution function.

%
%
%
%
%
%
%
%
%
%
\subsection{Preliminaries and macroscopic expansions}\label{ssec:pre_mac}

We recall the global equilibrium introduced in \eqref{F_0}:
\[
\mcF_\hbar (p) = \lt(e^{\frac{|p|^2}{2}+\theta_0}-\hbar\kappa\rt)^{-1}.
\]
We impose the admissibility condition \eqref{eq:ht-cond}. For fixed equilibrium parameter $\theta_0$, this can be achieved by taking $\hbar>0$ sufficiently small. In the bosonic case $\kappa=1$, this condition prevents the denominator of $\mcF_\hbar $ from vanishing. In the fermionic case $\kappa=-1$, it ensures that $\eta_\hbar =1-2\hbar\mcF_\hbar $ is uniformly positive.

We define
\[
\mu_\hbar :=\mcF_\hbar +\hbar\kappa\mcF_\hbar ^2,
\quad
\eta_\hbar :=1+2\hbar\kappa\mcF_\hbar .
\]
Then $\eta_\hbar ^2=1+4\hbar\kappa\mu_\hbar $. A direct computation gives
\bq\label{eq:elem}
\nabla_p\mcF_\hbar =-p\mu_\hbar ,
\quad
\nabla_p\sqrt{\mu_\hbar }=-\frac12p\eta_\hbar \sqrt{\mu_\hbar },
\eq
and
\bq\label{eq:mom-id-mu0}
\intr |p|^2\mu_\hbar \,\rdp=3\intr \mcF_\hbar \,\rdp,
\quad
\intr |p|^2\eta_\hbar \mu_\hbar \,\rdp=3\intr \mu_\hbar \,\rdp.
\eq
These identities will be used repeatedly in the linearization and in the analysis of the linearized operator.

We perturb the solution around $\mcF_\hbar $ by writing
\[
    f=\mcF_\hbar +\sqrt{\mu_\hbar }\,g.
\]
Then
\bq\label{eq:F-exp}
    f(1+\hbar\kappa f) = \mu_\hbar +\eta_\hbar g\sqrt{\mu_\hbar } +\hbar\kappa\mu_\hbar  g^2.
\eq
For notational convenience, we set
\[
    M:=\intr \mcF_\hbar \,\rdp.
\]
Since $\mcF_\hbar >0$, we have $M>0$.

We use the following notation for function spaces and norms.   When there is no ambiguity, we write $L^2:=L^2(\R^3\times\R^3)$. We use the subscripts $x$ and $p$ only when the variable of integration needs to be specified. For an integer $s\ge0$, $H^s$ denotes the usual Sobolev space on $\mR^3\times\mR^3$. We use multi-indices $\alpha=(\alpha_1,\alpha_2,\alpha_3)$ and
$\beta=(\beta_1,\beta_2,\beta_3)$, and write
\[
\pa_x^\alpha\pa_p^\beta = \pa_{x_1}^{\alpha_1}\pa_{x_2}^{\alpha_2} \pa_{x_3}^{\alpha_3} \pa_{p_1}^{\beta_1}\pa_{p_2}^{\beta_2} \pa_{p_3}^{\beta_3}.
\]
The $L^2_p$ inner product is denoted by
\[
\lal \varphi,\psi\ral_{L^2_p} := \intr \varphi(p)\psi(p)\,\rdp.
\]

We also define the weighted dissipation norm
\[
    |g|_D^2 :=  \|\nabla_p g\|_{L^2_p}^2 + \|p\eta_\hbar  g\|_{L^2_p}^2,
\]
and, for functions depending on both $x$ and $p$,
\[
 \|g\|_D^2 := \|\nabla_p g\|_{L^2_{x,p}}^2 + \|p\eta_\hbar  g\|_{L^2_{x,p}}^2.
\]
Under \eqref{eq:ht-cond}, $\eta_\hbar $ is bounded from below. Indeed, if $\kappa=1$, then $\eta_\hbar =1+2\hbar\mcF_\hbar >1$. If $\kappa=-1$, then
\[
\eta_\hbar  = 1-2\hbar\mcF_\hbar  = \frac{e^{\frac{|p|^2}{2}+\theta_0}-\hbar}{e^{\frac{|p|^2}{2}+\theta_0}+\hbar} \ge \frac{1-\hbar e^{-\theta_0}}{1+\hbar e^{-\theta_0}}>0.
\]
For $\kappa=0$, we simply have $\eta_\hbar =1$. Thus there exists $c_0>0$ such that
\[
    \eta_\hbar (p)\ge c_0>0
    \quad\text{for all }p\in\mR^3.
\]
Consequently, the $D$-norm controls the $L^2_p$ norm. Indeed,
\[
|g|_D^2 = \|\nabla_p g\|_{L^2_p}^2+\|p\eta_\hbar g\|_{L^2_p}^2 \ge 2\intr |p\eta_\hbar  g|\,|\nabla_p g|\,\rdp \ge -2c_0\intr pg\cdot\nabla_p g\,\rdp = 3c_0\|g\|_{L^2_p}^2,
\]
where the last identity follows by integration by parts. Thus $|\cdot|_D$ is a genuine norm which controls the $L^2_p$ norm.

We finally introduce a compact notation for the $p$-dependent coefficients generated by differentiating $\mcF_\hbar $, $\mu_\hbar $, and $\eta_\hbar $. For a function $\varphi(p)$ belonging to the class
\[
\lt\{\eta_\hbar ^i\mu_\hbar ^j,\, p\eta_\hbar ^i\mu_\hbar ^j,\, |p|^2\eta_\hbar ^i\mu_\hbar ^j,\, |p|^2p\eta_\hbar ^i\mu_\hbar ^j\rt\}_{i\in\mN,\,j\in\mQ^{++}}, \quad \Q^{++}:= \{ r \in \Q : r > 0\}
\]
we denote by $\mcP_{\varphi,|\beta|}$ a finite linear combination of terms obtained by differentiating $\varphi$ up to order $|\beta|$ in $p$. More precisely, $\mcP_{\varphi,|\beta|}$ denotes a finite sum of terms of the form
\[
C_{\ell}\,p^{\ell}\eta_\hbar ^{i_\ell}\mu_\hbar ^{j_\ell},
\]
where $C_\ell$ is a constant, $\ell$ is a multi-index, $i_\ell\in\mN$, and $j_\ell\in\mQ^{++}$. The powers and coefficients may depend on $\varphi$ and $|\beta|$, but their number is finite.

We claim that these coefficients are uniformly bounded and have exponential decay in $p$. Indeed,
\[
\mu_\hbar  = \mcF_\hbar +\hbar\kappa\mcF_\hbar ^2 = \frac{e^{\frac{|p|^2}{2}+\theta_0}}{\lt(e^{\frac{|p|^2}{2}+\theta_0}-\hbar\kappa\rt)^2} = \frac{e^{-\lt(\frac{|p|^2}{2}+\theta_0\rt)}}{ \lt(1-\hbar\kappa e^{-\lt(\frac{|p|^2}{2}+\theta_0\rt)}\rt)^2}.
\]
By \eqref{eq:ht-cond}, there exist constants $0<c_{\hbar,\kappa,\theta_0}<C_{\hbar,\kappa,\theta_0}<\infty$ such that
\bq\label{eq:mu0-comp}
c_{\hbar,\kappa,\theta_0} e^{-\frac{|p|^2}{2}} \le \mu_\hbar (p) \le C_{\hbar,\kappa,\theta_0} e^{-\frac{|p|^2}{2} }.
\eq
Together with the boundedness of $\eta_\hbar $ and the identities \eqref{eq:elem}, this shows that each $p$-derivative of $\mcF_\hbar $, $\mu_\hbar $, and $\eta_\hbar $ is a finite sum of products of polynomials in $p$ and powers of $\mu_\hbar $ and $\eta_\hbar $. Consequently, each term appearing in $\mcP_{\varphi,|\beta|}$ is bounded by a polynomial times an exponentially decaying function, except for harmless bounded powers of $\eta_\hbar $. Hence
\bq\label{eq:P-coeff-b}
    \|\mcP_{\varphi,|\beta|}\|_{L^2_p}
    +
    \|\mcP_{\varphi,|\beta|}\|_{L^\infty_p}
    \le C_\beta.
\eq
This observation will be used repeatedly below without further comment when $p$-derivatives fall on the coefficients $\mcF_\hbar $, $\mu_\hbar $, $\eta_\hbar $, or on the basis functions of the macroscopic space.

We next expand the macroscopic fields $\rho$, $u$, and $\rho \Theta$ in terms of the perturbation $g$.

\begin{lemma}\label{lem:decomp}
Let $f=\mcF_\hbar +\sqrt{\mu_\hbar }g$. Then the macroscopic fields satisfy
\bq\label{def of u}
\rho = M+\intr g\sqrt{\mu_\hbar }\,\rdp,
\quad 
 u = \frac{1}{\intr\mu_\hbar \,\rdp}\intr p\eta_\hbar g\sqrt{\mu_\hbar }\,\rdp + N_u,
 \quad
\rho \Theta = M + \frac13\intr |p|^2\eta_\hbar g\sqrt{\mu_\hbar }\,\rdp + N_\Theta.
\eq
Here $N_u$ and $N_\Theta$ are nonlinear with respect to $g$ and are given by
\bq\label{eq:Nu-def}
N_u := - \frac{\intr p\eta_\hbar g\sqrt{\mu_\hbar }\,\rdp}{\intr\mu_\hbar \,\rdp}R_1(g) + \frac{\hbar\kappa\intr p\mu_\hbar  g^2\,\rdp}{\intr\mu_\hbar \,\rdp}\lt(1-R_1(g)\rt),
\eq
and
\bq\label{eq:NT-def}
N_\Theta := \frac{\hbar\kappa}{3}\intr |p|^2\mu_\hbar  g^2\,\rdp - \frac13 \lt( \frac{1}{\intr\mu_\hbar \,\rdp}\intr p\eta_\hbar g\sqrt{\mu_\hbar }\,\rdp + N_u \rt) \cdot R_2(g),
\eq
where
\begin{align}\label{eq:R1R2-def}
\begin{aligned}
R_1(g) &:= \frac{A(g)}{\intr\mu_\hbar \,\rdp} - \frac{A(g)^2}{\intr\mu_\hbar \,\rdp\intr\{\mu_\hbar +\eta_\hbar g\sqrt{\mu_\hbar }+\hbar\kappa\mu_\hbar g^2\}\,\rdp},\cr
R_2(g) &:= \intr p\eta_\hbar g\sqrt{\mu_\hbar }\,\rdp + \hbar\kappa\intr p\mu_\hbar g^2\,\rdp,
\end{aligned}
\end{align}
with
\bq\label{eq:A-def}
A(g) := \intr \eta_\hbar g\sqrt{\mu_\hbar }\,\rdp + \hbar\kappa\intr \mu_\hbar g^2\,\rdp.
\eq
\end{lemma}

\begin{proof}
Substituting $f=\mcF_\hbar +\sqrt{\mu_\hbar }g$, we first obtain
\bq\label{eq:F-exp-macro}
f(1+\hbar\kappa f) = \mu_\hbar +\eta_\hbar g\sqrt{\mu_\hbar } +\hbar\kappa\mu_\hbar g^2.
\eq
The density expansion follows immediately:
\[
\rho=\intr f\,\rdp = \intr\mcF_\hbar \,\rdp+\intr g\sqrt{\mu_\hbar }\,\rdp = M+\intr g\sqrt{\mu_\hbar }\,\rdp.
\]

We next compute $u$. By \eqref{eq:F-exp-macro} and oddness,
\[
\intr p f(1+\hbar\kappa f)\,\rdp = \intr p\eta_\hbar g\sqrt{\mu_\hbar }\,\rdp + \hbar\kappa\intr p\mu_\hbar g^2\,\rdp,
\]
where we used $\intr p\mu_\hbar \,\rdp=0$. Moreover,
\[
\intr f(1+\hbar\kappa f)\,\rdp = \intr\mu_\hbar \,\rdp+A(g),
\]
where $A(g)$ is defined in \eqref{eq:A-def}. Thus,
\[
u = \frac{\intr p\eta_\hbar g\sqrt{\mu_\hbar }\,\rdp + \hbar\kappa\intr p\mu_\hbar g^2\,\rdp}{\intr\mu_\hbar \,\rdp+A(g)}.
\]
Using the algebraic identity
\[
\frac1{c+x} = \frac1c \lt\{1-\frac{x}{c} + \frac{x^2}{c(c+x)}\rt\},
\]
with $c=\intr \mu_\hbar \,\rdp$ and $x=A(g)$, we get
\[
\frac1{\intr\mu_\hbar \,\rdp+A(g)} = \frac1{\intr\mu_\hbar \,\rdp} \lt(1-R_1(g)\rt),
\]
where $R_1(g)$ is defined in \eqref{eq:R1R2-def}. Hence,
\[
u = \frac{1}{\intr\mu_\hbar \,\rdp} \intr p\eta_\hbar g\sqrt{\mu_\hbar }\,\rdp - \frac{\intr p\eta_\hbar g\sqrt{\mu_\hbar }\,\rdp}{\intr\mu_\hbar \,\rdp}R_1(g) + \frac{\hbar\kappa\intr p\mu_\hbar g^2\,\rdp}{\intr\mu_\hbar \,\rdp} \lt(1-R_1(g)\rt),
\]
which gives \eqref{def of u} and \eqref{eq:Nu-def}.

It remains to expand $\rho \Theta$. By definition,
\[
\rho \Theta = \frac13\intr (|p|^2-u\cdot p)f(1+\hbar\kappa f)\,\rdp.
\]
Using \eqref{eq:F-exp-macro}, we have
\[
\rho \Theta = \frac13 \intr |p|^2 \lt\{ \mu_\hbar +\eta_\hbar g\sqrt{\mu_\hbar } +\hbar\kappa\mu_\hbar g^2 \rt\} \rdp - \frac13 u\cdot \intr p \lt\{ \mu_\hbar +\eta_\hbar g\sqrt{\mu_\hbar }+\hbar\kappa\mu_\hbar g^2 \rt\} \rdp.
\]
By oddness, $\intr p\mu_\hbar \,\rdp=0$. Also, by \eqref{eq:mom-id-mu0}, we find
\[
\frac13\intr |p|^2\mu_\hbar \,\rdp=M.
\]
Hence, we have
\[
\rho \Theta = M + \frac13\intr |p|^2\eta_\hbar g\sqrt{\mu_\hbar }\,\rdp + \frac{\hbar\kappa}{3}\intr |p|^2\mu_\hbar g^2\,\rdp - \frac13  u\cdot \lt( \intr p\eta_\hbar g\sqrt{\mu_\hbar }\,\rdp + \hbar\kappa\intr p\mu_\hbar g^2\,\rdp \rt).
\]
Since
\[
u= \frac{1}{\intr\mu_\hbar \,\rdp} \intr p\eta_\hbar g\sqrt{\mu_\hbar }\,\rdp + N_u
\]
and the last parenthesis is precisely $R_2(g)$, we obtain
\[
\rho \Theta = M + \frac13\intr |p|^2\eta_\hbar g\sqrt{\mu_\hbar }\,\rdp + N_\Theta,
\]
with $N_\Theta$ given by \eqref{eq:NT-def}. This completes the proof.
\end{proof}

%
%
%
%
%
%
%
%
%
%

\subsection{The linearized equation}

We now derive the equation satisfied by the perturbation $g$. Recall that $f=\mcF_\hbar +\sqrt{\mu_\hbar }g$. For notational simplicity, we use the following abbreviations:
\[
\mathscr a(g):=\intr g\sqrt{\mu_\hbar }\,\rdp,
\quad
\mathscr b(g):=\intr p\eta_\hbar g\sqrt{\mu_\hbar }\,\rdp,
\quad
\mathscr c(g):=\intr |p|^2\eta_\hbar g\sqrt{\mu_\hbar }\,\rdp.
\]

We first define the linearized collision operator. Let
\begin{align}\label{eq:L-def}
\begin{aligned}
    L(g) &:= \frac{M}{\sqrt{\mu_\hbar }} \nabla_p\cdot \lt\{ \lt(\nabla_p g+\frac12p\eta_\hbar g\rt)\sqrt{\mu_\hbar } - \frac{p\mu_\hbar }{\intr |p|^2\mu_\hbar \,\rdp} \intr (|p|^2\eta_\hbar -3)g\sqrt{\mu_\hbar }\,\rdp  \rt.  \cr
    &\hspace{6cm}\lt. -    \frac{\mu_\hbar }{\intr \mu_\hbar \,\rdp}   \intr p\eta_\hbar g\sqrt{\mu_\hbar }\,\rdp    \rt\}.
\end{aligned}
\end{align}
Equivalently, we may write
\[
    L=L_0+P_1,
\]
where
\[
L_0(g) := \frac{M}{\sqrt{\mu_\hbar }} \nabla_p\cdot \lt\{ \mu_\hbar \nabla_p\lt(\frac{g}{\sqrt{\mu_\hbar }}\rt) \rt\} = M\lt[ \Delta_p g + g\lt\{ \frac32\eta_\hbar   - |p|^2\lt(\frac14\eta_\hbar ^2+\hbar\kappa\mu_\hbar \rt) \rt\}\rt],
\]
and $P_1$ is the finite-dimensional operator
\[
P_1(g) := \sum_{i=1}^4\lal g,v_i\ral_{L^2_p}v_i,
\]
with
\[
v_i := \sqrt{M}\frac{p_i\eta_\hbar \sqrt{\mu_\hbar }}{\lt(\intr \mu_\hbar \,\rdp\rt)^{\frac12}} \quad (i=1,2,3),
\quad
v_4 := \sqrt{M}\frac{(|p|^2\eta_\hbar -3)\sqrt{\mu_\hbar }} {\lt(\intr |p|^2\mu_\hbar \,\rdp\rt)^{\frac12}}.
\]

The nonlinear remainder will be written in flux form. Define
\begin{align}\label{eq:non-flux}
\begin{aligned}
\calR(g)  &:= \frac{\mathscr c(g)}{3}\nabla_p(g\sqrt{\mu_\hbar }) + N_\Theta\lt\{-p\mu_\hbar +\nabla_p(g\sqrt{\mu_\hbar })\rt\} + \hbar\kappa M p\mu_\hbar g^2 \cr
    &\quad
    -  M\frac{\mathscr b(g)}{\intr\mu_\hbar \,\rdp} \lt\{ \eta_\hbar g\sqrt{\mu_\hbar } + \hbar\kappa\mu_\hbar g^2 \rt\} -    M N_u \lt\{ \mu_\hbar +\eta_\hbar g\sqrt{\mu_\hbar }+\hbar\kappa\mu_\hbar g^2 \rt\} \cr
    &\quad
    + \mathscr a(g)p \lt\{ \eta_\hbar g\sqrt{\mu_\hbar } + \hbar\kappa\mu_\hbar g^2\rt\} -\mathscr a(g)  \lt( \frac{\mathscr b(g)}{\intr\mu_\hbar \,\rdp} + N_u \rt) \lt\{ \mu_\hbar +\eta_\hbar g\sqrt{\mu_\hbar }+\hbar\kappa\mu_\hbar g^2\rt\}.
\end{aligned}
\end{align}
Here $N_u$ and $N_\Theta$ are the nonlinear terms defined in Lemma \ref{lem:decomp}. We then define
\bq\label{eq:Gam-def}
    \Gamma(g):=\frac{1}{\sqrt{\mu_\hbar }}\nabla_p\cdot\calR(g) .
\eq

\begin{proposition} 
Let $f=\mcF_\hbar +\sqrt{\mu_\hbar }g$. Then the nonlinear quantum Fokker--Planck equation \eqref{eq:main} can be written in the perturbative form
\bq\label{eq:linearized}
\pa_t g+p\cdot\nabla_x g = L(g)+\Gamma(g),
\eq
where $L$ is given by \eqref{eq:L-def} and $\Gamma$ is given by \eqref{eq:Gam-def}.
\end{proposition}

\begin{proof}
Substituting $f=\mcF_\hbar +\sqrt{\mu_\hbar }g$ into the left-hand side of \eqref{eq:main}, we get
\bq\label{eq:lhs}
    \pa_t f+p\cdot\nabla_x f    =    \sqrt{\mu_\hbar }\lt(\pa_t g+p\cdot\nabla_x g\rt).
\eq

By Lemma \ref{lem:decomp}, we have
\[
\rho \Theta=M+\frac13\mathscr c(g)+N_\Theta,
\]
and hence
\begin{align}\label{eq:1st-term}
\begin{aligned}
\rho \Theta\nabla_p f &= \lt(M+\frac13\mathscr c(g)+N_\Theta\rt) \lt\{-p\mu_\hbar +\nabla_p(g\sqrt{\mu_\hbar })\rt\} \cr
&=M\lt\{ -p\mu_\hbar +\nabla_p(g\sqrt{\mu_\hbar }) -\frac{1}{3M}p\mu_\hbar \mathscr c(g) \rt\} + \frac13\mathscr c(g)\nabla_p(g\sqrt{\mu_\hbar }) + N_\Theta\lt\{-p\mu_\hbar +\nabla_p(g\sqrt{\mu_\hbar })\rt\}.
\end{aligned}
\end{align}
Similarly, using
\[
\rho=M+ \mathscr a(g),
\quad
u=\frac{\mathscr b(g)}{\intr\mu_\hbar \,\rdp}+N_u,
\]
and \eqref{eq:F-exp}, we obtain
\begin{align}\label{eq:2nd-term}
\begin{aligned}
\rho(p-u)f(1+\hbar\kappa f) &= \lt(M+ \mathscr a(g)\rt) \lt( p-\frac{\mathscr b(g)}{\intr\mu_\hbar \,\rdp}-N_u \rt) \lt\{ \mu_\hbar +\eta_\hbar g\sqrt{\mu_\hbar }+\hbar\kappa\mu_\hbar g^2\rt\} \cr
&= M\lt\{ p\mu_\hbar  + p\eta_\hbar g\sqrt{\mu_\hbar } - \frac{\mu_\hbar }{\intr\mu_\hbar \,\rdp}\mathscr b(g) + \frac{1}{M} \mathscr a(g)p\mu_\hbar  \rt\} \cr
&\quad + \hbar\kappa M p\mu_\hbar g^2 - M\frac{\mathscr b(g)}{\intr\mu_\hbar \,\rdp} \lt\{ \eta_\hbar g\sqrt{\mu_\hbar } + \hbar\kappa\mu_\hbar g^2 \rt\} \cr
&\quad - M N_u \lt\{ \mu_\hbar +\eta_\hbar g\sqrt{\mu_\hbar }+\hbar\kappa\mu_\hbar g^2 \rt\}  + \mathscr a(g)p \lt\{ \eta_\hbar g\sqrt{\mu_\hbar } + \hbar\kappa\mu_\hbar g^2 \rt\} \cr
&\quad - \mathscr a(g) \lt( \frac{\mathscr b(g)}{\intr\mu_\hbar \,\rdp} + N_u \rt) \lt\{ \mu_\hbar +\eta_\hbar g\sqrt{\mu_\hbar }+\hbar\kappa\mu_\hbar g^2 \rt\}.
\end{aligned}
\end{align}
The first line on the right-hand side of \eqref{eq:2nd-term} contains the linear terms, while the remaining terms are nonlinear.

We now add \eqref{eq:1st-term} and \eqref{eq:2nd-term}. The zeroth-order terms $-Mp\mu_\hbar $ and $Mp\mu_\hbar $ cancel. Moreover,
\[
\nabla_p(g\sqrt{\mu_\hbar })+p\eta_\hbar g\sqrt{\mu_\hbar }
=
\lt(\nabla_p g+\frac12p\eta_\hbar g\rt)\sqrt{\mu_\hbar } = \mu_\hbar  \nabla_p \lt( \frac{g}{\sqrt{\mu_\hbar }} \rt).
\]
Using $\intr |p|^2\mu_\hbar \,\rdp=3M$, we also have
\[
 \mathscr a(g)p\mu_\hbar -\frac13p\mu_\hbar \mathscr c(g) = -M\frac{p\mu_\hbar }{\intr |p|^2\mu_\hbar \,\rdp} \intr (|p|^2\eta_\hbar -3)g\sqrt{\mu_\hbar }\,\rdp.
\]
Hence,
\begin{align*} 
&\rho \Theta\nabla_p f+\rho(p-u)f(1+\hbar\kappa f) \cr
&\quad = M\lt\{ \lt(\nabla_pg+\frac12p\eta_\hbar g\rt)\sqrt{\mu_\hbar } - \frac{p\mu_\hbar }{\intr |p|^2\mu_\hbar \,\rdp} \intr (|p|^2\eta_\hbar -3)g\sqrt{\mu_\hbar }\,\rdp \rt. \cr
&\hspace{5.8cm}\lt. - \frac{\mu_\hbar }{\intr\mu_\hbar \,\rdp} \intr p\eta_\hbar g\sqrt{\mu_\hbar }\,\rdp \rt\} + \calR(g) .
\end{align*}
Applying $\nabla_p\cdot$ to the above, dividing by $\sqrt{\mu_\hbar }$, and using \eqref{eq:lhs}, we obtain
\[
\pa_t g+p\cdot\nabla_xg = L(g)+\frac1{\sqrt{\mu_\hbar }}\nabla_p\cdot\calR(g) .
\]
By the definition \eqref{eq:Gam-def}, this is precisely \eqref{eq:linearized}. This completes the proof.
\end{proof}

%
%
%
%
%
%
%
%
%
%
\subsection{Macroscopic space and coercivity of the linearized operator}

In this subsection, we identify the null space of the linearized operator and establish its coercivity on the microscopic component. We first introduce the macroscopic space
\[
\mcN := \mathrm{span}\{\sqrt{\mu_\hbar },\,p_1\sqrt{\mu_\hbar },\,p_2\sqrt{\mu_\hbar },\, p_3\sqrt{\mu_\hbar },\,|p|^2\sqrt{\mu_\hbar }\}.
\]
Let
\[
\ph := \frac{\int_{\mR^3}|p|^2\mu_\hbar \,\rdp} {\int_{\mR^3}\mu_\hbar \,\rdp}.
\]
Notice that $\ph\to3$ as $\hbar\to0$. We choose the following orthonormal basis
of $\mcN$ in $L^2_p$:
\[
    e_1 := \frac{\sqrt{\mu_\hbar }} {\lt(\int_{\mR^3}\mu_\hbar \,\rdp\rt)^{\frac12}},
    \quad   
    e_{1+j}:= \frac{p_j\sqrt{\mu_\hbar }} {\lt(\frac13\int_{\mR^3}|p|^2\mu_\hbar \,\rdp\rt)^{\frac12}}  \ (j=1,2,3),
 \quad
    e_5 := \frac{(|p|^2-\ph)\sqrt{\mu_\hbar }} {\lt(\int_{\mR^3}(|p|^2-\ph)^2\mu_\hbar \,\rdp\rt)^{\frac12}}.
\]
By the radial symmetry of $\mu_\hbar $ and the oddness of the corresponding integrands, one checks directly that
\[
\lal e_i,e_j\ral_{L^2_p} = 
\begin{cases}
1,& i=j,\\
0,& i\neq j.
\end{cases}
\]
We define
\[
    P_0g:=\lal g,e_1\ral_{L^2_p}e_1,
    \quad
    Pg:=\sum_{i=1}^5\lal g,e_i\ral_{L^2_p}e_i.
\]
Thus $P$ is the orthogonal projection onto $\mcN$.

 We now collect the basic spectral and coercive properties of the linearized collision operator. The operator $L_0$ is the dissipative Fokker--Planck part, whereas $P_1$ is a finite-rank correction generated by the self-consistent dependence of the macroscopic fields $u$ and $\Theta$ on $f$. This correction enlarges the null space from the one-dimensional space $\mathrm{span}\{\sqrt{\mu_\hbar }\}$ of $L_0$ to the full macroscopic space $\mcN$. The following lemma identifies this null space and gives the microscopic coercivity estimate that will be used throughout the energy argument.

 \begin{lemma} \label{lem:Prop-L}
Let $f,g\in H^1(\mR^3_p)$. Then the linear operators $L_0$ and $L$ satisfy the following properties. 
\begin{enumerate}
    \item $L_0$ and $L$ are self-adjoint:
    \[
    \lal L_0f,g\ral_{L^2_p} = \lal f,L_0g\ral_{L^2_p},
    \quad
    \lal Lf,g\ral_{L^2_p} = \lal f,Lg\ral_{L^2_p}.
    \]
    \item The null spaces are given by
    \[
    \ker L_0=\mathrm{span}\{\sqrt{\mu_\hbar }\},
    \quad
    \ker L=\mcN.
    \]
    \item There exists a positive constant $\lambda_0>0$ such that
    \bq\label{eq:L-coercivity}
        \lal Lg,g\ral_{L^2_p}  \le -\lambda_0 |(I-P)g|_D^2.
    \eq
\end{enumerate}
In particular, $L$ is nonpositive in $L^2_p$.
\end{lemma}

\begin{proof}
We prove the assertions in several steps.

\medskip
\noindent
{\bf Step 1. Self-adjointness.} For $L_0$, integration by parts gives 
\begin{align}\label{eq:L0-sym-f}
\begin{aligned}
    \lal L_0f,g\ral_{L^2_p} &= \int_{\mR^3} \frac{M}{\sqrt{\mu_\hbar }} \nabla_p\cdot  \lt\{ \mu_\hbar \nabla_p\lt(\frac{f}{\sqrt{\mu_\hbar }}\rt) \rt\}  g\,\rdp \cr
    &= M\int_{\mR^3}  \nabla_p\cdot \lt\{ \mu_\hbar \nabla_p\lt(\frac{f}{\sqrt{\mu_\hbar }}\rt) \rt\} \frac{g}{\sqrt{\mu_\hbar }}\,\rdp \cr
    &= -M\int_{\mR^3} \mu_\hbar  \nabla_p\lt(\frac{f}{\sqrt{\mu_\hbar }}\rt) \cdot \nabla_p\lt(\frac{g}{\sqrt{\mu_\hbar }}\rt)\,\rdp.
\end{aligned}
\end{align}
This expression is symmetric in $f$ and $g$, and hence
\[
\lal L_0f,g\ral_{L^2_p} = \lal f,L_0g\ral_{L^2_p}.
\]
For $P_1$, by its definition,
\[
    \lal P_1f,g\ral_{L^2_p} =  \lt\lal  \sum_{i=1}^4\lal f,v_i\ral_{L^2_p}v_i, g \rt\ral_{L^2_p}  = \sum_{i=1}^4 \lal f,v_i\ral_{L^2_p} \lal v_i,g\ral_{L^2_p} =  \lt\lal  f,  \sum_{i=1}^4\lal g,v_i\ral_{L^2_p}v_i
\rt\ral_{L^2_p} = \lal f,P_1g\ral_{L^2_p}.
\]
Thus $L=L_0+P_1$ is self-adjoint.

\medskip
\noindent
{\bf Step 2. Inclusion $\mcN\subset\ker L$.} From \eqref{eq:L0-sym-f}, we have
\[
\lal L_0g,g\ral_{L^2_p} = -M\int_{\mR^3} \mu_\hbar  \lt| \nabla_p\lt(\frac{g}{\sqrt{\mu_\hbar }}\rt) \rt|^2\,\rdp.
\]
Thus, if $g\in\ker L_0$, then $g/\sqrt{\mu_\hbar }$ is constant, and hence
\[
\ker L_0=\mathrm{span}\{\sqrt{\mu_\hbar }\}.
\]

We next show that $\mcN\subset\ker L$. Using \eqref{eq:elem} and the moment identities \eqref{eq:mom-id-mu0}, we compute the action of $L$ on the five collision invariants.

First,
\[
\nabla_p\sqrt{\mu_\hbar } = -\frac12p\eta_\hbar \sqrt{\mu_\hbar }.
\]
Thus the first term in the flux of $L$ gives
\[
\lt(\nabla_p\sqrt{\mu_\hbar } +\frac12p\eta_\hbar \sqrt{\mu_\hbar }\rt)\sqrt{\mu_\hbar }=0.
\]
Moreover,
\[
\int_{\mR^3}p\eta_\hbar \mu_\hbar \,\rdp=0,
\quad
\int_{\mR^3}(|p|^2\eta_\hbar -3)\mu_\hbar \,\rdp=0,
\]
where the second identity follows from
\[
\int_{\mR^3}|p|^2\eta_\hbar \mu_\hbar \,\rdp = 3\int_{\mR^3}\mu_\hbar \,\rdp.
\]
Hence, we have $L(\sqrt{\mu_\hbar })=0$. 

For $g=p_i\sqrt{\mu_\hbar }$, we have
\[
\nabla_p(p_i\sqrt{\mu_\hbar }) = \mfe_i\sqrt{\mu_\hbar } -\frac12p_i p\eta_\hbar \sqrt{\mu_\hbar },
\]
and subsequently,
\[
\lt(\nabla_p(p_i\sqrt{\mu_\hbar }) +\frac12p\eta_\hbar p_i\sqrt{\mu_\hbar }\rt)\sqrt{\mu_\hbar } = \mfe_i\mu_\hbar ,
\]
where $\{\mfe_i\}_{i=1}^3$ denotes the standard basis of $\R^3$. The temperature correction term vanishes by oddness:
\[
\int_{\mR^3}(|p|^2\eta_\hbar -3)p_i\mu_\hbar \,\rdp=0.
\]
The velocity correction term is
\[
-\frac{\mu_\hbar }{\int_{\mR^3}\mu_\hbar \,\rdp} \int_{\mR^3}p\eta_\hbar p_i\mu_\hbar \,\rdp.
\]
By radial symmetry and \eqref{eq:mom-id-mu0},
\[
\int_{\mR^3}p_jp_i\eta_\hbar \mu_\hbar \,\rdp = \delta_{ij}\frac13\int_{\mR^3}|p|^2\eta_\hbar \mu_\hbar \,\rdp = \delta_{ij}\int_{\mR^3}\mu_\hbar \,\rdp.
\]
Thus the two terms cancel, and $L(p_i\sqrt{\mu_\hbar })=0$ for $i=1,2,3$.

Finally, for $g=|p|^2\sqrt{\mu_\hbar }$, we compute
\[
\nabla_p(|p|^2\sqrt{\mu_\hbar }) = 2p\sqrt{\mu_\hbar } -\frac12|p|^2p\eta_\hbar \sqrt{\mu_\hbar },
\]
and hence
\[
\lt(\nabla_p(|p|^2\sqrt{\mu_\hbar }) +\frac12p\eta_\hbar |p|^2\sqrt{\mu_\hbar }\rt)\sqrt{\mu_\hbar } = 2p\mu_\hbar .
\]
The velocity correction term vanishes by oddness:
\[
\int_{\mR^3}p\eta_\hbar  |p|^2\mu_\hbar \,\rdp=0.
\]
The temperature correction term is
\[
-\frac{p\mu_\hbar }{\int_{\mR^3}|p|^2\mu_\hbar \,\rdp} \int_{\mR^3}(|p|^2\eta_\hbar -3)|p|^2\mu_\hbar \,\rdp.
\]
We claim that
\bq\label{eq:tem-m-i}
\int_{\mR^3}(|p|^2\eta_\hbar -3)|p|^2\mu_\hbar \,\rdp = 2\int_{\mR^3}|p|^2\mu_\hbar \,\rdp.
\eq
Indeed, using integration by parts and $\nabla_p\mu_\hbar =-p\eta_\hbar \mu_\hbar $, we have
\[
0 = \int_{\mR^3}\nabla_p\cdot(p|p|^2\mu_\hbar )\,\rdp = 5\int_{\mR^3}|p|^2\mu_\hbar \,\rdp - \int_{\mR^3}|p|^4\eta_\hbar \mu_\hbar \,\rdp.
\]
Therefore,
\[
\int_{\mR^3}|p|^4\eta_\hbar \mu_\hbar \,\rdp
=
5\int_{\mR^3}|p|^2\mu_\hbar \,\rdp,
\]
which implies \eqref{eq:tem-m-i}. Thus the temperature correction term is $-2p\mu_\hbar $, and it cancels the term $2p\mu_\hbar $. Hence, we have $L(|p|^2\sqrt{\mu_\hbar })=0$. This implies $\mcN\subset\ker L$.

\medskip
\noindent
{\bf Step 3. Dissipation identity.} Let
\[
h:=\frac{g}{\sqrt{\mu_\hbar }}.
\]
Using the definition of $L$, the identities
\[
\nabla_p\sqrt{\mu_\hbar } = -\frac12p\eta_\hbar \sqrt{\mu_\hbar },
\quad
\nabla_p\mu_\hbar  = -p\eta_\hbar \mu_\hbar ,
\]
and integration by parts, we first obtain
\begin{align}\label{projection}
\begin{aligned}
\frac{1}{M}\lal Lg,g\ral_{L^2_p} &= -\int_{\mR^3}\mu_\hbar |\nabla_p h|^2\,\rdp + \frac{1}{\int_{\mR^3}|p|^2\mu_\hbar \,\rdp} \lt( \int_{\mR^3}(|p|^2\eta_\hbar -3)g\sqrt{\mu_\hbar }\,\rdp \rt)^2 \cr
&\quad + \frac{1}{\int_{\mR^3}\mu_\hbar \,\rdp} \lt| \int_{\mR^3}p\eta_\hbar g\sqrt{\mu_\hbar }\,\rdp \rt|^2.
\end{aligned}
\end{align}
We rewrite the last two terms in terms of $h$. Since $g=h\sqrt{\mu_\hbar }$,
\[
\int_{\mR^3}p\eta_\hbar g\sqrt{\mu_\hbar }\,\rdp = \int_{\mR^3}p\eta_\hbar  h\mu_\hbar \,\rdp = -\int_{\mR^3}h\nabla_p\mu_\hbar \,\rdp = \int_{\mR^3}\mu_\hbar \nabla_p h\,\rdp.
\]
Similarly,
\[
\int_{\mR^3}(|p|^2\eta_\hbar -3)g\sqrt{\mu_\hbar }\,\rdp = \int_{\mR^3}(|p|^2\eta_\hbar -3)h\mu_\hbar \,\rdp.
\]
Using $\nabla_p\cdot(p\mu_\hbar )=3\mu_\hbar -|p|^2\eta_\hbar \mu_\hbar $, we get
\[
\int_{\mR^3}(|p|^2\eta_\hbar -3)h\mu_\hbar \,\rdp = -\int_{\mR^3}h\nabla_p\cdot(p\mu_\hbar )\,\rdp = \int_{\mR^3}p\mu_\hbar \cdot\nabla_p h\,\rdp.
\]
Therefore \eqref{projection} becomes
\bq\label{projection-h}
\frac{1}{M}\lal Lg,g\ral_{L^2_p} = -\int_{\mR^3}\mu_\hbar |\nabla_p h|^2\,\rdp + \frac{1}{\int_{\mR^3}|p|^2\mu_\hbar \,\rdp} \lt( \int_{\mR^3}p\mu_\hbar \cdot\nabla_p h\,\rdp \rt)^2 + \frac{1}{\int_{\mR^3}\mu_\hbar \,\rdp} \lt| \int_{\mR^3}\mu_\hbar \nabla_p h\,\rdp \rt|^2.
\eq

Define the vector-valued function
\[
\mcJ := \sqrt{\mu_\hbar }\nabla_p h - \frac{\int_{\mR^3}p\mu_\hbar \cdot\nabla_p h\,\rdp} {\int_{\mR^3}|p|^2\mu_\hbar \,\rdp} p\sqrt{\mu_\hbar } - \frac{\int_{\mR^3}\mu_\hbar \nabla_p h\,\rdp}{\int_{\mR^3}\mu_\hbar \,\rdp}\sqrt{\mu_\hbar }.
\]
Here the last term is understood componentwise. By expanding the square, we obtain
\begin{align*}
\|\mcJ\|_{L^2_p}^2 &= \int_{\mR^3}\mu_\hbar |\nabla_p h|^2\,\rdp + \frac{ \lt( \int_{\mR^3}p\mu_\hbar \cdot\nabla_p h\,\rdp \rt)^2}{\lt(\int_{\mR^3}|p|^2\mu_\hbar \,\rdp\rt)^2} \int_{\mR^3}|p|^2\mu_\hbar \,\rdp + \frac{ \lt| \int_{\mR^3}\mu_\hbar \nabla_p h\,\rdp\rt|^2}{\lt(\int_{\mR^3}\mu_\hbar \,\rdp\rt)^2}\int_{\mR^3}\mu_\hbar \,\rdp \cr
&\quad - \frac{2}{\int_{\mR^3}|p|^2\mu_\hbar \,\rdp}\lt(\int_{\mR^3}p\mu_\hbar \cdot\nabla_p h\,\rdp\rt)^2 - \frac{2}{\int_{\mR^3}\mu_\hbar \,\rdp} \lt|\int_{\mR^3}\mu_\hbar \nabla_p h\,\rdp\rt|^2 \cr
&\quad + \frac{2}{\lt(\int_{\mR^3}|p|^2\mu_\hbar \,\rdp\rt) \lt(\int_{\mR^3}\mu_\hbar \,\rdp\rt)} \lt(\int_{\mR^3}p\mu_\hbar \cdot\nabla_p h\,\rdp\rt) \cdot \lt(\int_{\mR^3}\mu_\hbar \nabla_p h\,\rdp\rt) \int_{\mR^3}p\mu_\hbar \,\rdp.
\end{align*}
The last term vanishes since $\int p\mu_\hbar \,\rdp=0$. Hence
\bq\label{eq:J-iden}
\|\mcJ\|_{L^2_p}^2 = \int_{\mR^3}\mu_\hbar |\nabla_p h|^2\,\rdp - \frac{1}{\int_{\mR^3}|p|^2\mu_\hbar \,\rdp} \lt( \int_{\mR^3}p\mu_\hbar \cdot\nabla_p h\,\rdp \rt)^2 - \frac{1}{\int_{\mR^3}\mu_\hbar \,\rdp} \lt| \int_{\mR^3}\mu_\hbar \nabla_p h\,\rdp \rt|^2.
\eq
Comparing \eqref{projection-h} and \eqref{eq:J-iden}, we get
\bq\label{eq:J-iden-f}
\frac{1}{M}\lal Lg,g\ral_{L^2_p} = -\|\mcJ\|_{L^2_p}^2 \le0.
\eq

\medskip
\noindent
{\bf Step 4. Inclusion $\ker L\subset\mcN$.} Let $g\in\ker L$. Then $\lal Lg,g\ral_{L^2_p} = 0$. By the dissipation identity \eqref{eq:J-iden-f}, we have $\mcJ=0$,
where
\[
\mcJ = \sqrt{\mu_\hbar }\nabla_p h - \frac{\int_{\mR^3}p\mu_\hbar \cdot\nabla_p h\,\rdp} {\int_{\mR^3}|p|^2\mu_\hbar \,\rdp} p\sqrt{\mu_\hbar } - \frac{\int_{\mR^3}\mu_\hbar \nabla_p h\,\rdp} {\int_{\mR^3}\mu_\hbar \,\rdp} \sqrt{\mu_\hbar },
\quad
h=\frac{g}{\sqrt{\mu_\hbar }}.
\]
Thus $\nabla_p h$ is an affine function of $p$, and hence $h=a+b\cdot p+c|p|^2$ for some constants $a,c\in\mR$ and $b\in\mR^3$. Consequently, $g=(a+b\cdot p+c|p|^2)\sqrt{\mu_\hbar }\in\mcN$. This proves $\ker L\subset\mcN$. Since the opposite inclusion was shown in Step 2, we conclude 
\[
\ker L=\mcN.
\]

\medskip
\noindent
{\bf Step 5. Microscopic coercivity.} We first provide a weighted Poincar\'e inequality associated with $\mu_\hbar $. Set
\[
\calM_0(p) := e^{-\lt(\frac{|p|^2}{2}+\theta_0\rt)} = e^{-U_0(p)},
\quad
U_0(p):=\frac{|p|^2}{2}+\theta_0.
\]
Then $\nabla_p^2U_0 = \mI_3$. Moreover, by \eqref{eq:mu0-comp}, the weight $\mu_\hbar $ is uniformly comparable
to $\calM_0$. Thus, the weighted Poincar\'e inequality in \cite[Corollary 3.4]{Lem00} yields
\bq\label{eq:wei-P-mu0}
\int_{\mR^3}|\psi|^2\mu_\hbar \,\rdp \le C \int_{\mR^3}|\nabla_p\psi|^2\mu_\hbar \,\rdp
\eq
for every $\psi$ satisfying $\int_{\mR^3}\psi\mu_\hbar \,\rdp=0$. 

Since $L$ is self-adjoint and $\ker L=\mcN$, we have
\[
\lal Lg,g\ral_{L^2_p} = \lal L(I-P)g,(I-P)g\ral_{L^2_p}.
\]
We consider the following semi-norm:
\[
\|g\|_K^2 := \lal -L_0g,g\ral_{L^2_p} = M\int_{\mR^3}\mu_\hbar  \lt| \nabla_p\lt(\frac{g}{\sqrt{\mu_\hbar }}\rt) \rt|^2\,\rdp.
\]
We first claim that there exists $\delta>0$ such that
\bq\label{coercivity 1}
\lal -Lg,g\ral_{L^2_p} \ge \delta\|(I-P)g\|_K^2.
\eq
Suppose the contrary. Then there exists a sequence $g_n$ such that, setting
\[
h_n:=\frac{(I-P)g_n}{\|(I-P)g_n\|_K},
\]
we have
\[
h_n\in\mcN^\perp,
\quad
\|h_n\|_K=1,
\quad
\lal -Lh_n,h_n\ral_{L^2_p}\to0.
\]
The orthogonality $h_n\in\mcN^\perp$ means
\bq\label{orthogonal}
\int_{\mR^3}h_n\sqrt{\mu_\hbar }\,\rdp = \int_{\mR^3}p_i h_n\sqrt{\mu_\hbar }\,\rdp = \int_{\mR^3}(|p|^2-\ph)h_n\sqrt{\mu_\hbar }\,\rdp = 0
\eq
for $i=1,2,3$.

We apply \eqref{eq:wei-P-mu0} to $\psi_n:=\frac{h_n}{\sqrt{\mu_\hbar }}$ with weight $\mu_\hbar $. Since
\[
\int_{\mR^3}\psi_n\mu_\hbar \,\rdp = \int_{\mR^3}h_n\sqrt{\mu_\hbar }\,\rdp = 0,
\]
there exists $C>0$ such that
\[
\|h_n\|_{L^2_p}^2 = \int_{\mR^3}|\psi_n|^2\mu_\hbar \,\rdp \le C\int_{\mR^3}|\nabla_p\psi_n|^2\mu_\hbar \,\rdp \le C\|h_n\|_K^2 = C.
\]
Thus, up to a subsequence, $h_n$ converges weakly in $L^2_p$ and weakly with respect to the $K$-norm to some $h_0$. Moreover, since $P_1$ is a finite-rank operator, the weak convergence of $h_n$ implies $P_1h_n\to P_1h_0$ strongly in $L^2_p$. Consequently,
\[
\lal P_1h_n,h_n\ral_{L^2_p} \to \lal P_1h_0,h_0\ral_{L^2_p}.
\]
Since
\[
\lal -Lh_n,h_n\ral_{L^2_p} = \lal -L_0h_n,h_n\ral_{L^2_p} - \lal P_1h_n,h_n\ral_{L^2_p} = 1-\lal P_1h_n,h_n\ral_{L^2_p},
\]
we obtain $\lal P_1h_0,h_0\ral_{L^2_p}=1$. On the other hand, by lower semicontinuity, $\|h_0\|_K^2\le1$. Hence,
\[
\lal -Lh_0,h_0\ral_{L^2_p} = \|h_0\|_K^2-\lal P_1h_0,h_0\ral_{L^2_p} \le0.
\]
By \eqref{eq:J-iden-f}, $\lal -Lh_0,h_0\ral_{L^2_p}\ge0$, and thus $\lal -Lh_0,h_0\ral_{L^2_p}=0$, that is, $h_0\in\ker L=\mcN$. Passing to the limit in the orthogonality relations \eqref{orthogonal}, we also have $h_0\in\mcN^\perp$. Hence $h_0=0$. This contradicts $\lal P_1h_0,h_0\ral_{L^2_p}=1$. Therefore \eqref{coercivity 1} holds.

We next upgrade \eqref{coercivity 1} to the $D$-norm. Applying \eqref{eq:wei-P-mu0} to $\psi:=\frac{(I-P)g}{\sqrt{\mu_\hbar }}$,
we note that
\[
\int_{\mR^3}\psi\mu_\hbar \,\rdp = \int_{\mR^3}(I-P)g\sqrt{\mu_\hbar }\,\rdp = 0,
\]
since $\sqrt{\mu_\hbar }\in\mcN$. Hence there exists a constant $C>0$ such that
\bq\label{ineq}
\|(I-P)g\|_{L^2_p}^2 \le C\|(I-P)g\|_K^2.
\eq
Combining \eqref{coercivity 1} and \eqref{ineq}, we obtain
\bq\label{eq:L-L2-coer}
\lal -Lg,g\ral_{L^2_p} \ge C\|(I-P)g\|_{L^2_p}^2.
\eq

We also derive an explicit estimate involving the $D$-norm. From the formula
for $L_0$, we have
\begin{align}\label{ineq 2}
\begin{aligned}
\lal -L_0h,h\ral_{L^2_p} &= M\int_{\mR^3} \lt| \nabla_p h+\frac12p\eta_\hbar h \rt|^2\,\rdp \cr
&= M\int_{\mR^3} |\nabla_p h|^2 + \frac14|p\eta_\hbar h|^2 + p\eta_\hbar  h\cdot\nabla_p h \,\rdp \cr
&\ge \frac{M}{4}|h|_D^2 - C\|h\|_{L^2_p}^2.
\end{aligned}
\end{align}
Here, the last inequality follows from the boundedness of $\eta_\hbar $ and its derivatives, after integrating the cross term by parts. Moreover,
\bq\label{ineq 3}
\lal P_1h,h\ral_{L^2_p} = M\lt[ \frac{ \lt(\int_{\mR^3}(|p|^2\eta_\hbar -3)h\sqrt{\mu_\hbar }\,\rdp\rt)^2} {\int_{\mR^3}|p|^2\mu_\hbar \,\rdp} + \frac{ \lt|\int_{\mR^3}p\eta_\hbar h\sqrt{\mu_\hbar }\,\rdp\rt|^2}{\int_{\mR^3}\mu_\hbar \,\rdp}\rt] \le C\|h\|_{L^2_p}^2.
\eq
Taking $h=(I-P)g$, we combine \eqref{ineq 2} and \eqref{ineq 3} to obtain
\begin{align*} 
\lal -Lg,g\ral_{L^2_p} &= \lal -L_0(I-P)g,(I-P)g\ral_{L^2_p} - \lal P_1(I-P)g,(I-P)g\ral_{L^2_p} \cr
&\ge \frac{M}{4}|(I-P)g|_D^2 - C\|(I-P)g\|_{L^2_p}^2.
\end{align*}
Finally, for sufficiently small $\e_0>0$, using \eqref{eq:L-L2-coer} we write
\begin{align*}
\lal -Lg,g\ral_{L^2_p} &= (1-\e_0)\lal -Lg,g\ral_{L^2_p} + \e_0\lal -Lg,g\ral_{L^2_p} \cr
&\ge (1-\e_0)C\|(I-P)g\|_{L^2_p}^2 + \e_0 \lt\{ \frac{M}{4}|(I-P)g|_D^2 - C\|(I-P)g\|_{L^2_p}^2 \rt\}.
\end{align*}
Choosing $\e_0>0$ sufficiently small so that the coefficient of $\|(I-P)g\|_{L^2_p}^2$ remains nonnegative, we obtain
\[
\lal -Lg,g\ral_{L^2_p} \ge \lambda_0 |(I-P)g|_D^2
\]
for some $\lambda_0>0$. This proves \eqref{eq:L-coercivity}. The nonpositivity of $L$ follows immediately, see also \eqref{eq:J-iden-f}.
\end{proof}

The previous lemma gives the basic microscopic coercivity of the linearized operator. In the higher-order energy estimates, however, $p$-derivatives fall on the coefficients of $L$ and on the macroscopic projection $P$. We therefore provide the following derivative version of the coercivity estimate.

\begin{lemma}
Let $\beta$ be a multi-index with $|\beta|\ge1$. Then there exist positive constants $\lambda_1$ and $C_\beta$ such that
\bq\label{eq:L-p-der-coer}
\lt\lal \pa_p^\beta Lg, \pa_p^\beta g \rt\ral_{L^2_p} \le -\lambda_1 \lt| \pa_p^\beta(I-P)g \rt|_D^2 + C_\beta \lt( \|g\|_{L^2_p}^2 + \sum_{1\le|\nu|\le|\beta|} \lt| \pa_p^{\beta-\nu}(I-P)g \rt|_D^2 \rt).
\eq
\end{lemma}

\begin{proof}
Since $L(Pg)=0$, we have $\pa_p^\beta Lg = \pa_p^\beta L((I-P)g)$. Set $h := (I-P)g$. Recall that $L=L_0+P_1$,
where
\[
L_0 h = M\lt[ \Delta_p h + h\lt\{ \frac32\eta_\hbar  - |p|^2 \lt( \frac14\eta_\hbar ^2+\hbar\kappa\mu_\hbar  \rt) \rt\}\rt].
\]
For notational simplicity, set
\[
V(p) := \frac32\eta_\hbar  - |p|^2 \lt( \frac14\eta_\hbar ^2+\hbar\kappa\mu_\hbar  \rt).
\]
Then
\[
\pa_p^\beta L_0h = L_0(\pa_p^\beta h) + M \sum_{1\le|\nu|\le|\beta|} C_{\beta,\nu} \pa_p^{\beta-\nu}h\, \pa_p^\nu V.
\]
Using $\pa_p^\beta g = \pa_p^\beta h + \pa_p^\beta Pg$, we decompose
\begin{align*} 
\lt\lal \pa_p^\beta Lg, \pa_p^\beta g \rt\ral_{L^2_p} &= \lt\lal L_0(\pa_p^\beta h), \pa_p^\beta Pg \rt\ral_{L^2_p} + \lt\lal L_0(\pa_p^\beta h), \pa_p^\beta h \rt\ral_{L^2_p} \cr
&\quad + M \int_{\mR^3} \sum_{1\le|\nu|\le|\beta|} C_{\beta,\nu} \pa_p^{\beta-\nu}h\, \pa_p^\nu V\, \pa_p^\beta g \,\rdp + \lt\lal \pa_p^\beta P_1h, \pa_p^\beta g \rt\ral_{L^2_p} \cr
&=: I_1+I_2+I_3+I_4.
\end{align*}

\medskip
\noindent
{\it Estimate of $I_1$.}
Using the divergence form of $L_0$, we have
\begin{align*}
I_1 &= M \int_{\mR^3} \frac{\pa_p^\beta Pg}{\sqrt{\mu_\hbar }} \nabla_p\cdot \lt\{ \mu_\hbar  \nabla_p \lt( \frac{\pa_p^\beta h}{\sqrt{\mu_\hbar }} \rt) \rt\} \,\rdp \cr
&= - M \int_{\mR^3} \lt\{ \nabla_p\pa_p^\beta(Pg) + \frac12p\eta_\hbar \pa_p^\beta(Pg) \rt\} \cdot \lt\{ \nabla_p\pa_p^\beta h + \frac12p\eta_\hbar \pa_p^\beta h \rt\} \,\rdp.
\end{align*}
Since $Pg$ is a finite linear combination of the basis functions of $\mcN$, the coefficient bound \eqref{eq:P-coeff-b} gives
\[
\lt\| \nabla_p\pa_p^\beta(Pg) + \frac12p\eta_\hbar \pa_p^\beta(Pg) \rt\|_{L^2_p} \le C_\beta \|g\|_{L^2_p}.
\]
Therefore, for any $\eta>0$,
\bq\label{eq:I1-est}
|I_1| \le C_\beta \|g\|_{L^2_p} \lt| \pa_p^\beta h \rt|_D \le \eta \lt| \pa_p^\beta h \rt|_D^2 + C_{\beta,\eta} \|g\|_{L^2_p}^2.
\eq

\medskip
\noindent
{\it Estimate of $I_2$.}
We first apply the microscopic coercivity obtained in Lemma \ref{lem:Prop-L} to $\pa_p^\beta h$. Since $\pa_p^\beta h$ is not necessarily orthogonal to $\mcN$, we have
\[
\lt\lal L(\pa_p^\beta h), \pa_p^\beta h \rt\ral_{L^2_p} \le -\lambda_0 \lt| (I-P)\pa_p^\beta h \rt|_D^2.
\]
The macroscopic part of $\pa_p^\beta h$ is lower order. Indeed, for each basis function $e_i$ of $\mcN$, $\lt\lal \pa_p^\beta h, e_i \rt\ral_{L^2_p} = (-1)^{|\beta|} \lt\lal h, \pa_p^\beta e_i \rt\ral_{L^2_p}$. By the coefficient bound \eqref{eq:P-coeff-b}, we get  $\lt| \lt\lal \pa_p^\beta h, e_i \rt\ral_{L^2_p} \rt| \le C_\beta \|h\|_{L^2_p} \le C_\beta \|g\|_{L^2_p}$. This gives
\[
\lt| P\pa_p^\beta h \rt|_D^2 \le C_\beta \|g\|_{L^2_p}^2.
\]
Consequently,
\[
\lt\lal L(\pa_p^\beta h), \pa_p^\beta h \rt\ral_{L^2_p} \le -\lambda_2 \lt| \pa_p^\beta h \rt|_D^2 + C_\beta \|g\|_{L^2_p}^2
\]
for some $\lambda_2>0$.

Since
\[
I_2 = \lt\lal L_0(\pa_p^\beta h), \pa_p^\beta h \rt\ral_{L^2_p} = \lt\lal L(\pa_p^\beta h), \pa_p^\beta h \rt\ral_{L^2_p} - \lt\lal P_1(\pa_p^\beta h), \pa_p^\beta h \rt\ral_{L^2_p},
\]
and since $P_1$ is nonnegative, we obtain
\bq\label{eq:I2-est}
I_2 \le -\lambda_2 \lt| \pa_p^\beta h \rt|_D^2 + C_\beta \|g\|_{L^2_p}^2.
\eq

\medskip
\noindent
{\it Estimate of $I_3$.} By the coefficient bound \eqref{eq:P-coeff-b} and the definition of the $D$-norm, we have
\[
\lt\| \pa_p^{\beta-\nu}h\, \pa_p^\nu V \rt\|_{L^2_p} \le C_\beta \lt| \pa_p^{\beta-\nu}h \rt|_D
\]
for every multi-index $\nu$ satisfying $1\le|\nu|\le|\beta|$. Moreover,
\bq\label{eq:pb-g-L2-b}
\|\pa_p^\beta g\|_{L^2_p} \le \|\pa_p^\beta h\|_{L^2_p} + \|\pa_p^\beta Pg\|_{L^2_p} \le C \lt| \pa_p^\beta h \rt|_D + C_\beta \|g\|_{L^2_p}.
\eq
Therefore,
\[
|I_3| \le C_\beta \sum_{1\le|\nu|\le|\beta|} \lt| \pa_p^{\beta-\nu}h \rt|_D \lt( \lt| \pa_p^\beta h \rt|_D + \|g\|_{L^2_p} \rt).
\]
Applying Young's inequality, for any $\eta>0$, we obtain
\bq\label{eq:I3-est}
|I_3| \le \eta \lt| \pa_p^\beta h \rt|_D^2 + C_{\beta,\eta} \lt( \|g\|_{L^2_p}^2 + \sum_{1\le|\nu|\le|\beta|} \lt| \pa_p^{\beta-\nu}h \rt|_D^2 \rt).
\eq

\medskip
\noindent
{\it Estimate of $I_4$.} By the definition of $P_1$, we find
\[
\pa_p^\beta P_1h = \sum_{i=1}^4 \lt\lal h,v_i \rt\ral_{L^2_p} \pa_p^\beta v_i.
\]
Since the functions $\pa_p^\beta v_i$ are controlled by \eqref{eq:P-coeff-b}, we have $\|\pa_p^\beta P_1h\|_{L^2_p} \le C_\beta \|h\|_{L^2_p} \le C_\beta \|g\|_{L^2_p}$. Using \eqref{eq:pb-g-L2-b}, we obtain
\bq\label{eq:I4-est}
|I_4|  \le C_\beta \|g\|_{L^2_p} \|\pa_p^\beta g\|_{L^2_p}  \le C_\beta \|g\|_{L^2_p} \lt( \lt| \pa_p^\beta h \rt|_D + \|g\|_{L^2_p} \rt)  \le \eta \lt| \pa_p^\beta h \rt|_D^2 + C_{\beta,\eta} \|g\|_{L^2_p}^2.
\eq

Combining \eqref{eq:I1-est}, \eqref{eq:I2-est}, \eqref{eq:I3-est}, and \eqref{eq:I4-est}, and choosing $\eta>0$ sufficiently small, we obtain
\[
\lt\lal \pa_p^\beta Lg, \pa_p^\beta g \rt\ral_{L^2_p} \le -\lambda_1 \lt| \pa_p^\beta h \rt|_D^2 + C_\beta \lt( \|g\|_{L^2_p}^2 + \sum_{1\le|\nu|\le|\beta|} \lt| \pa_p^{\beta-\nu}h \rt|_D^2 \rt)
\]
for some $\lambda_1>0$. Since $h=(I-P)g$, this gives \eqref{eq:L-p-der-coer}.
\end{proof}

%
%
%
%
%
%
%
%
%
%
\section{Nonlinear estimates and local well-posedness}\label{sec:non-lwp}
%
%
%
%
%
%
%
%
%
%
\subsection{Estimates for the nonlinear operator}
\label{ssec:non-est}

In this subsection, we collect estimates for the nonlinear operator $\Gamma$ appearing in the perturbative equation \eqref{eq:linearized}. These estimates will be used both in the local well-posedness argument and later in the global energy estimates and large-time analysis. We first recall several Sobolev and Moser-type inequalities.

\begin{lemma} \label{Sobolev ineq}
Let $s\ge4$. The following estimates hold.
\begin{enumerate}
    \item For $d\ge3$ and $f\in H^{[\frac d2]+1}(\mR^d)$, we have
    \[
    \|f\|_{L^\infty}     \le    C\|\nabla f\|_{H^{[\frac d2]}}.
    \]
    \item If $f,h\in H^s\cap L^\infty$, then for $|\gamma|\le s$,
    \[
    \|\pa^\gamma(fh)\|_{L^2}    \le    C\|f\|_{L^\infty}\|\pa^\gamma h\|_{L^2}    +    C\|\pa^\gamma f\|_{L^2}\|h\|_{L^\infty}.
    \]
    \item Let $|\gamma|\in\mN$, $1\le q\le\infty$, and let $H\in C^{|\gamma|}(B(0,\|f\|_{L^\infty}))$. If $f\in L^\infty\cap W^{|\gamma|,q}$, then there exists a constant $C=C(|\gamma|,q,H)$ such that
    \[
    \|\pa^\gamma H(f)\|_{L^q}    \le    C(1+\|f\|_{L^\infty}^{|\gamma|-1})\|\pa^\gamma f\|_{L^q}.
    \]
    \item If $f \in H^s$, then for any multi-indices $\alpha$ and $\gamma$ satisfying $\gamma\le\alpha$ and $|\alpha|\le s-1$, 
\bq\label{GN ineq 1}
\|\pa_x^{\alpha-\gamma} f\|_{L^3_x} \|\pa_x^\gamma f\|_{L^6_x}    \le    C\|f\|_{H^s}\|\pa_x^\alpha\nabla_x f\|_{L^2_x}.
\eq

    \item  If $f \in H^s$, then for any multi-indices $\alpha$ and $\gamma$ satisfying $\gamma\le\alpha$ and $|\alpha|\le s$,  
\bq\label{GN ineq 2}
\|\pa_x^\gamma f\|_{L^3_x}\|\pa_x^{\alpha-\gamma} f\|_{L^6_x}    \le    C\|f\|_{H^s}\|\pa_x^\alpha f\|_{L^2_x}.
\eq
\end{enumerate}
\end{lemma}

\begin{proof}
The first inequality is the standard Sobolev embedding. The second and third are Moser-type estimates; see, for instance, \cite[Proposition 2.1]{Maj84}. We prove only the interpolation estimates \eqref{GN ineq 1} and \eqref{GN ineq 2}, following the strategy used in \cite{GW12}.

Let $|\alpha|\le s-1$. By the Gagliardo--Nirenberg interpolation inequality,
\[
\|\pa_x^{\alpha-\gamma} f\|_{L^3_x} \|\pa_x^\gamma f\|_{L^6_x} \le C \|\pa_x^{\omega_1}f\|_{L^2_x}^{\frac{|\gamma|+1}{|\alpha|+1}} \|\pa_x^\alpha\nabla_x f\|_{L^2_x}^{1-\frac{|\gamma|+1}{|\alpha|+1}} \|f\|_{L^2_x}^{1-\frac{|\gamma|+1}{|\alpha|+1}} \|\pa_x^\alpha\nabla_x f\|_{L^2_x}^{\frac{|\gamma|+1}{|\alpha|+1}}.
\]
Here $\omega_1$ is determined by
\[
\frac{|\alpha|-|\gamma|}{3}-\frac13 = \lt(\frac{|\omega_1|}{3}-\frac12\rt) \frac{|\gamma|+1}{|\alpha|+1} + \lt(\frac{|\alpha|+1}{3}-\frac12\rt) \lt(1-\frac{|\gamma|+1}{|\alpha|+1}\rt).
\]
Thus
\[
|\omega_1|
=
\frac{|\alpha|+1}{|\gamma|+1}
\lt(\frac32-1\rt)
\le
\frac{|\alpha|+1}{2}
\le s.
\]
This gives $\|\pa_x^{\omega_1}f\|_{L^2_x} + \|f\|_{L^2_x} \le C\|f\|_{H^s}$, and hence
\[
\|\pa_x^\gamma f\,\pa_x^{\alpha-\gamma}f\|_{L^2_x}
\le
C\|f\|_{H^s}\|\pa_x^\alpha\nabla_x f\|_{L^2_x}.
\]
This proves \eqref{GN ineq 1}.

We next prove \eqref{GN ineq 2}. By symmetry, we may assume $1\le|\gamma|\le [\frac{|\alpha|}2]$; the case $|\gamma|=0$ is immediate from Sobolev embedding. By the Gagliardo--Nirenberg interpolation inequality,
\[
\|\pa_x^\gamma f\|_{L^3_x}\|\pa_x^{\alpha-\gamma} f\|_{L^6_x} \le C \|\pa_x^{\omega_2}f\|_{L^2_x}^{1-\frac{|\gamma|-1}{|\alpha|}} \|\pa_x^\alpha f\|_{L^2_x}^{\frac{|\gamma|-1}{|\alpha|}} \|f\|_{L^2_x}^{\frac{|\gamma|-1}{|\alpha|}} \|\pa_x^\alpha f\|_{L^2_x}^{1-\frac{|\gamma|-1}{|\alpha|}}.
\]
Here $\omega_2$ is determined by
\[
\frac{|\gamma|}{3}-\frac13 = \lt(\frac{|\omega_2|}{3}-\frac12\rt) \lt(1-\frac{|\gamma|-1}{|\alpha|}\rt) + \lt(\frac{|\alpha|}{3}-\frac12\rt) \frac{|\gamma|-1}{|\alpha|}.
\]
Thus
\[
|\omega_2| = \frac{3|\alpha|}{2(|\alpha|-|\gamma|+1)} \le 3.
\]
Since $s\ge4$, we have  $\|\pa_x^{\omega_2}f\|_{L^2_x} + \|f\|_{L^2_x} \le C\|f\|_{H^s}$. Consequently,
\[
\|\pa_x^\gamma f\,\pa_x^{\alpha-\gamma}f\|_{L^2_x} \le C\|f\|_{H^s}\|\pa_x^\alpha f\|_{L^2_x}.
\]
This completes the proof.
\end{proof}
The ranges of $\alpha$ in the inequalities \eqref{GN ineq 1} and \eqref{GN ineq 2} overlap, but the two inequalities are used in different places depending on their usage. In particular, \eqref{GN ineq 2} can be used in the top-order case $|\alpha|=s$.

We next state the nonlinear estimates needed below. We first isolate estimates for the nonlinear macroscopic remainders $N_u$ and $N_\Theta$. These bounds will be used repeatedly in the estimates of $\Gamma(g)$. 
\begin{lemma}\label{lem:Nu-NT-est}
Let $T_*>0$ and $s\ge4$ and suppose that $g\in \mcC([0,T_*];H^s(\R^3 \times \R^3))$ satisfies $\|g\|_{H^s}\ll1$. Then, for every multi-index $\alpha$ with $|\alpha|= s$, the nonlinear macroscopic remainders $N_u$ and $N_\Theta$ satisfy
\bq\label{Nu-NT-bound-top}
\|\pa_x^\alpha N_u\|_{L^2_x} + \|\pa_x^\alpha N_\Theta\|_{L^2_x} \le C\|g\|_{H^s}\|\pa_x^\alpha g\|_{L^2_{x,p}}.
\eq
Moreover, if $|\alpha|\le s-1$, then
\bq\label{Nu-NT-bound-low}
\|\pa_x^\alpha N_u\|_{L^2_x} + \|\pa_x^\alpha N_\Theta\|_{L^2_x} \le C\|g\|_{H^s}\|\pa_x^\alpha\nabla_x g\|_{L^2_{x,p}}. 
\eq
\end{lemma}

\begin{proof}
We only give the proof for $N_u$, since $N_\Theta$ is estimated in the same way. Recall from Lemma \ref{lem:decomp} that
\[
N_u = - \frac{\intr p\eta_\hbar g\sqrt{\mu_\hbar }\,\rdp}{\intr\mu_\hbar \,\rdp}R_1(g) + \frac{\hbar\kappa\intr p\mu_\hbar  g^2\,\rdp}{\intr\mu_\hbar \,\rdp} \lt(1-R_1(g)\rt),
\]
where
\[
R_1(g) = \frac{A(g)}{\intr\mu_\hbar \,\rdp} -  \frac{A(g)^2}{\intr\mu_\hbar \,\rdp\intr\{\mu_\hbar +\eta_\hbar g\sqrt{\mu_\hbar }+\hbar\kappa\mu_\hbar g^2\}\,\rdp},
\]
and
\[
A(g) = \intr \eta_\hbar g\sqrt{\mu_\hbar }\,\rdp + \hbar\kappa\intr \mu_\hbar g^2\,\rdp.
\]
By the exponential decay of $\mu_\hbar $ and the boundedness of $\eta_\hbar $, all the moments appearing above are controlled by the corresponding $L^2_p$ norms. In particular,
\[
\lt| \intr (1+|p|^2)\eta_\hbar g\sqrt{\mu_\hbar }\,\rdp \rt| + \lt| \intr (1+|p|^2)\mu_\hbar g^2\,\rdp \rt| \le C\lt(\|g\|_{L^2_p}+\|g\|_{L^2_p}^2\rt).
\]
Since $\|g\|_{H^s}\ll1$, the denominator
\[
\intr \{\mu_\hbar +\eta_\hbar g\sqrt{\mu_\hbar }+\hbar\kappa\mu_\hbar g^2\}\,\rdp
\]
is bounded away from zero. Hence, the map $z\mapsto \frac{z^2}{1+z}$ is evaluated only in a small neighborhood of the origin when applied to the normalized moment $A(g)/\intr\mu_\hbar \,\rdp$.

We only present the estimate for one of the most involved nonlinear terms. The remaining terms in $N_u$ and $N_\Theta$ have the same structure or are simpler, and can be treated by the same Sobolev--Moser estimates. Set
\[
\calA(g) := \frac{ \intr \lt( \eta_\hbar g\sqrt{\mu_\hbar } + \hbar\kappa\mu_\hbar g^2 \rt)\,\rdp }{\intr \mu_\hbar \,\rdp} = \frac{A(g)}{\intr \mu_\hbar \,\rdp}.
\]
By the exponential decay of $\mu_\hbar $, the boundedness of $\eta_\hbar $, and the Moser product estimate in Lemma \ref{Sobolev ineq}, we have, for $0\le|\gamma|\le s$ and $1\le q\le\infty$,
\[
\|\pa_x^\gamma\calA(g)\|_{L^q_x} \le C \lt( 1+\|g\|_{L^\infty_{x,p}} \rt) \|\pa_x^\gamma g\|_{L^q_xL^2_p} \le C \|\pa_x^\gamma g\|_{L^q_xL^2_p},
\]
where the last inequality follows from the smallness of $\|g\|_{H^s}$. In particular, by the Sobolev embedding in $x$, we get $\|\calA(g)\|_{L^\infty_x} \le C \lt( \|g\|_{H^s} + \|g\|_{H^s}^2 \rt)$. Hence, by taking $\|g\|_{H^s}$ sufficiently small, we may assume that
\[
\|\calA(g)\|_{L^\infty_x}<\frac12.
\]

We now consider the function
\[
H(z):=\frac{z^2}{1+z}.
\]
Since $H$ is smooth on the range of $\calA(g)$, the composition estimate in Lemma \ref{Sobolev ineq} gives, for $1\le|\gamma|\le s$,
\[
\|\pa_x^\gamma H(\calA(g))\|_{L^q_x} \le C \lt( 1+\|\calA(g)\|_{L^\infty_x}^{|\gamma|-1} \rt) \|\pa_x^\gamma\calA(g)\|_{L^q_x} \le C \|\pa_x^\gamma g\|_{L^q_xL^2_p}.
\]
For $|\gamma|=0$, using
\[
|H(z)| \le \frac{|z|^2}{1-|z|} \le 2|z|^2 \quad \text{for } |z|<\frac12,
\]
we obtain
\[
\|H(\calA(g))\|_{L^q_x} \le 2\|\calA(g)\|_{L^\infty_x} \|\calA(g)\|_{L^q_x} \le C\|g\|_{H^s}\|g\|_{L^q_xL^2_p}  \le C\|g\|_{L^q_xL^2_p}.
\]
Consequently, for every multi-index $\gamma$ with $|\gamma|\le s$, we have
\bq\label{eq:H-A-deri-b}
\|\pa_x^\gamma H(\calA(g))\|_{L^q_x} \le C \|\pa_x^\gamma g\|_{L^q_xL^2_p}.
\eq

One component of $N_u$ has the form
\[
\frac{\intr p\eta_\hbar g\sqrt{\mu_\hbar }\,\rdp}{\intr \mu_\hbar \,\rdp} H(\calA(g)).
\]
For this term, by Leibniz's rule, H\"older's inequality, Minkowski's inequality, and \eqref{eq:H-A-deri-b}, we have
\begin{align*}
\lt\| \pa_x^\alpha \lt\{ \frac{\intr p\eta_\hbar g\sqrt{\mu_\hbar }\,\rdp}{\intr\mu_\hbar \,\rdp} H(\calA(g)) \rt\} \rt\|_{L^2_x} 
&\le C \sum_{\gamma\le\alpha} \lt\| \pa_x^{\alpha-\gamma} \lt( \frac{\intr p\eta_\hbar g\sqrt{\mu_\hbar }\,\rdp}{\intr\mu_\hbar \,\rdp} \rt) \pa_x^\gamma H(\calA(g)) \rt\|_{L^2_x} \cr
&\le C \sum_{\gamma\le\alpha} \|\pa_x^{\alpha-\gamma}g\|_{L^3_xL^2_p} \|\pa_x^\gamma g\|_{L^6_xL^2_p}.
\end{align*}
If $|\alpha|\le s-1$, then the Gagliardo--Nirenberg estimate
\eqref{GN ineq 1} gives
\[
\sum_{\gamma\le\alpha} \|\pa_x^{\alpha-\gamma}g\|_{L^3_xL^2_p} \|\pa_x^\gamma g\|_{L^6_xL^2_p} \le C\|g\|_{H^s} \|\pa_x^\alpha\nabla_xg\|_{L^2_{x,p}}.
\]
If $|\alpha|=s$, then \eqref{GN ineq 2} gives
\[
\sum_{\gamma\le\alpha} \|\pa_x^{\alpha-\gamma}g\|_{L^3_xL^2_p} \|\pa_x^\gamma g\|_{L^6_xL^2_p} \le C\|g\|_{H^s} \|\pa_x^\alpha g\|_{L^2_{x,p}}.
\]
 
The term displayed above is one of the higher-order contributions to $N_u$. The leading nonlinear terms are quadratic in $g$; for instance, they contain a product of a linear moment of $g$ with
$\calA(g)/(1+\calA(g))$. All remaining components are estimated in the same way, or more directly, since they contain either an explicit quadratic moment of $g$ or a product of one linear moment and one quadratic remainder. Therefore,
\[
\|\pa_x^\alpha N_u\|_{L^2_x}
\le
\begin{cases}
C\|g\|_{H^s}\|\pa_x^\alpha\nabla_xg\|_{L^2_{x,p}},
& |\alpha|\le s-1,\\[1mm]
C\|g\|_{H^s}\|\pa_x^\alpha g\|_{L^2_{x,p}},
& |\alpha|=s.
\end{cases}
\]

It remains only to comment on $N_\Theta$. From Lemma \ref{lem:decomp},
\[
N_\Theta = \frac{\hbar\kappa}{3}\intr |p|^2\mu_\hbar g^2\,\rdp - \frac13 \lt( \frac{\intr p\eta_\hbar g\sqrt{\mu_\hbar }\,\rdp}{\intr\mu_\hbar \,\rdp} + N_u \rt) \cdot \lt( \intr p\eta_\hbar g\sqrt{\mu_\hbar }\,\rdp + \hbar\kappa\intr p\mu_\hbar g^2\,\rdp \rt).
\]
Thus each term in $N_\Theta$ is either quadratic in $g$ or contains the already estimated factor $N_u$ multiplied by a linear or quadratic moment of $g$. Repeating the preceding product estimates gives the same bounds for $N_\Theta$. This proves \eqref{Nu-NT-bound-top} and \eqref{Nu-NT-bound-low}.
\end{proof}

We now state the nonlinear estimates for $\Gamma$. The first one is a general bound for $\Gamma(g)$ tested against an arbitrary function and will be used in the local well-posedness argument.

 For later use, we introduce the notation
\[
\|g\|_{D,s}^2 := \sum_{|\alpha|+|\beta|\le s} \|\pa_x^\alpha\pa_p^\beta g\|_D^2.
\]

\begin{lemma}\label{Gamma 1}
Let $T_*>0$ and $s\ge4$. Suppose that $g,h\in \mcC([0,T_*];H^s(\R^3 \times \R^3))$ satisfy $\|g\|_{D,s}<\infty$ and $\|g\|_{H^s}\ll1$. Then, for $|\alpha|+|\beta|\le s$, we have
\bq\label{eq:Gamma1-est}
\lt| \lt\lal \pa_x^\alpha\pa_p^\beta\Gamma(g), \pa_x^\alpha\pa_p^\beta h \rt\ral_{L^2_{x,p}} \rt| \le C\|g\|_{H^s} \lt( \|\pa_x^\alpha\pa_p^\beta (I-P)h\|_D + \|\pa_x^\alpha h\|_{L^2_{x,p}} \rt) \|g\|_{D,s}.
\eq
\end{lemma}

\begin{proof}
We first prove the estimate for smooth functions with sufficient decay in $p$. The general case follows by a standard density argument.

The proof is based on the moment bounds, the exponential decay of the $p$-dependent coefficients, and the Sobolev--Moser estimates in Lemma \ref{Sobolev ineq}. Since all terms in $\Gamma(g)$ are estimated in a similar way, we give the details only for two representative contributions from \eqref{eq:non-flux}, suppressing harmless constants and signs:
\bq\label{nonlinear terms}
\Gamma_1 := \frac{1}{\sqrt{\mu_\hbar }} \nabla_p\cdot \lt[ \mathscr c(g)\nabla_p(g\sqrt{\mu_\hbar }) \rt], 
\quad 
\Gamma_2 := -\frac{1}{\sqrt{\mu_\hbar }} \nabla_p\cdot \lt[ N_u \lt\{ \mu_\hbar +\eta_\hbar g\sqrt{\mu_\hbar } +\hbar\kappa\mu_\hbar g^2 \rt\} \rt].
\eq
The remaining terms are handled in the same manner.

We introduce the first-order operators
\[
\calA_\pm \phi := \nabla_p\phi\pm \frac12p\eta_\hbar \phi.
\]
By \eqref{eq:elem}, we have 
\bq\label{eq:Apm-iden}
\nabla_p\lt(\frac{\phi}{\sqrt{\mu_\hbar }}\rt) = \frac{1}{\sqrt{\mu_\hbar }}\calA_+\phi,
\quad
\nabla_p(\phi\sqrt{\mu_\hbar }) = \sqrt{\mu_\hbar }\calA_-\phi.
\eq

For a sufficiently regular vector-valued function $F=F(p)$, define
\[
\calT F := \frac{1}{\sqrt{\mu_\hbar }} \nabla_p\cdot(\sqrt{\mu_\hbar }F).
\]
Using \eqref{eq:Apm-iden} and integration by parts in $p$, we obtain
\bq\label{eq:wei-f-pair}
\lt\lal \calT F,\phi \rt\ral_{L^2_p} = - \lt\lal F,\calA_+\phi \rt\ral_{L^2_p}.
\eq
 
Since
\[
\calT F = \nabla_p\cdot F - \frac12p\eta_\hbar \cdot F,
\]
we have
\[
\pa_p^\beta\calT F = \calT(\pa_p^\beta F) - \frac12 \sum_{0<\nu\le\beta} C_{\beta,\nu} \pa_p^\nu(p\eta_\hbar ) \cdot \pa_p^{\beta-\nu}F.
\]
Using \eqref{eq:wei-f-pair}, we obtain
\bq\label{eq:T-com-pair}
\lt\lal \pa_p^\beta\calT F, \pa_p^\beta\phi \rt\ral_{L^2_p} = - \lt\lal \pa_p^\beta F, \calA_+\pa_p^\beta\phi \rt\ral_{L^2_p} - \frac12 \sum_{0<\nu\le\beta} C_{\beta,\nu} \lt\lal \pa_p^\nu(p\eta_\hbar ) \cdot \pa_p^{\beta-\nu}F, \pa_p^\beta\phi \rt\ral_{L^2_p}.
\eq

We next estimate the two factors involving the test function. First, since $\phi=(I-P)\phi+P\phi$, we have
\bq\label{eq:Apl-te-b}
\lt\| \calA_+\pa_p^\beta\phi \rt\|_{L^2_p} \le \lt\| \calA_+\pa_p^\beta(I-P)\phi \rt\|_{L^2_p} + \lt\| \calA_+\pa_p^\beta P\phi \rt\|_{L^2_p} \le C \lt( |\pa_p^\beta(I-P)\phi|_D + \|\phi\|_{L^2_p} \rt).
\eq
Indeed, the first term is controlled by the definition of the $D$-norm. For the second term, we use that $P\phi$ is a finite linear combination of the basis functions of $\mcN$, whose $p$-derivatives are controlled by \eqref{eq:P-coeff-b}.

Similarly, since the $D$-norm controls the $L^2_p$ norm, we have
\bq\label{eq:L2-tes-b}
\|\pa_p^\beta\phi\|_{L^2_p} \le \|\pa_p^\beta(I-P)\phi\|_{L^2_p} + \|\pa_p^\beta P\phi\|_{L^2_p} \le C \lt( |\pa_p^\beta(I-P)\phi|_D + \|\phi\|_{L^2_p} \rt).
\eq

Moreover, by \eqref{eq:P-coeff-b},
\bq\label{eq:p-eta-deri-b}
\|\pa_p^\nu(p\eta_\hbar )\|_{L^\infty_p} \le C_\nu
\quad
\text{for every multi-index $\nu$ with }|\nu|\ge1.
\eq
Hence, using \eqref{eq:T-com-pair},
\eqref{eq:Apl-te-b}, \eqref{eq:L2-tes-b}, and
\eqref{eq:p-eta-deri-b}, we obtain
\begin{align}\label{eq:wei-f-com}
\begin{aligned}
\lt| \lt\lal \pa_p^\beta\calT F, \pa_p^\beta\phi \rt\ral_{L^2_p} \rt| 
&\le \|\pa_p^\beta F\|_{L^2_p} \lt\| \calA_+\pa_p^\beta\phi \rt\|_{L^2_p} + C_\beta \sum_{0<\nu\le\beta} \|\pa_p^{\beta-\nu}F\|_{L^2_p} \|\pa_p^\beta\phi\|_{L^2_p} \cr
&\le C  \lt( |\pa_p^\beta(I-P)\phi|_D + \|\phi\|_{L^2_p} \rt) \sum_{\nu\le\beta} \|\pa_p^\nu F\|_{L^2_p}.
\end{aligned}
\end{align}
 
\medskip
\noindent
{\it Estimate of $\Gamma_1$.} Using \eqref{eq:Apm-iden}, we rewrite $\Gamma_1$ as
\[
\Gamma_1 = \calT \lt( \mathscr c(g)\calA_-g \rt).
\]
Since $\mathscr c(g)$ depends only on $x$, Leibniz's rule in $x$ and \eqref{eq:wei-f-com} yield
\[
\lt| \lt\lal \pa_x^\alpha\pa_p^\beta\Gamma_1, \pa_x^\alpha\pa_p^\beta h \rt\ral_{L^2_{x,p}} \rt| \le C \lt( \|\pa_x^\alpha\pa_p^\beta(I-P)h\|_D + \|\pa_x^\alpha h\|_{L^2_{x,p}} \rt) \sum_{\substack{\gamma\le\alpha\\ \nu\le\beta}} \lt\| \lt| \mathscr c(\pa_x^{\alpha-\gamma}g) \rt| \, \lt\| \pa_p^\nu \calA_-(\pa_x^\gamma g) \rt\|_{L^2_p} \rt\|_{L^2_x}.
\]
By the definition of $\mathscr c(g)$, the boundedness of $\eta_\hbar $, and the
exponential decay of $\mu_\hbar $, we have
\bq\label{eq:c-mom-b}
\lt| \mathscr c(\pa_x^{\alpha-\gamma}g)\rt| = \lt| \int_{\mR^3} |p|^2\eta_\hbar  \pa_x^{\alpha-\gamma}g \sqrt{\mu_\hbar }\,\rdp \rt| \le C \|\pa_x^{\alpha-\gamma}g\|_{L^2_p}.
\eq
Moreover, by Leibniz's rule and \eqref{eq:P-coeff-b},
\bq\label{eq:Ami-D-b}
\lt\| \pa_p^\nu\calA_-\phi \rt\|_{L^2_p} \le C_\nu \sum_{\nu'\le\nu} |\pa_p^{\nu'}\phi|_D.
\eq
Consequently, we obtain
\[
\lt| \lt\lal \pa_x^\alpha\pa_p^\beta\Gamma_1, \pa_x^\alpha\pa_p^\beta h \rt\ral_{L^2_{x,p}} \rt| \le C \lt( \|\pa_x^\alpha\pa_p^\beta(I-P)h\|_D + \|\pa_x^\alpha h\|_{L^2_{x,p}} \rt)
\sum_{\substack{\gamma\le\alpha\\ \nu\le\beta}} \lt\| \|\pa_x^{\alpha-\gamma}g\|_{L^2_p} \sum_{\nu'\le\nu} |\pa_x^\gamma\pa_p^{\nu'}g|_D \rt\|_{L^2_x}.
\]
 
To estimate the last expression, we use a standard low--high splitting in the $x$ variable. For each $\gamma\le\alpha$ and $\nu'\le\nu\le\beta$, we distinguish two cases.

If $|\alpha-\gamma|\le2$, then the Sobolev embedding in the $x$ variable gives
\begin{align*}
\lt\| \|\pa_x^{\alpha-\gamma}g\|_{L^2_p}  |\pa_x^\gamma\pa_p^{\nu'}g|_D \rt\|_{L^2_x} &\le \|\pa_x^{\alpha-\gamma}g\|_{L^\infty_xL^2_p} \|\pa_x^\gamma\pa_p^{\nu'}g\|_D \\
&\le C
\sum_{1\le |\ell|\le2} \|\pa_x^{\alpha-\gamma+\ell}g\|_{L^2_{x,p}}   \|\pa_x^\gamma\pa_p^{\nu'}g\|_D \le C\|g\|_{H^s}\|g\|_{D,s}.
\end{align*}
Here we used $|\alpha-\gamma|+2\le4\le s$.

If $|\alpha-\gamma|\ge3$, then we place the second factor in $L^\infty_x$. Again by the Sobolev embedding in the $x$ variable, we obtain
\begin{align*}
\lt\| \|\pa_x^{\alpha-\gamma}g\|_{L^2_p} |\pa_x^\gamma\pa_p^{\nu'}g|_D \rt\|_{L^2_x} &\le \|\pa_x^{\alpha-\gamma}g\|_{L^2_{x,p}} \|\pa_x^\gamma\pa_p^{\nu'}g\|_{L^\infty_xD} \\ 
&\le C \|\pa_x^{\alpha-\gamma}g\|_{L^2_{x,p}} \sum_{1\le |\ell|\le2} \|\pa_x^{\gamma+\ell}\pa_p^{\nu'}g\|_D \\ 
&\le C\|g\|_{H^s}\|g\|_{D,s}.
\end{align*}
Here, $\|\cdot\|_{L^\infty_xD}$ denotes the $L^\infty$ norm in $x$ of the corresponding $D$-norm in $p$. Indeed, since $|\alpha-\gamma|\ge3$ and
$|\alpha|+|\beta|\le s$, we have $|\gamma|+|\nu'|+2 \le |\gamma|+|\beta|+2 \le |\alpha|+|\beta|-1 \le s-1$. Combining the two cases and summing over $\gamma$, $\nu$, and $\nu'$, we obtain
\bq\label{eq:gg-Ds}
\sum_{\substack{\gamma\le\alpha\\ \nu\le\beta}} \lt\| \|\pa_x^{\alpha-\gamma}g\|_{L^2_p} \sum_{\nu'\le\nu} |\pa_x^\gamma\pa_p^{\nu'}g|_D \rt\|_{L^2_x} \le C\|g\|_{H^s}\|g\|_{D,s}.
\eq
Therefore, we have
\[
\lt| \lt\lal \pa_x^\alpha\pa_p^\beta\Gamma_1, \pa_x^\alpha\pa_p^\beta h \rt\ral_{L^2_{x,p}} \rt| \le C\|g\|_{H^s} \lt( \|\pa_x^\alpha\pa_p^\beta(I-P)h\|_D + \|\pa_x^\alpha h\|_{L^2_{x,p}} \rt) \|g\|_{D,s}.
\]

\medskip
\noindent
{\it Estimate of $\Gamma_2$.} For convenience, set
\[
G(g) := \sqrt{\mu_\hbar } + \eta_\hbar g + \hbar\kappa\sqrt{\mu_\hbar }g^2.
\]
Then $\mu_\hbar +\eta_\hbar g\sqrt{\mu_\hbar }+\hbar\kappa\mu_\hbar g^2 = \sqrt{\mu_\hbar }G(g)$, and hence
\[
\Gamma_2 = - \calT\lt(N_uG(g)\rt).
\]
Applying \eqref{eq:wei-f-com}, we obtain
\bq\label{eq:Gamma2-s-g1-p}
\lt| \lt\lal \pa_x^\alpha\pa_p^\beta\Gamma_2, \pa_x^\alpha\pa_p^\beta h \rt\ral_{L^2_{x,p}} \rt| \le C \lt( \|\pa_x^\alpha\pa_p^\beta(I-P)h\|_D + \|\pa_x^\alpha h\|_{L^2_{x,p}} \rt)
\sum_{\nu\le\beta} \lt\| \pa_x^\alpha\pa_p^\nu \lt\{ N_uG(g) \rt\} \rt\|_{L^2_{x,p}}.
\eq
Since $N_u$ depends only on $x$, the $p$-derivatives fall only on $G(g)$. By Leibniz's rule,
\[
\sum_{\nu\le\beta} \lt\| \pa_x^\alpha\pa_p^\nu \lt\{ N_uG(g) \rt\} \rt\|_{L^2_{x,p}} \le C 
\sum_{\substack{\gamma\le\alpha\\ \nu\le\beta}} \lt\| \pa_x^\gamma N_u\, \pa_x^{\alpha-\gamma}\pa_p^\nu G(g) \rt\|_{L^2_{x,p}}.
\]

We decompose $G(g)=G_0+G_1(g)+G_2(g)$, where
\[
G_0:=\sqrt{\mu_\hbar },
\quad
G_1(g):=\eta_\hbar g,
\quad
G_2(g):=\hbar\kappa\sqrt{\mu_\hbar }g^2.
\]
Accordingly, we estimate the three contributions separately.

First, since $G_0$ is independent of $x$, we get  $\pa_x^{\alpha-\gamma}\pa_p^\nu G_0=0$ unless $\gamma=\alpha$. Thus, by Lemma \ref{lem:Nu-NT-est} and the exponential decay of
$\mu_\hbar $,
\bq\label{eq:Gamma2-G0-b}
\sum_{\substack{\gamma\le\alpha\\ \nu\le\beta}} \lt\| \pa_x^\gamma N_u\, \pa_x^{\alpha-\gamma}\pa_p^\nu G_0 \rt\|_{L^2_{x,p}} \le C \|\pa_x^\alpha N_u\|_{L^2_x} \sum_{\nu\le\beta} \|\pa_p^\nu\sqrt{\mu_\hbar }\|_{L^2_p}  \le C\|g\|_{H^s} \|\pa_x^\alpha g\|_{L^2_{x,p}} \le C\|g\|_{H^s}\|g\|_{D,s}.
\eq

We next consider the terms involving $G_1(g)$ and $G_2(g)$. Summing the estimates in Lemma \ref{lem:Nu-NT-est} over $|\gamma|\le s$, we first note that
\bq\label{eq:Nu-Hs-b}
\|N_u\|_{H^s_x}
\le
C\|g\|_{H^s}^2.
\eq

We will use the following product estimate. Let $H=H(x,p)$ satisfy $\|H\|_{D,s}<\infty$. Then, by using almost the same argument as in \eqref{eq:gg-Ds},
\bq\label{eq:Nu-H-pro-b}
\sum_{\substack{\gamma\le\alpha\\ \nu\le\beta}} \lt\| \pa_x^\gamma N_u\, \pa_x^{\alpha-\gamma}\pa_p^\nu H \rt\|_{L^2_{x,p}}  \le C \|N_u\|_{H^s_x} \|H\|_{D,s}.
\eq

We now apply \eqref{eq:Nu-H-pro-b} to $G_1(g)=\eta_\hbar g$. By Leibniz's rule,
\[
\pa_x^\gamma\pa_p^\nu G_1(g) = \sum_{\nu_1\le\nu} C_{\nu,\nu_1} \pa_p^{\nu_1}\eta_\hbar \, \pa_x^\gamma \pa_p^{\nu-\nu_1}g.
\]
Since $L^\infty$-norms of $\eta_\hbar $ and its derivatives are controlled by \eqref{eq:P-coeff-b}, we have
\bq\label{eq:G1-D-b}
\|G_1(g)\|_{D,s} \le C\|g\|_{D,s}.
\eq
Therefore, by  \eqref{eq:Nu-Hs-b}, \eqref{eq:Nu-H-pro-b}, and
\eqref{eq:G1-D-b}, we obtain
\bq\label{eq:Gamma2-G1-b}
 \sum_{\substack{\gamma\le\alpha\\ \nu\le\beta}} \lt\| \pa_x^\gamma N_u\, \pa_x^{\alpha-\gamma} \pa_p^\nu G_1(g) \rt\|_{L^2_{x,p}}  \le C \|N_u\|_{H^s_x} \|G_1(g)\|_{D,s}  \le C \|g\|_{H^s}^2 \|g\|_{D,s}.
\eq

We next consider $G_2(g)=\hbar\kappa\sqrt{\mu_\hbar }g^2$. By Leibniz's rule, the coefficient bound \eqref{eq:P-coeff-b}, the
Sobolev embedding, and the Moser product estimate in Lemma
\ref{Sobolev ineq}, we have
\bq\label{eq:G2-D-b1}
\|G_2(g)\|_{D,s} \le C \|g\|_{H^s} \|g\|_{D,s}.
\eq
Applying \eqref{eq:Nu-H-pro-b} with $H=G_2(g)$, and using
\eqref{eq:Nu-Hs-b} and \eqref{eq:G2-D-b1}, we obtain
\bq\label{eq:Gam2-G2-b}
\sum_{\substack{\gamma\le\alpha\\ \nu\le\beta}} \lt\| \pa_x^\gamma N_u\, \pa_x^{\alpha-\gamma} \pa_p^\nu G_2(g) \rt\|_{L^2_{x,p}} \le C \|N_u\|_{H^s_x} \|G_2(g)\|_{D,s}  \le C \|g\|_{H^s}^3 \|g\|_{D,s}.
\eq
Since $\|g\|_{H^s}\ll1$, the estimates \eqref{eq:Gamma2-G1-b} and \eqref{eq:Gam2-G2-b} imply that
\[
\sum_{\substack{\gamma\le\alpha\\ \nu\le\beta}} \lt\| \pa_x^\gamma N_u\, \pa_x^{\alpha-\gamma} \pa_p^\nu \lt\{ G_1(g)+G_2(g) \rt\} \rt\|_{L^2_{x,p}}  \le C \|g\|_{H^s} \|g\|_{D,s}.
\]
Combining this with \eqref{eq:Gamma2-G0-b}, we conclude that
\[
\sum_{\nu\le\beta} \lt\| \pa_x^\alpha\pa_p^\nu \lt\{ N_uG(g) \rt\} \rt\|_{L^2_{x,p}} \le C \|g\|_{H^s} \|g\|_{D,s}.
\]
This together with \eqref{eq:Gamma2-s-g1-p} yields
\[
\lt| \lt\lal \pa_x^\alpha\pa_p^\beta\Gamma_2, \pa_x^\alpha\pa_p^\beta h \rt\ral_{L^2_{x,p}} \rt|  \le C \|g\|_{H^s} \lt( \|\pa_x^\alpha\pa_p^\beta(I-P)h\|_D + \|\pa_x^\alpha h\|_{L^2_{x,p}} \rt) \|g\|_{D,s}.
\]

The remaining terms in $\Gamma(g)$ are estimated in the same way. They contain either moment factors as in $\Gamma_1$, nonlinear macroscopic remainders $N_u$ and $N_\Theta$ as in $\Gamma_2$, or lower-order products with exponentially decaying $p$-dependent coefficients. Combining all the estimates gives \eqref{eq:Gamma1-est}.
\end{proof}

Since $P$ is a finite-dimensional projection and the basis functions of $\mcN$ have exponential decay, we have
\[
\|g\|_{D,s} \le \lt( \sum_{|\alpha|+|\beta|\le s} \|\pa_x^\alpha\pa_p^\beta(I-P)g\|_D^2 \rt)^{\frac12} + C\|g\|_{H^s}.
\]
Thus the right-hand side of \eqref{eq:Gamma1-est} can be estimated in terms of the microscopic dissipation and the Sobolev energy.

The following estimate is more refined. It uses the interpolation strategy employed in \cite{GW12}: when $|\alpha|\le s-1$, the nonlinear term can be bounded by one additional spatial derivative, whereas at the top order $|\alpha|=s$ one only uses derivatives of the same order.

\begin{lemma}\label{Gamma 2}
Let $T_*>0$ and $s\ge4$ and suppose that $g\in \mcC([0,T_*];H^s(\R^3 \times \R^3))$ satisfies $\|g\|_{H^s}\ll1$. Then, for any $0<\delta\le1$, the following estimates hold. 

If $|\alpha|\le s-1$, then
\bq\label{eq:Gamma2-low}
\lt| \lt\lal \pa_x^\alpha\Gamma(g), \pa_x^\alpha g \rt\ral_{L^2_{x,p}} \rt| \le C\|g\|_{H^s} \lt( \frac1\delta \|\pa_x^\alpha(I-P)g\|_D^2 + \delta \|\pa_x^\alpha\nabla_x(I-P)g\|_D^2  +
\delta \|\pa_x^\alpha\nabla_x g\|_{L^2_{x,p}}^2 \rt).
\eq

If $|\alpha|=s$, then
\bq\label{eq:Gamma2-top}
\lt| \lt\lal \pa_x^\alpha\Gamma(g), \pa_x^\alpha g \rt\ral_{L^2_{x,p}} \rt| \le C\|g\|_{H^s} \lt( \frac1\delta \|\pa_x^\alpha(I-P)g\|_D^2 + \delta \|\pa_x^\alpha g\|_{L^2_{x,p}}^2 \rt).
\eq
\end{lemma}

\begin{proof}
Recall that the perturbative decomposition of the collision operator is given by
\[
\frac{1}{\sqrt{\mu_\hbar }} Q\lt(\mcF_\hbar +\sqrt{\mu_\hbar }g\rt) = L(g)+\Gamma(g).
\]
Since the collision operator preserves mass, momentum, and energy, we get $L(g)+\Gamma(g)\in\mcN^\perp$. On the other hand, since $L$ is self-adjoint and $\ker L=\mcN$
by Lemma \ref{lem:Prop-L}, we have $\lal Lg,\psi\ral_{L^2_p} = \lal g,L\psi\ral_{L^2_p} = 0$ for every $\psi\in\mcN$. Thus,  $L(g)\in\mcN^\perp$, and consequently $\Gamma(g)\in\mcN^\perp$.
This deduces
\bq\label{eq:Gam-mic-t}
\lt\lal \pa_x^\alpha\Gamma(g), \pa_x^\alpha g \rt\ral_{L^2_{x,p}} = \lt\lal \pa_x^\alpha\Gamma(g), \pa_x^\alpha(I-P)g \rt\ral_{L^2_{x,p}}.
\eq

We now split the right-hand side of \eqref{eq:Gam-mic-t} according to the individual terms in the nonlinear flux \eqref{eq:non-flux}. Due to the exponential decay of $\mu_\hbar $, each moment appearing in \eqref{eq:non-flux} is controlled by the corresponding $L^2_p$-norm. For example,
\bq\label{eq:mom-con-ex}
\lt| \intr (1+|p|^2)\eta_\hbar \sqrt{\mu_\hbar }\, \pa_x^\alpha g\,\rdp \rt| \le C\|\pa_x^\alpha g\|_{L^2_p}.
\eq
To avoid repetition, we estimate only the representative terms $\Gamma_1$ and $\Gamma_2$ defined in \eqref{nonlinear terms}. The remaining terms are handled in the same way or more directly.

\medskip
\noindent
{\it Estimate of $\Gamma_1$.} Using \eqref{eq:Apm-iden}, we get $\Gamma_1=\calT\lt(\mathscr c(g)\calA_-g\rt)$. By \eqref{eq:wei-f-pair}, \eqref{eq:Gam-mic-t}, and
Leibniz's rule in $x$, we obtain
\[
\lt\lal \pa_x^\alpha\Gamma_1, \pa_x^\alpha (I-P)g \rt\ral_{L^2_{x,p}} = - \sum_{\gamma\le\alpha} C_\gamma \intr \mathscr c(\pa_x^{\alpha-\gamma}g) \lt( \intr \calA_-(\pa_x^\gamma g) \cdot \calA_+\pa_x^\alpha(I-P)g \,\rdp \rt)\rdx.
\]
By \eqref{eq:c-mom-b} and the definition of the $D$-norm, we have $\lt| \mathscr c(\pa_x^{\alpha-\gamma}g) \rt| \le C\|\pa_x^{\alpha-\gamma}g\|_{L^2_p}$ and
\[
\|\calA_-(\pa_x^\gamma g)\|_{L^2_p} \le C|\pa_x^\gamma g|_D,
\quad
\|\calA_+\pa_x^\alpha(I-P)g\|_{L^2_p} \le C|\pa_x^\alpha(I-P)g|_D.
\]
This yields
\bq\label{eq:Gam1-ref-pre}
\lt| \lt\lal \pa_x^\alpha\Gamma_1, \pa_x^\alpha (I-P)g \rt\ral_{L^2_{x,p}} \rt| \le C \sum_{\gamma\le\alpha} \lt\| \|\pa_x^{\alpha-\gamma}g\|_{L^2_p} |\pa_x^\gamma g|_D \rt\|_{L^2_x} \|\pa_x^\alpha(I-P)g\|_D.
\eq

We now apply H\"older's inequality in the $x$ variable, \eqref{GN ineq 1}, and \eqref{GN ineq 2} to get
\bq\label{eq:Gam1-ref-inter}
\sum_{\gamma\le\alpha} \lt\| \|\pa_x^{\alpha-\gamma}g\|_{L^2_p} |\pa_x^\gamma g|_D \rt\|_{L^2_x} \le \sum_{\gamma\le\alpha} \|\pa_x^{\alpha-\gamma}g\|_{L^3_xL^2_p} \|\pa_x^\gamma g\|_{L^6_xD} 
\le
\begin{cases}
C\|g\|_{H^s}
\|\pa_x^\alpha\nabla_x g\|_D,
& |\alpha|\le s-1,\\[1mm]
C\|g\|_{H^s}
\|\pa_x^\alpha g\|_D,
& |\alpha|=s.
\end{cases}
\eq
Since $P$ is finite-dimensional and the basis functions of $\mcN$ have exponential decay, we deduce
\[
\|\pa_x^\alpha\nabla_x g\|_D \le \|\pa_x^\alpha\nabla_x(I-P)g\|_D + C\|\pa_x^\alpha\nabla_x g\|_{L^2_{x,p}},
\]
and
\[
\|\pa_x^\alpha g\|_D \le \|\pa_x^\alpha(I-P)g\|_D + C\|\pa_x^\alpha g\|_{L^2_{x,p}}.
\]
Combining these bounds with \eqref{eq:Gam1-ref-pre} and
\eqref{eq:Gam1-ref-inter}, and applying Young's inequality, we
obtain, for $|\alpha|\le s-1$,
\[
 \lt| \lt\lal \pa_x^\alpha\Gamma_1, \pa_x^\alpha (I-P)g \rt\ral_{L^2_{x,p}} \rt| \le C\|g\|_{H^s} \lt( \frac1\delta \|\pa_x^\alpha(I-P)g\|_D^2 + \delta \|\pa_x^\alpha\nabla_x(I-P)g\|_D^2 + \delta \|\pa_x^\alpha\nabla_x g\|_{L^2_{x,p}}^2 \rt),
\]
while for $|\alpha|=s$,
\[
 \lt| \lt\lal \pa_x^\alpha\Gamma_1, \pa_x^\alpha (I-P)g \rt\ral_{L^2_{x,p}} \rt|  \le C\|g\|_{H^s} \lt( \frac1\delta \|\pa_x^\alpha(I-P)g\|_D^2 + \delta \|\pa_x^\alpha g\|_{L^2_{x,p}}^2 \rt).
\]

\medskip
\noindent
{\it Estimate of $\Gamma_2$.}
Recall from the proof of Lemma \ref{Gamma 1} that $G(g) = G_0+G_1(g)+G_2(g)$, where $G_0=\sqrt{\mu_\hbar }$, $G_1(g)=\eta_\hbar g$, $G_2(g)=\hbar\kappa\sqrt{\mu_\hbar }g^2$. We also recall that $ \Gamma_2 = -\calT\lt(N_uG(g)\rt)$.
Using \eqref{eq:wei-f-pair} and Leibniz's rule in $x$, we obtain
\[
\lt\lal \pa_x^\alpha\Gamma_2, \pa_x^\alpha(I-P)g \rt\ral_{L^2_{x,p}}  = \sum_{\gamma\le\alpha} C_\gamma \inttr \pa_x^\gamma N_u\, \pa_x^{\alpha-\gamma}G(g) \cdot \calA_+\pa_x^\alpha(I-P)g \,\rdp\rdx 
\]
up to harmless signs. Consequently,
\bq\label{eq:Gam2-ref-pre}
 \lt| \lt\lal \pa_x^\alpha\Gamma_2, \pa_x^\alpha(I-P)g \rt\ral_{L^2_{x,p}} \rt|  \le C \|\pa_x^\alpha(I-P)g\|_D \sum_{\gamma\le\alpha} \lt\| \pa_x^\gamma N_u\, \pa_x^{\alpha-\gamma}G(g) \rt\|_{L^2_{x,p}}.
\eq

We now derive a refined estimate for the last sum. Since $G_0$ is independent of $x$, only the term with $\gamma=\alpha$ remains. Hence, by Lemma \ref{lem:Nu-NT-est} and the exponential decay of $\mu_\hbar $, we have
\[
 \sum_{\gamma\le\alpha} \lt\| \pa_x^\gamma N_u\, \pa_x^{\alpha-\gamma}G_0 \rt\|_{L^2_{x,p}}
 \le
\begin{cases}
C\|g\|_{H^s}
\|\pa_x^\alpha\nabla_x g\|_{L^2_{x,p}},
& |\alpha|\le s-1,\\[1mm]
C\|g\|_{H^s}
\|\pa_x^\alpha g\|_{L^2_{x,p}},
& |\alpha|=s.
\end{cases}
\]

For the terms involving $G_1(g)$ and $G_2(g)$, we use the same product decomposition as in the proof of Lemma \ref{Gamma 1}, together with the $x$-variable interpolation argument used above. Since $G_1(g)=\eta_\hbar g$ and $\eta_\hbar $ is independent of $x$ and bounded, Lemma \ref{lem:Nu-NT-est} gives
\[
\sum_{\gamma\le\alpha} \lt\| \pa_x^\gamma N_u\, \pa_x^{\alpha-\gamma}G_1(g) \rt\|_{L^2_{x,p}}
 \le
\begin{cases}
C\|g\|_{H^s}^2
\|\pa_x^\alpha\nabla_x g\|_{L^2_{x,p}},
& |\alpha|\le s-1,\\[1mm]
C\|g\|_{H^s}^2
\|\pa_x^\alpha g\|_{L^2_{x,p}},
& |\alpha|=s.
\end{cases}
\]
Similarly, since $G_2(g)=\hbar\kappa\sqrt{\mu_\hbar }g^2$, the exponential decay of $\sqrt{\mu_\hbar }$ and the Moser product estimate in Lemma \ref{Sobolev ineq} yield
\[
 \sum_{\gamma\le\alpha} \lt\| \pa_x^\gamma N_u\, \pa_x^{\alpha-\gamma}G_2(g) \rt\|_{L^2_{x,p}}
 \le
\begin{cases}
C\|g\|_{H^s}^3
\|\pa_x^\alpha\nabla_x g\|_{L^2_{x,p}},
& |\alpha|\le s-1,\\[1mm]
C\|g\|_{H^s}^3
\|\pa_x^\alpha g\|_{L^2_{x,p}},
& |\alpha|=s.
\end{cases}
\]

Since $\|g\|_{H^s}\ll1$, combining the above estimates, we obtain
\begin{align}\label{eq:Gam2-ref-pro-b}
\sum_{\gamma\le\alpha} \lt\| \pa_x^\gamma N_u\, \pa_x^{\alpha-\gamma}G(g) \rt\|_{L^2_{x,p}}
\le
\begin{cases}
C\|g\|_{H^s}
\|\pa_x^\alpha\nabla_x g\|_{L^2_{x,p}},
& |\alpha|\le s-1,\\[1mm]
C\|g\|_{H^s}
\|\pa_x^\alpha g\|_{L^2_{x,p}},
& |\alpha|=s.
\end{cases}
\end{align}
Combining \eqref{eq:Gam2-ref-pre} and \eqref{eq:Gam2-ref-pro-b}, and applying Young's inequality, we obtain the bounds \eqref{eq:Gamma2-low} and \eqref{eq:Gamma2-top} for $\Gamma_2$. The remaining terms in \eqref{eq:non-flux} are handled in the same manner, using the moment bound \eqref{eq:mom-con-ex}, the coefficient bound \eqref{eq:P-coeff-b}, Lemma \ref{lem:Nu-NT-est}, and the Sobolev--Moser estimates in Lemma \ref{Sobolev ineq}. This completes the proof.
\end{proof}

We next provide an estimate involving $p$-derivatives. This estimate will be used in the higher-order energy estimates.

\begin{lemma}\label{Gamma 3}
Let $T_*>0$ and $s\ge4$ and suppose that $g\in \mcC([0,T_*];H^s(\R^3 \times \R^3))$ satisfies $\|g\|_{H^s}\ll1$. Then, for $|\alpha|+|\beta|\le s$ with
$|\beta|\ge1$ and for any $0<\delta\le1$, we have
\begin{align}\label{eq:Gamma3-est}
\begin{aligned}
&\lt| \lt\lal \pa_x^\alpha\pa_p^\beta\Gamma(g), \pa_x^\alpha\pa_p^\beta(I-P)g \rt\ral_{L^2_{x,p}} \rt| \cr
&\quad \le C\|g\|_{H^s} \lt( \delta \sum_{1\le|\gamma|\le s} \|\pa_x^\gamma g\|_{L^2_{x,p}}^2 + \frac1\delta \sum_{|\gamma|\le s-|\beta|} \|\pa_x^\gamma\pa_p^\beta(I-P)g\|_D^2 + \frac1\delta \sum_{\substack{|\gamma|+|\nu|\le s\\|\nu|\le|\beta|-1}} \|\pa_x^\gamma\pa_p^\nu(I-P)g\|_D^2 \rt).
\end{aligned}
\end{align}
\end{lemma}

\begin{proof}
As in the proof of Lemma \ref{Gamma 1}, we split the nonlinear flux into its individual components and estimate only the contributions arising from the representative terms $\Gamma_1$ and $\Gamma_2$ defined in \eqref{nonlinear terms}. The remaining terms are handled in the same way or more directly.

\medskip
\noindent
{\it Estimate of $\Gamma_1$.}
Recall that $\Gamma_1=\calT\lt(\mathscr c(g)\calA_-g\rt)$. Applying \eqref{eq:wei-f-com} and Leibniz's rule in $x$, we
obtain
\begin{align*} 
&\lt| \lt\lal \pa_x^\alpha\pa_p^\beta\Gamma_1, \pa_x^\alpha\pa_p^\beta(I-P)g \rt\ral_{L^2_{x,p}} \rt|  \le C  \|\pa_x^\alpha\pa_p^\beta(I-P)g\|_D  \sum_{\substack{\gamma\le\alpha\\ \nu\le\beta}} \lt\| \lt| \mathscr c(\pa_x^{\alpha-\gamma}g) \rt| \lt\|
\pa_p^\nu \calA_-(\pa_x^\gamma g) \rt\|_{L^2_p} \rt\|_{L^2_x}.
\end{align*}
Since the $D$-norm controls the $L^2_p$-norm, we get $ \|\pa_x^\alpha(I-P)g\|_{L^2_{x,p}} \le C\|\pa_x^\alpha(I-P)g\|_D$. Moreover, by \eqref{eq:c-mom-b} and \eqref{eq:Ami-D-b},
\[
\lt| \mathscr c(\pa_x^{\alpha-\gamma}g) \rt| \le C\|\pa_x^{\alpha-\gamma}g\|_{L^2_p} \quad \text{and} \quad \lt\| \pa_p^\nu \calA_-(\pa_x^\gamma g) \rt\|_{L^2_p} \le C_\nu \sum_{\nu'\le\nu} |\pa_x^\gamma\pa_p^{\nu'}g|_D.
\]
Consequently,
\begin{align}\label{eq:Gam3-Gam1-2}
\begin{aligned}
&\lt| \lt\lal \pa_x^\alpha\pa_p^\beta\Gamma_1, \pa_x^\alpha\pa_p^\beta(I-P)g \rt\ral_{L^2_{x,p}} \rt|  \le C \|\pa_x^\alpha\pa_p^\beta(I-P)g\|_D  
\sum_{\substack{\gamma\le\alpha\\ \nu\le\beta}} \lt\| \|\pa_x^{\alpha-\gamma}g\|_{L^2_p} |\pa_x^\gamma\pa_p^{\nu}g|_D \rt\|_{L^2_x}.
\end{aligned}
\end{align}

We decompose $\pa_x^\gamma\pa_p^{\nu}g=\pa_x^\gamma\pa_p^{\nu}(I-P)g+\pa_x^\gamma\pa_p^{\nu}Pg$. Since $P$ is finite-dimensional and the basis functions of $\mcN$ have
exponential decay, we have
\bq\label{eq:P-D-b-Gam3}
|\pa_x^\gamma\pa_p^{\nu}Pg|_D \le C_{\nu} \|\pa_x^\gamma g\|_{L^2_p}.
\eq
Using \eqref{eq:P-D-b-Gam3}, the Sobolev embedding in the $x$ variable, and the same low--high splitting as in the proof of Lemma \ref{Gamma 1}, we obtain
\begin{align}\label{eq:Gam3-pro-b}
\begin{aligned}
& \sum_{\substack{\gamma\le\alpha\\ \nu\le\beta}} \lt\| \|\pa_x^{\alpha-\gamma}g\|_{L^2_p} |\pa_x^\gamma\pa_p^{\nu}g|_D \rt\|_{L^2_x} \cr
&\quad\le C\|g\|_{H^s} \lt( \sum_{1\le|\gamma|\le s} \|\pa_x^\gamma g\|_{L^2_{x,p}}  + \sum_{|\gamma|\le s-|\beta|} \|\pa_x^\gamma\pa_p^\beta(I-P)g\|_D + \sum_{\substack{|\gamma|+|\nu|\le s\\|\nu|\le|\beta|-1}} \|\pa_x^\gamma\pa_p^\nu(I-P)g\|_D \rt).
\end{aligned}
\end{align}
Here, the terms without a positive number of $x$-derivatives are included in the microscopic $D$-norm sums on the right-hand side.

Combining \eqref{eq:Gam3-Gam1-2} and \eqref{eq:Gam3-pro-b}, we obtain
\begin{align*} 
&\lt| \lt\lal \pa_x^\alpha\pa_p^\beta\Gamma_1, \pa_x^\alpha\pa_p^\beta(I-P)g \rt\ral_{L^2_{x,p}} \rt| \cr
&\quad\le C\|g\|_{H^s}  \|\pa_x^\alpha\pa_p^\beta(I-P)g\|_D  \cr
&\quad\quad \times \lt( \sum_{1\le|\gamma|\le s} \|\pa_x^\gamma g\|_{L^2_{x,p}} + \sum_{|\gamma|\le s-|\beta|} \|\pa_x^\gamma\pa_p^\beta(I-P)g\|_D + \sum_{\substack{|\gamma|+|\nu|\le s\\|\nu|\le|\beta|-1}} \|\pa_x^\gamma\pa_p^\nu(I-P)g\|_D \rt).
\end{align*}
Applying Young's inequality gives
\begin{align*} 
&\lt| \lt\lal \pa_x^\alpha\pa_p^\beta\Gamma_1, \pa_x^\alpha\pa_p^\beta(I-P)g \rt\ral_{L^2_{x,p}} \rt| \cr
&\quad\le C\|g\|_{H^s} \lt( \delta \sum_{1\le|\gamma|\le s} \|\pa_x^\gamma g\|_{L^2_{x,p}}^2 + \frac1\delta \sum_{|\gamma|\le s-|\beta|} \|\pa_x^\gamma\pa_p^\beta(I-P)g\|_D^2 + \frac1\delta \sum_{\substack{|\gamma|+|\nu|\le s\\|\nu|\le|\beta|-1}} \|\pa_x^\gamma\pa_p^\nu(I-P)g\|_D^2 \rt).
\end{align*}

\medskip
\noindent
{\it Estimate of $\Gamma_2$.}
Recall from the proof of Lemma \ref{Gamma 1} that $G(g) =G_0+G_1(g)+G_2(g)$, where $G_0=\sqrt{\mu_\hbar }$, $G_1(g)=\eta_\hbar g$, $G_2(g)=\hbar\kappa\sqrt{\mu_\hbar }g^2$, and $ \Gamma_2 = -\calT\lt(N_uG(g)\rt)$. Applying \eqref{eq:wei-f-com} and Leibniz's rule in $x$, we obtain
\begin{align}\label{eq:Gam3-Gam2-0}
\begin{aligned}
&\lt| \lt\lal \pa_x^\alpha\pa_p^\beta\Gamma_2, \pa_x^\alpha\pa_p^\beta(I-P)g \rt\ral_{L^2_{x,p}} \rt|  \le C  \|\pa_x^\alpha\pa_p^\beta(I-P)g\|_D   
\sum_{\substack{\gamma\le\alpha\\ \nu\le\beta}} \lt\| \pa_x^\gamma N_u\, \pa_x^{\alpha-\gamma}\pa_p^\nu G(g) \rt\|_{L^2_{x,p}}.
\end{aligned}
\end{align}
Since $N_u$ depends only on $x$, all $p$-derivatives fall on $G(g)$. Using the coefficient bound \eqref{eq:P-coeff-b}, Lemma \ref{lem:Nu-NT-est}, the Sobolev--Moser estimates in Lemma
\ref{Sobolev ineq}, and the same low--high splitting in the $x$ variable as above, we obtain
\begin{align}\label{eq:Gam3-Gam2-pro}
\begin{aligned}
&\sum_{\substack{\gamma\le\alpha\\ \nu\le\beta}} \lt\| \pa_x^\gamma N_u\, \pa_x^{\alpha-\gamma}\pa_p^\nu G(g) \rt\|_{L^2_{x,p}} \cr
&\quad\le C\|g\|_{H^s} \lt( \sum_{1\le|\gamma|\le s} \|\pa_x^\gamma g\|_{L^2_{x,p}} + \sum_{|\gamma|\le s-|\beta|} \|\pa_x^\gamma\pa_p^\beta(I-P)g\|_D + \sum_{\substack{|\gamma|+|\nu|\le s\\|\nu|\le|\beta|-1}} \|\pa_x^\gamma\pa_p^\nu(I-P)g\|_D \rt).
\end{aligned}
\end{align}
Here, the terms involving $G_0$ are controlled directly by Lemma \ref{lem:Nu-NT-est}. The terms involving $G_1(g)$ are estimated using the boundedness of $\eta_\hbar $ and its derivatives. Finally, the terms involving $G_2(g)$ are of higher order and are controlled by Leibniz's rule and the Moser product estimate.

Combining \eqref{eq:Gam3-Gam2-0} and \eqref{eq:Gam3-Gam2-pro}, and applying Young's inequality, we obtain
\begin{align*} 
&\lt| \lt\lal \pa_x^\alpha\pa_p^\beta\Gamma_2, \pa_x^\alpha\pa_p^\beta(I-P)g \rt\ral_{L^2_{x,p}} \rt| \cr
&\quad\le C\|g\|_{H^s} \lt( \delta \sum_{1\le|\gamma|\le s} \|\pa_x^\gamma g\|_{L^2_{x,p}}^2 + \frac1\delta \sum_{|\gamma|\le s-|\beta|} \|\pa_x^\gamma\pa_p^\beta(I-P)g\|_D^2 + \frac1\delta\sum_{\substack{|\gamma|+|\nu|\le s\\|\nu|\le|\beta|-1}} \|\pa_x^\gamma\pa_p^\nu(I-P)g\|_D^2 \rt).
\end{align*}

The remaining terms in $\Gamma(g)$ are estimated in the same manner. They contain either moment factors as in $\Gamma_1$, nonlinear macroscopic remainders $N_u$ and $N_\Theta$ as in $\Gamma_2$, or lower-order products with exponentially decaying $p$-dependent coefficients. Combining the preceding estimates gives \eqref{eq:Gamma3-est}.
\end{proof}

Finally, we show the Lipschitz continuity of the nonlinear operator in the perturbative regime.

\begin{lemma}\label{Gamma 1.5}
Let $T_* >0$ and $s\ge4$. Suppose that $g,h\in \mcC([0,T_*];H^s(\R^3 \times \R^3))$ satisfy $\|g\|_{D,s}+\|h\|_{D,s}<\infty$ and $\|g\|_{H^s}+\|h\|_{H^s}\ll1$. Then, for any test function $\phi$ with $\|\phi\|_D<\infty$, we have
\bq\label{eq:Gam-diff-est}
\lt| \lt\lal \Gamma(g)-\Gamma(h),\phi \rt\ral_{L^2_{x,p}} \rt|  \le C \lt[ \lt( \|g\|_{H^s}+\|h\|_{H^s} \rt) \|g-h\|_D + \lt( \|g\|_{D,s}+\|h\|_{D,s} \rt) \|g-h\|_{L^2_{x,p}} \rt] \|\phi\|_D.
\eq
\end{lemma}

\begin{proof}
Set  $w:=g-h$. Recall that
\[
\Gamma(g) = \frac{1}{\sqrt{\mu_\hbar }} \nabla_p\cdot\calR(g) .
\]
Using \eqref{eq:wei-f-pair}, we have
\bq\label{eq:Gam-diff-wei-pair}
\lt\lal \Gamma(g)-\Gamma(h),\phi \rt\ral_{L^2_{x,p}} = - \lt\lal \frac{\calR(g) -\mathcal R(h)}{\sqrt{\mu_\hbar }}, \calA_+\phi \rt\ral_{L^2_{x,p}}.
\eq
Since $\|\calA_+\phi\|_{L^2_{x,p}} \le C\|\phi\|_D$, it is enough to estimate
\[
\lt\| \frac{\calR(g) -\mathcal R(h)}{\sqrt{\mu_\hbar }} \rt\|_{L^2_{x,p}}.
\]

We first provide several elementary bounds. By the exponential decay of $\mu_\hbar $ and the boundedness of $\eta_\hbar $, we have, pointwise in $x$,
\bq\label{eq:mom-diff-b}
|\mathscr a(w)|+|\mathscr b(w)|+|\mathscr c(w)| \le C\|w\|_{L^2_p}.
\eq
Moreover, the explicit formulas for $N_u$ and $N_\Theta$, the smallness assumption, and the same moment and product estimates as in Lemma \ref{lem:Nu-NT-est} give
\bq\label{eq:Nu-NT-diff-b}
\|N_u(g)-N_u(h)\|_{L^2_x} + \|N_\Theta(g)-N_\Theta(h)\|_{L^2_x} \le C \lt( \|g\|_{H^s} + \|h\|_{H^s} \rt) \|w\|_{L^2_{x,p}}.
\eq
We also have
\bq\label{eq:Nu-NT-Hs-b-diff-pr}
\|N_u(g)\|_{H^s_x} + \|N_\Theta(g)\|_{H^s_x} \le C\|g\|_{H^s}^2,
\eq
and the analogous estimate with $g$ replaced by $h$.

We estimate a representative contribution involving a moment factor. By \eqref{eq:Apm-iden}, the first term in $\calR(g) /\sqrt{\mu_\hbar }$ is, up to a harmless constant, $\mathscr c(g)\calA_-g$. We write 
\[
\mathscr c(g)\calA_-g-\mathscr c(h)\calA_-h = \mathscr c(h)\calA_-w + \mathscr c(h)\calA_-g. 
\]
For the first term, using \eqref{eq:mom-diff-b} and the Sobolev embedding in the $x$ variable, we obtain
\[
\|\mathscr c(h)\calA_-w\|_{L^2_{x,p}} \le \|\mathscr c(h)\|_{L^\infty_x} \|\calA_-w\|_{L^2_{x,p}} \le C \|h\|_{H^s} \|w\|_D.
\]
For the second term, we place the moment factor in $L^2_x$ and the remaining factor in $L^\infty_xL^2_p$. Thus,
\[
\|\mathscr c(h)\calA_-g\|_{L^2_{x,p}} \le \|\mathscr c(h)\|_{L^2_x} \|\calA_-g\|_{L^\infty_xL^2_p} \le C \|w\|_{L^2_{x,p}} \|g\|_{D,s}.
\]
Consequently,
\[
 \|\mathscr c(g)\calA_-g-\mathscr c(h)\calA_-h\|_{L^2_{x,p}}  \le C \lt[ \lt( \|g\|_{H^s} + \|h\|_{H^s} \rt) \|w\|_D + \lt( \|g\|_{D,s} + \|h\|_{D,s} \rt) \|w\|_{L^2_{x,p}} \rt].
\]

We next consider a representative term involving the nonlinear macroscopic remainder $N_u$. Recall that $G(g) = \sqrt{\mu_\hbar } + \eta_\hbar g + \hbar\kappa\sqrt{\mu_\hbar }g^2$.
We decompose
\[
N_u(g)G(g)-N_u(h)G(h) = \lt\{ N_u(g)-N_u(h) \rt\} G(g) + N_u(h) \lt\{ G(g)-G(h) \rt\}.
\]
By \eqref{eq:Nu-NT-diff-b}, the coefficient bounds, and the Sobolev--Moser estimates, we have
\[
\lt\| \lt\{ N_u(g)-N_u(h) \rt\} G(g) \rt\|_{L^2_{x,p}}  \le C \lt( \|g\|_{H^s} + \|h\|_{H^s} \rt) \|w\|_{L^2_{x,p}}.
\]
Moreover, $G(g)-G(h) = \eta_\hbar w + \hbar\kappa\sqrt{\mu_\hbar }(g+h)w$. Using \eqref{eq:Nu-NT-Hs-b-diff-pr}, the Sobolev embedding, and
the Moser product estimate, we obtain
\[
\lt\| N_u(h) \lt\{ G(g)-G(h) \rt\} \rt\|_{L^2_{x,p}}  \le C \|h\|_{H^s}^2(1 + \|g\|_{H^s} + \|h\|_{H^s}) \|w\|_D   \le C \|h\|_{H^s}^2 \|w\|_D.
\]
Hence, we have
\[
 \|N_u(g)G(g)-N_u(h)G(h)\|_{L^2_{x,p}}  \le C \lt[ \lt( \|g\|_{H^s} + \|h\|_{H^s} \rt) \|w\|_D + \lt( \|g\|_{D,s} + \|h\|_{D,s} \rt) \|w\|_{L^2_{x,p}} \rt].
\]
Here we used the fact that the $D,s$-norm controls the $H^s$-norm.

The term involving $N_\Theta$ is estimated in the same way. The remaining contributions to $\calR(g) -\mathcal R(h)$ consist of products of moment factors, nonlinear macroscopic remainders, and coefficients with exponential decay in $p$. This yields  
\bq\label{eq:R-diff-b}
 \lt\| \frac{\calR(g) -\mathcal R(h)}{\sqrt{\mu_\hbar }} \rt\|_{L^2_{x,p}}  \le C \lt[ \lt( \|g\|_{H^s} + \|h\|_{H^s} \rt) \|w\|_D + \lt( \|g\|_{D,s} + \|h\|_{D,s} \rt)
 \|w\|_{L^2_{x,p}} \rt].
\eq
Combining \eqref{eq:Gam-diff-wei-pair} and \eqref{eq:R-diff-b} gives \eqref{eq:Gam-diff-est}. This completes the proof.
\end{proof}

%
%
%
%
%
%
%
%
%
%

\subsection{Local well-posedness}

We now establish the local well-posedness of the perturbative equation \eqref{eq:linearized}. The proof is based on the semigroup formulation for the linearized operator, an iterative approximation scheme, and the nonlinear estimates obtained in Section \ref{ssec:non-est}. We present the main construction here and defer the detailed uniform estimates, the Cauchy argument, and the preservation of the physically admissible range to Appendix \ref{app:lwp}.

\begin{theorem}\label{LE}
Let $s\ge4$. There exist sufficiently small constants $\delta_{\rm in}>0$ and $T_*>0$ such that, for any initial datum $g_0\in H^s(\mR^3\times\mR^3)$ satisfying $\|g_0\|_{H^s}\le\delta_{\rm in}$ and $f_0:=\mcF_\hbar +g_0\sqrt{\mu_\hbar }\ge0$ for almost every $(x,p)\in\mR^3\times\mR^3$, the Cauchy problem \eqref{eq:linearized} with initial datum $g_0$ admits a unique solution $g\in \mcC([0,T_*];H^s(\R^3 \times \R^3))$. Moreover,
\bq\label{eq:local-b-main}
\sup_{0\le t\le T_*}\|g(t)\|_{H^s}^2 + \sum_{|\alpha|+|\beta|\le s} \int_0^{T_*} \|\pa_x^\alpha\pa_p^\beta(I-P)g(\tau)\|_D^2\,\rdta \le \delta_{\rm in}.
\eq
The corresponding distribution $f(t,x,p) := \mcF_\hbar (p)+g(t,x,p)\sqrt{\mu_\hbar (p)}$ satisfies $f(t,x,p)\ge0$ for almost every $(x,p)\in\mR^3\times\mR^3$ and all $0\le t\le T_*$. In the fermionic case $\kappa=-1$, if in addition $f_0(x,p)\le\frac1{\hbar}$ for almost every $(x,p)\in\mR^3\times\mR^3$, then we have
\[
    f(t,x,p)\le\frac1{\hbar}
\]
for almost every $(x,p)\in\mR^3\times\mR^3$ and all
$0\le t\le T_*$.
\end{theorem}

\begin{proof}
We write the perturbative equation in the abstract form
\[
\pa_t g=Bg+\Gamma(g), \quad B:=L-p\cdot\nabla_x.
\]
We first recall the linear theory associated with $B$. Let $B_0$ be the operator defined on $\calC_c^\infty(\R^3\times\R^3)$ by $B_0h:=Lh-p\cdot\nabla_xh$. We use the same notation $B$ for its maximal realization in $L^2_{x,p}$:
\[
D(B):=\lt\{h\in L^2 : B_0h\in L^2\ \text{in the sense of distributions}\rt\}, \quad Bh:=B_0h.
\]
A standard cutoff and approximation argument shows that $B$ is densely defined and closed. Moreover, its adjoint is the maximal realization of the formal adjoint $B_0^\sharp h:=Lh+p\cdot\nabla_xh$. Since $L$ is self-adjoint and nonpositive and $p\cdot\nabla_x$ is skew-symmetric in $L^2$, an integration-by-parts argument, together with the same cutoff approximation, shows that both $B$ and $B^*$ are dissipative. Hence, by the Lumer--Phillips theorem, $B$ generates a contraction $\calC_0$-semigroup $\{e^{tB}\}_{t\ge0}$ on $L^2$; see \cite[Corollary 1.4.4]{Paz83}. We also refer to \cite[Lemma 3.4]{AGGMMS12} for a related maximal-realization argument for a kinetic Fokker--Planck operator.

In addition, by the microscopic coercivity of $L$ established in Lemma \ref{lem:Prop-L}, we have
\bq\label{eq:B-diss-local}
    \lal Bh,h\ral_{L^2_{x,p}} = \lal Lh,h\ral_{L^2_{x,p}} \le  -\lambda_0\|(I-P)h\|_D^2
\eq
for every sufficiently regular function $h$. This estimate will be used both in the uniform estimates and in the Cauchy argument.

We construct a solution by iteration. Set $g^0=0$ and define $\{g^n\}_{n\ge0}$ recursively by
\bq\label{eq:local-iter}
\pa_t g^{n+1} = Bg^{n+1}+\Gamma(g^n), \quad g^{n+1}|_{t=0} = g_0.
\eq
Equivalently, Duhamel's formula gives
\[
    g^{n+1}(t) = e^{tB}g_0 + \int_0^t e^{(t-\tau)B}\Gamma(g^n(\tau))\,\rdta.
\]
The scheme is first carried out for smooth approximations of the initial datum and the nonlinear source term. The estimates below are uniform with respect to the approximation parameters, which are suppressed in the notation. The corresponding technical details are given in Appendix \ref{app:lwp}.

For fixed $T_*>0$ and $\delta_{\rm in}>0$, define
\[
\mcI(s,T_*; \delta_{\rm in}) := \lt\{ g\in \mcC([0,T_*];H^s) : \sup_{0\le t\le T_*} \|g(t)\|_{H^s}^2 + \sum_{|\alpha|+|\beta|\le s} \int_0^{T_*} \|\pa_x^\alpha\pa_p^\beta(I-P)g(\tau)\|_D^2\,\rdta \le \delta_{\rm in} \rt\}.
\]

We first establish uniform estimates for the approximate sequence. The proof is divided into two steps. For the purely spatial derivatives, we apply $\pa_x^\alpha$ to \eqref{eq:local-iter}, take the
$L^2_{x,p}$-inner product with $\pa_x^\alpha g^{n+1}$, and use the skew-symmetry of the transport operator. This gives
\[
\frac12\frac{\rd}{\rdt} \|\pa_x^\alpha g^{n+1}\|_{L^2_{x,p}}^2  = \lt\lal L(\pa_x^\alpha g^{n+1}), \pa_x^\alpha g^{n+1} \rt\ral_{L^2_{x,p}} + \lt\lal \pa_x^\alpha\Gamma(g^n), \pa_x^\alpha g^{n+1} \rt\ral_{L^2_{x,p}}.
\]
By \eqref{eq:B-diss-local}, the linear term provides the microscopic dissipation
\[
\lt\lal L(\pa_x^\alpha g^{n+1}), \pa_x^\alpha g^{n+1} \rt\ral_{L^2_{x,p}} \le -\lambda_0 \|\pa_x^\alpha(I-P)g^{n+1}\|_D^2.
\]
The nonlinear term is estimated by Lemma \ref{Gamma 1}.

We next estimate mixed derivatives. For $ |\alpha|+|\beta|\le s$, $|\beta|\ge1$, we apply $\pa_x^\alpha\pa_p^\beta$ to
\eqref{eq:local-iter} and test the resulting equation against $\pa_x^\alpha\pa_p^\beta g^{n+1}$. The commutator $[\pa_p^\beta,p\cdot\nabla_x]g^{n+1}$ contains only strictly lower $p$-derivatives. Similarly, the $p$-derivative coercivity estimate for $L$ gives the same-order microscopic dissipation up to lower $p$-order terms. These contributions are handled by a triangular induction over $|\beta|$. Whenever the full $D$-norm of an approximate solution appears in the nonlinear estimates, we use the finite-dimensionality of $P$ and the exponential decay of the basis functions of $\mcN$ to write
\[
\|\pa_x^\alpha\pa_p^\beta g^n\|_D \le \|\pa_x^\alpha\pa_p^\beta(I-P)g^n\|_D + C_\beta \|\pa_x^\alpha g^n\|_{L^2_{x,p}}.
\]
Thus, on a finite time interval, the full $D$-norm is controlled by the microscopic dissipation and the Sobolev energy.

Combining the spatial and mixed derivative estimates with suitable weights, we obtain a closed energy inequality. As shown in Lemma \ref{lem:local-uni-b}, if $ \delta_{\rm in}>0$ and $T_*>0$ are sufficiently small, then
\bq\label{eq:local-uni-iter-b}
    g^n\in\mcI(s,T_*; \delta_{\rm in})    \quad    \text{for every }n\ge0.
\eq

We next prove the convergence of the approximate sequence. Set
\[
    w^{n+1} := g^{n+1}-g^n.
\]
For $n\ge1$, subtracting the equations for $g^{n+1}$ and $g^n$ gives
\[
\pa_t w^{n+1} = Bw^{n+1} + \Gamma(g^n)-\Gamma(g^{n-1}), \quad w^{n+1}|_{t=0}=0.
\]
Taking the $L^2_{x,p}$-inner product with $w^{n+1}$ and using \eqref{eq:B-diss-local}, we obtain
\[
\frac12\frac{\rd}{\rdt} \|w^{n+1}(t)\|_{L^2_{x,p}}^2 + \lambda_0 \|(I-P)w^{n+1}(t)\|_D^2  \le \lt| \lt\lal \Gamma(g^n)-\Gamma(g^{n-1}), w^{n+1} \rt\ral_{L^2_{x,p}} \rt|.
\]
Applying the difference estimate in Lemma \ref{Gamma 1.5}, splitting the $D$-norm into its microscopic and finite-dimensional macroscopic parts, and using \eqref{eq:local-uni-iter-b}, we obtain a contraction estimate after integration in time. More precisely, Lemma \ref{lem:local-cau} shows that $\{g^n\}_{n\ge0}$ is a Cauchy sequence in $\mcC([0,T_*];L^2(\R^3 \times \R^3))$. Moreover, $\{(I-P)g^n\}_{n\ge0}$ is a Cauchy sequence in the corresponding microscopic dissipation space.

By the uniform $H^s$ bound and interpolation, we obtain $g^n\to g$ in $\mcC([0,T_*];H^{s-1}(\R^3 \times \R^3)).$ Moreover, after passing to a subsequence if necessary, $g^n \rightharpoonup^\ast g$ in $L^\infty(0,T_*;H^s(\R^3 \times \R^3))$ and $\pa_x^\alpha\pa_p^\beta(I-P)g^n \rightharpoonup \pa_x^\alpha\pa_p^\beta(I-P)g$ weakly in the corresponding $L^2(0,T_*;D)$ spaces for $|\alpha|+|\beta|\le s$.

Passing to the limit in the approximate equations, we obtain a solution $g$ of \eqref{eq:linearized}. By weak lower semicontinuity,
\[
\sup_{0\le t\le T_*}\|g(t)\|_{H^s}^2 + \sum_{|\alpha|+|\beta|\le s} \int_0^{T_*} \|\pa_x^\alpha\pa_p^\beta(I-P)g(\tau)\|_D^2\,\rdta \le  \delta_{\rm in}.
\]
The weak continuity of $g$ in $H^s$ follows from the equation and the uniform $H^s$ bound. Applying the energy estimate on an arbitrary interval $[t_1,t_2]\subset[0,T_*]$ and then letting $t_2\to t_1$, we obtain the continuity of $\|g(t)\|_{H^s}$. Combining weak continuity with norm continuity yields $g\in C([0,T_*];H^s)$.

It remains to prove uniqueness. Let $g$ and $h$ be two solutions with the same initial datum and satisfying \eqref{eq:local-b-main}. Set $w:=g-h$. Then,
\[
    \pa_t w = Bw+\Gamma(g)-\Gamma(h), \quad w|_{t=0}=0.
\]
Taking the $L^2_{x,p}$-inner product with $w$, using \eqref{eq:B-diss-local}, and applying Lemma \ref{Gamma 1.5}, we obtain an inequality of the form
\[
\frac{\rd}{\rdt} \|w(t)\|_{L^2_{x,p}}^2 + \lambda \|(I-P)w(t)\|_D^2 \le A(t)\|w(t)\|_{L^2_{x,p}}^2,
\]
where $A\in L^1(0,T_*)$. Since $w(0)=0$, Gr\"onwall's inequality implies $w = 0$ on $[0,T_*]$. Therefore, the solution is unique.

Finally, the preservation of the nonnegativity of $f = \mcF_\hbar +g\sqrt{\mu_\hbar }$, and the upper bound $f\le\frac1{\hbar}$ in the fermionic case are proved in Lemma \ref{lem:local-pos-pauli}. This completes the proof.
\end{proof}

%
%
%
%
%
%
%
%
%
%
\section{Global energy estimates} \label{sec:global-energy}

In this section, we derive the global a priori estimates for the perturbative equation
\bq\label{eq:pert-glo}
    \pa_t g+p\cdot\nabla_x g=Lg+\Gamma(g).
\eq
Throughout this section, $g$ denotes a smooth solution to \eqref{eq:pert-glo} on a time interval $[0,T]$, satisfying the bootstrap smallness condition
\bq\label{eq:boots}
    \sup_{0\le t\le T}\|g(t)\|_{H^s}\le  \delta_{\rm in},
\eq
where $s\ge4$ and $0< \delta_{\rm in}\ll1$. The goal is to prove an energy inequality which controls both the microscopic dissipation and the spatial gradients of the macroscopic variables. More precisely, we shall construct an energy functional $\mcE_s(g)$, equivalent to $\|g\|_{H^s}^2$, and prove
\[
    \frac{\rd}{\rdt}\mcE_s(g)(t)+\lambda\mcD_s(g)(t)\le0,
\]
where
\[
    \mcD_s(g) := \sum_{|\alpha|+|\beta|\le s} \|\pa_x^\alpha\pa_p^\beta(I-P)g\|_D^2 + \|\nabla_x(a,b,c)\|_{H^{s-1}_x}^2.
\]
This estimate will be used in Section \ref{sec:gwp-lt} to extend the local solution globally in time and to derive its large-time behavior.

%
%
%
%
%
%
%
%
%
%
\subsection{Macro--micro decomposition and fluid variables}

We begin with the macro--micro decomposition. Recall that $g=Pg+(I-P)g$, where $P$ is the projection onto the null space $\mcN=\ker L$. We write
\bq\label{eq:Pg-abc}
    Pg = \{a+p\cdot b+|p|^2c\}\sqrt{\mu_\hbar },
\eq
where the scalar functions $a=a(t,x)$, $c=c(t,x)$, and the vector field $b=b(t,x)$ are defined by
\[
a := \frac{ \intr g\sqrt{\mu_\hbar }\,\rdp} { \intr\mu_\hbar \,\rdp} - \ph \frac{ \intr(|p|^2-\ph)g\sqrt{\mu_\hbar }\,\rdp} { \intr(|p|^2-\ph)^2\mu_\hbar \,\rdp}, 
\quad 
b := \frac{ \intr p\,g\sqrt{\mu_\hbar }\,\rdp} { \frac13\intr|p|^2\mu_\hbar \,\rdp}, 
\quad 
b_i := \frac{\ \intr p_i g\sqrt{\mu_\hbar }\,\rdp} {  \frac13\intr|p|^2\mu_\hbar \,\rdp},
\]
and
\[
c := \frac{ \intr(|p|^2-\ph)g\sqrt{\mu_\hbar }\,\rdp}{ \intr(|p|^2-\ph)^2\mu_\hbar \,\rdp}.
\]
We also introduce the constants
\[
    C_a:=\intr\mu_\hbar \,\rdp,
    \quad
    C_b:=\frac13\intr|p|^2\mu_\hbar \,\rdp,
    \quad
    C_c:=\intr|p|^4\mu_\hbar \,\rdp.
\]
These constants depend only on the global equilibrium and remain fixed throughout the paper.

By the radial symmetry of $\mu_\hbar $, the macroscopic moments are given by
\[
    \lal g,\sqrt{\mu_\hbar }\ral_{L^2_p} =  C_a a+3C_b c, 
    \quad 
    \lal g,p\sqrt{\mu_\hbar }\ral_{L^2_p} = C_b b,
\quad 
    \lal g,|p|^2\sqrt{\mu_\hbar }\ral_{L^2_p} =   3C_b a+C_c c.
\]
Moreover, for the second-order moments, we have
\begin{align*} 
\lal g,p_ip_j\sqrt{\mu_\hbar }\ral_{L^2_p} &= \lal Pg,p_ip_j\sqrt{\mu_\hbar }\ral_{L^2_p} + \lal (I-P)g,p_ip_j\sqrt{\mu_\hbar }\ral_{L^2_p} \\
&= \lt( C_ba+\frac{C_c}{3}c \rt)\delta_{ij} + \lal (I-P)g,p_ip_j\sqrt{\mu_\hbar }\ral_{L^2_p},
\end{align*}
and
\begin{align*} 
\lal g,p_i|p|^2\sqrt{\mu_\hbar }\ral_{L^2_p} &= \lal Pg,p_i|p|^2\sqrt{\mu_\hbar }\ral_{L^2_p} + \lal (I-P)g,p_i|p|^2\sqrt{\mu_\hbar }\ral_{L^2_p} \\
&= \frac{C_c}{3}b_i + \lal (I-P)g,p_i|p|^2\sqrt{\mu_\hbar }\ral_{L^2_p}.
\end{align*}

We next derive the balance laws for the macroscopic coefficients. Taking the $L^2_p$-inner products of \eqref{eq:pert-glo} with $\sqrt{\mu_\hbar }$, $p\sqrt{\mu_\hbar }$, and $|p|^2\sqrt{\mu_\hbar }$, respectively, and using the conservation properties of $L$ and $\Gamma$, we obtain
\begin{align}\label{eq:abc-bal-raw}
\begin{aligned}
    \pa_t(C_a a+3C_b c) + C_b\nabla_x\cdot b &=0, \cr
    C_b\pa_t b + \nabla_x \lt( C_ba+\frac{C_c}{3}c \rt) &= - \lt\lal p\cdot\nabla_x(I-P)g,p\sqrt{\mu_\hbar } \rt\ral_{L^2_p}, \cr
    \pa_t(3C_b a+C_c c) + \frac{C_c}{3}\nabla_x\cdot b &= - \lt\lal p\cdot\nabla_x(I-P)g,|p|^2\sqrt{\mu_\hbar } \rt\ral_{L^2_p}.
\end{aligned}
\end{align}
Let $D_0:=C_aC_c-9C_b^2$. Then, by the strict Cauchy--Schwarz inequality,
\bq\label{eq:Cabc}
    (3C_b)^2 = \lt( \intr|p|^2\mu_\hbar \,\rdp \rt)^2 <  \lt( \intr\mu_\hbar \,\rdp \rt) \lt( \intr|p|^4\mu_\hbar \,\rdp \rt) = C_aC_c.
\eq
The inequality is strict because the functions $1$ and $|p|^2$ are linearly independent in $L^2(\mR^3;\mu_\hbar (p)\,\rdp)$.
 
Solving \eqref{eq:abc-bal-raw} for the time derivatives, we obtain
\begin{align}\label{eq:abc_t}
\begin{aligned}
    \pa_t a &= \frac{3C_b}{D_0} \lt\lal p\cdot\nabla_x(I-P)g, |p|^2\sqrt{\mu_\hbar } \rt\ral_{L^2_p}, \cr
    \pa_t b &= - \nabla_x \lt( a+\frac{C_c}{3C_b}c \rt) - \frac1{C_b} \lt\lal p\cdot\nabla_x(I-P)g, p\sqrt{\mu_\hbar } \rt\ral_{L^2_p}, \cr
    \pa_t c &= - \frac13\nabla_x\cdot b - \frac{C_a}{D_0} \lt\lal p\cdot\nabla_x(I-P)g, |p|^2\sqrt{\mu_\hbar } \rt\ral_{L^2_p}.
\end{aligned}
\end{align}

We finally provide the equivalence between the macroscopic coefficients and the projection $Pg$. This equivalence is used repeatedly to replace macroscopic moments by the variables $a,b,c$, and conversely.

\begin{lemma} 
There exist constants $C_0,C_1>0$, depending only on $\mu_\hbar $, such that
\bq\label{eq:Pg-equiv}
    C_0\|(a,b,c)\|_{L^2_x}^2 \le \|Pg\|_{L^2_{x,p}}^2 \le C_1\|(a,b,c)\|_{L^2_x}^2.
\eq
Moreover, for every multi-index $\alpha$,
\bq\label{eq:Pg-equiv-der}
    C_0\|\pa_x^\alpha(a,b,c)\|_{L^2_x}^2 \le \|\pa_x^\alpha Pg\|_{L^2_{x,p}}^2 \le C_1\|\pa_x^\alpha(a,b,c)\|_{L^2_x}^2.
\eq
\end{lemma}

\begin{proof}
Using \eqref{eq:Pg-abc}, the radial symmetry of $\mu_\hbar $, and the definitions of $C_a,C_b,C_c$, we compute
\[
\|Pg\|_{L^2_{x,p}}^2 = C_a\|a\|_{L^2_x}^2 + C_b\|b\|_{L^2_x}^2 + C_c\|c\|_{L^2_x}^2 + 6C_b\intr ac\,\rdx.
\]
The upper bound follows immediately from Young's inequality. For the lower bound, for any $\delta_0>0$,
\[
\|Pg\|_{L^2_{x,p}}^2 \ge (C_a-3\delta_0 C_b)\|a\|_{L^2_x}^2 + C_b\|b\|_{L^2_x}^2 + \lt( C_c-\frac{3C_b}{\delta_0} \rt)\|c\|_{L^2_x}^2.
\]
By \eqref{eq:Cabc}, we may choose $\delta_0>0$ such that
\[
    \frac{3C_b}{C_c}<\delta_0<\frac{C_a}{3C_b}.
\]
For this choice, all three coefficients are positive. This proves \eqref{eq:Pg-equiv}. Applying the same argument to $\pa_x^\alpha Pg$ gives \eqref{eq:Pg-equiv-der}.
\end{proof}

%
%
%
%
%
%
%
%
%
%
\subsection{Microscopic energy estimates} \label{ssec:micro-energy}

We next derive the microscopic energy estimates. These estimates control the microscopic component $(I-P)g$ through the coercivity of the linearized operator $L$. The spatial gradients of the macroscopic variables $\nabla_x(a,b,c)$ will be recovered in Section \ref{ssec:mac-diss}.

\begin{lemma}\label{priori 1}
Let $s\ge4$, and let $g$ satisfy the bootstrap condition \eqref{eq:boots}. Then there exist positive constants $C_1,C>0$ such that
\[
\frac{\rd}{\rdt}\|g\|_{L^2_p(H^s_x)}^2 + C_1 \sum_{|\alpha|\le s} \|\pa_x^\alpha(I-P)g\|_D^2 \le C \delta_{\rm in} \|\nabla_x(a,b,c)\|_{H^{s-1}_x}^2.
\]
\end{lemma}

\begin{proof}
Applying $\pa_x^\alpha$, $|\alpha|\le s$, to \eqref{eq:pert-glo}, multiplying by $\pa_x^\alpha g$, and integrating over $\R^3 \times \R^3$, we obtain
\bq\label{eq:spat-ener-0}
\frac12\frac{\rd}{\rdt}\|\pa_x^\alpha g\|_{L^2_{x,p}}^2 = \lt\lal L(\pa_x^\alpha g), \pa_x^\alpha g \rt\ral_{L^2_{x,p}} + \lt\lal \pa_x^\alpha\Gamma(g), \pa_x^\alpha g \rt\ral_{L^2_{x,p}}  =: I_1+I_2.
\eq
By the coercivity estimate \eqref{eq:L-coercivity},
\bq\label{eq:I1-spat}
I_1 \le -\lambda_0 \|\pa_x^\alpha(I-P)g\|_D^2.
\eq

We estimate the nonlinear term by Lemma \ref{Gamma 2}. If $|\alpha|=s$, then
\[
|I_2| \le C\|g\|_{H^s} \lt( \frac1\eta \|\pa_x^\alpha(I-P)g\|_D^2 + \eta \|\pa_x^\alpha g\|_{L^2_{x,p}}^2 \rt)
\]
for any $\eta>0$. If $|\alpha|\le s-1$, then
\[
|I_2| \le C\|g\|_{H^s} \lt( \frac1\eta \|\pa_x^\alpha(I-P)g\|_D^2 + \eta \|\pa_x^\alpha\nabla_x(I-P)g\|_D^2 + \eta \|\pa_x^\alpha\nabla_x g\|_{L^2_{x,p}}^2 \rt).
\]
We choose $\eta>0$ fixed and sufficiently small, and then choose $ \delta_{\rm in}>0$ sufficiently small so that the terms involving $\|\pa_x^\alpha(I-P)g\|_D^2$ and $\|\pa_x^\alpha\nabla_x(I-P)g\|_D^2$ are absorbed by the linear dissipation after summing over $|\alpha|\le s$.

It remains to treat the terms involving spatial derivatives of $g$. By the decomposition $g=Pg+(I-P)g$ and the equivalence \eqref{eq:Pg-equiv-der}, for $|\alpha|\le s-1$,
\[
\|\pa_x^\alpha\nabla_x g\|_{L^2_{x,p}}^2 \le C\|\pa_x^\alpha\nabla_x(I-P)g\|_{L^2_{x,p}}^2 + C\|\pa_x^\alpha\nabla_x(a,b,c)\|_{L^2_x}^2.
\]
For the top-order term, we similarly use
\[
\|\pa_x^\alpha g\|_{L^2_{x,p}}^2 \le C \|\pa_x^\alpha(I-P)g\|_{L^2_{x,p}}^2 + C \|\pa_x^\alpha(a,b,c)\|_{L^2_x}^2, \quad |\alpha|=s.
\]
The microscopic term is controlled by the corresponding $D$-norm, whereas the macroscopic term is included in $\|\nabla_x(a,b,c)\|_{H^{s-1}_x}^2$. Therefore, summing \eqref{eq:spat-ener-0} over $|\alpha|\le s$, using \eqref{eq:I1-spat}, and choosing $ \delta_{\rm in}>0$ sufficiently small, we obtain
\[
\frac{\rd}{\rdt}\|g\|_{L^2_p(H^s_x)}^2 + C_1 \sum_{|\alpha|\le s} \|\pa_x^\alpha(I-P)g\|_D^2 \le C \delta_{\rm in}\|\nabla_x(a,b,c)\|_{H^{s-1}_x}^2.
\]
This completes the proof.
\end{proof}

\begin{lemma}\label{priori 2}
Let $s\ge4$, and let $g$ satisfy the bootstrap condition \eqref{eq:boots}. Then, for each fixed multi-index
$\beta$ with $1\le|\beta|\le s$, there exist constants $C_2,C>0$ such that
\begin{align}\label{eq:sum2}
\begin{aligned}
&\frac{\rd}{\rdt} \sum_{|\alpha|\le s-|\beta|} \|\pa_x^\alpha\pa_p^\beta(I-P)g\|_{L^2_{x,p}}^2 + C_2 \sum_{|\alpha|\le s-|\beta|} \|\pa_x^\alpha\pa_p^\beta(I-P)g\|_D^2 \cr
&\quad\le C \sum_{\substack{|\gamma|+|\nu|\le s\\|\nu|\le|\beta|-1}} \|\pa_x^\gamma\pa_p^\nu(I-P)g\|_D^2 + C \|\nabla_x(a,b,c)\|_{H^{s-|\beta|}_x}^2 + C \delta_{\rm in} \|\nabla_x(a,b,c)\|_{H^{s-1}_x}^2.
\end{aligned}
\end{align}
\end{lemma}

\begin{proof}
We first derive the microscopic equation. Applying $I-P$ to \eqref{eq:pert-glo}, using $L(Pg)=0$ and $\Gamma(g)\in\mcN^\perp$, we obtain
\bq\label{eq:mic-equa}
\pa_t(I-P)g + (I-P)(p\cdot\nabla_xg) = L(I-P)g+\Gamma(g).
\eq
Since $g=Pg+(I-P)g$, we may rewrite \eqref{eq:mic-equa} as
\bq\label{eq:mic-equa-exp}
\pa_t(I-P)g + p\cdot\nabla_x(I-P)g = L(I-P)g + \Gamma(g) + P\{p\cdot\nabla_x(I-P)g\} - (I-P)(p\cdot\nabla_xPg).
\eq

Fix a multi-index $\beta$ satisfying $1\le|\beta|\le s$. For each $|\alpha|\le s-|\beta|$, we apply $\pa_x^\alpha\pa_p^\beta$ to \eqref{eq:mic-equa-exp}, take the $L^2_{x,p}$-inner product with $\pa_x^\alpha\pa_p^\beta(I-P)g$, and obtain
\[
\frac12 \frac{\rd}{\rdt} \|\pa_x^\alpha\pa_p^\beta(I-P)g\|_{L^2_{x,p}}^2 = \sum_{j=1}^5J_j,
\]
where
\begin{align*}
J_1 &:= - \lt\lal \pa_x^\alpha\pa_p^\beta \{p\cdot\nabla_x(I-P)g\}, \pa_x^\alpha\pa_p^\beta(I-P)g \rt\ral_{L^2_{x,p}}, \cr
J_2 &:= \lt\lal \pa_x^\alpha\pa_p^\beta L(I-P)g, \pa_x^\alpha\pa_p^\beta(I-P)g \rt\ral_{L^2_{x,p}}, \cr
J_3 &:= \lt\lal \pa_x^\alpha\pa_p^\beta\Gamma(g), \pa_x^\alpha\pa_p^\beta(I-P)g \rt\ral_{L^2_{x,p}}, \cr
J_4 &:= \lt\lal \pa_x^\alpha\pa_p^\beta P\{p\cdot\nabla_x(I-P)g\}, \pa_x^\alpha\pa_p^\beta(I-P)g \rt\ral_{L^2_{x,p}}, \cr
J_5 &:= - \lt\lal \pa_x^\alpha\pa_p^\beta (I-P)(p\cdot\nabla_xPg), \pa_x^\alpha\pa_p^\beta(I-P)g \rt\ral_{L^2_{x,p}}.
\end{align*}

\medskip
\noindent
{\it Estimate of $J_1$.} Since $p\cdot\nabla_x$ is linear in $p$, we have
\[
[\pa_p^\beta,p\cdot\nabla_x]h = \sum_{\substack{1\le j\le3\\\beta_j\ge1}} \beta_j \pa_{x_j}\pa_p^{\beta-\mfe_j}h,
\]
where $\mfe_j$ denotes the $j$-th standard basis vector in $\mR^3$. The principal transport term vanishes by skew-symmetry. Thus,
\begin{align}\label{eq:J1-est-s5}
\begin{aligned}
|J_1| &\le C_\beta \sum_{\substack{1\le j\le3\\\beta_j\ge1}} \|\pa_x^{\alpha+\mfe_j} \pa_p^{\beta-\mfe_j}(I-P)g\|_{L^2_{x,p}} \|\pa_x^\alpha\pa_p^\beta(I-P)g\|_{L^2_{x,p}} \cr
&\le \eta \|\pa_x^\alpha\pa_p^\beta(I-P)g\|_D^2 + C_{\eta,\beta} \sum_{\substack{1\le j\le3\\\beta_j\ge1}} \|\pa_x^{\alpha+\mfe_j} \pa_p^{\beta-\mfe_j}(I-P)g\|_D^2
\end{aligned}
\end{align}
for any $\eta>0$. Since  $|\alpha+\mfe_j|+|\beta-\mfe_j| = |\alpha|+|\beta| \le s$ and $|\beta-\mfe_j| = |\beta|-1$, the last terms in \eqref{eq:J1-est-s5} are lower $p$-order contributions.

\medskip
\noindent
{\it Estimate of $J_2$.}
By the $p$-derivative coercivity estimate \eqref{eq:L-p-der-coer}, applied to $h := \pa_x^\alpha(I-P)g$, we have
\[
J_2 \le -\lambda_1 \|\pa_x^\alpha\pa_p^\beta(I-P)g\|_D^2 + C_\beta \sum_{|\nu|\le|\beta|-1} \|\pa_x^\alpha\pa_p^\nu(I-P)g\|_D^2.
\]

\medskip
\noindent
{\it Estimate of $J_3$.}
By Lemma \ref{Gamma 3}, for any $0<\eta\le1$, we have
\[
|J_3| \le C\|g\|_{H^s} \lt( \eta \sum_{1\le|\gamma|\le s} \|\pa_x^\gamma g\|_{L^2_{x,p}}^2 + \frac1\eta \sum_{|\gamma|\le s-|\beta|} \|\pa_x^\gamma\pa_p^\beta(I-P)g\|_D^2 + \frac1\eta \sum_{\substack{|\gamma|+|\nu|\le s\\|\nu|\le|\beta|-1}} \|\pa_x^\gamma\pa_p^\nu(I-P)g\|_D^2 \rt).
\]
Using the bootstrap assumption $\|g\|_{H^s} \le  \delta_{\rm in}$, the decomposition $g=Pg+(I-P)g$, and the equivalence \eqref{eq:Pg-equiv-der}, we have
\[
\sum_{1\le|\gamma|\le s} \|\pa_x^\gamma g\|_{L^2_{x,p}}^2 \le C \|\nabla_x(a,b,c)\|_{H^{s-1}_x}^2 + C \sum_{1\le|\gamma|\le s} \|\pa_x^\gamma(I-P)g\|_{L^2_{x,p}}^2.
\]
The last microscopic $L^2_{x,p}$-terms are controlled by the corresponding $D$-norms. Since they have no $p$-derivatives, they are included in
\[
\sum_{\substack{|\gamma|+|\nu|\le s\\|\nu|\le|\beta|-1}} \|\pa_x^\gamma\pa_p^\nu(I-P)g\|_D^2.
\]
Consequently,
\[
|J_3| \le \frac{C \delta_{\rm in}}{\eta} \sum_{|\gamma|\le s-|\beta|} \|\pa_x^\gamma\pa_p^\beta(I-P)g\|_D^2 + C_{\eta}\delta_{\rm in} \sum_{\substack{|\gamma|+|\nu|\le s\\|\nu|\le|\beta|-1}} \|\pa_x^\gamma\pa_p^\nu(I-P)g\|_D^2 + C\eta\delta_{\rm in} \|\nabla_x(a,b,c)\|_{H^{s-1}_x}^2.
\]

\medskip
\noindent
{\it Estimate of $J_4$.} By definition of $P$, we get
\[
P\{p\cdot\nabla_x(I-P)g\} = \sum_{i=1}^5 \lt\lal p\cdot\nabla_x(I-P)g,e_i \rt\ral_{L^2_p} e_i,
\]
where $\{e_i\}_{i=1}^5$ is a fixed basis of $\mcN$. By the exponential decay of the basis functions and their derivatives, we obtain
\bq\label{eq:J4-est-s5}
|J_4| \le C_\beta \|\pa_x^\alpha\nabla_x(I-P)g\|_{L^2_{x,p}} \|\pa_x^\alpha\pa_p^\beta(I-P)g\|_{L^2_{x,p}} 
\le \eta \|\pa_x^\alpha\pa_p^\beta(I-P)g\|_D^2 + C_{\eta,\beta} \|\pa_x^\alpha\nabla_x(I-P)g\|_D^2.
\eq
Since $|\alpha|+1 \le s-|\beta|+1 \le s$ and $|\beta|\ge1$, the last term in \eqref{eq:J4-est-s5} is also included in
\[
\sum_{\substack{|\gamma|+|\nu|\le s\\|\nu|\le|\beta|-1}} \|\pa_x^\gamma\pa_p^\nu(I-P)g\|_D^2.
\]

\medskip
\noindent
{\it Estimate of $J_5$.}
Using the explicit form \eqref{eq:Pg-abc}, we see that $\pa_x^\alpha\pa_p^\beta (I-P)(p\cdot\nabla_xPg)$ is a finite linear combination of fixed smooth exponentially decaying
functions of $p$, with coefficients involving $\pa_x^\alpha\nabla_x(a,b,c)$. Hence, we have
\[
|J_5| \le C_\beta \|\pa_x^\alpha\nabla_x(a,b,c)\|_{L^2_x} \|\pa_x^\alpha\pa_p^\beta(I-P)g\|_{L^2_{x,p}}  \le \eta \|\pa_x^\alpha\pa_p^\beta(I-P)g\|_D^2 + C_{\eta,\beta} \|\pa_x^\alpha\nabla_x(a,b,c)\|_{L^2_x}^2.
\]

Combining all of the above estimates, summing over $|\alpha|\le s-|\beta|$, and choosing $\eta>0$ sufficiently small, we obtain
\begin{align*} 
&\frac{\rd}{\rdt} \sum_{|\alpha|\le s-|\beta|} \|\pa_x^\alpha\pa_p^\beta(I-P)g\|_{L^2_{x,p}}^2 + \frac{\lambda_1}{2} \sum_{|\alpha|\le s-|\beta|} \|\pa_x^\alpha\pa_p^\beta(I-P)g\|_D^2 \cr
&\quad\le \frac{C\delta_{\rm in}}{\eta} \sum_{|\alpha|\le s-|\beta|} \|\pa_x^\alpha\pa_p^\beta(I-P)g\|_D^2 + C \sum_{\substack{|\gamma|+|\nu|\le s\\|\nu|\le|\beta|-1}} \|\pa_x^\gamma\pa_p^\nu(I-P)g\|_D^2 \cr
&\quad \quad + C \|\nabla_x(a,b,c)\|_{H^{s-|\beta|}_x}^2 + C\delta_{\rm in} \|\nabla_x(a,b,c)\|_{H^{s-1}_x}^2.
\end{align*}
Here the lower $p$-order terms arising from the estimate of $J_3$ have been included in the second term on the right-hand side, since $\eta>0$ is fixed and $\delta_{\rm in}>0$ is sufficiently small.

After fixing $\eta>0$, we choose $\delta_{\rm in}>0$ sufficiently small so that
\[
\frac{C\delta_{\rm in}}{\eta} \le \frac{\lambda_1}{4}.
\]
The current $p$-order term is then absorbed into the left-hand side. The remaining microscopic terms involve strictly fewer $p$-derivatives and will be absorbed later by a triangular weighted summation over $|\beta|$. Thus, we obtain \eqref{eq:sum2}. This completes the proof.
 \end{proof}

%
%
%
%
%
%
%
%
%
%
\subsection{Macroscopic dissipation estimate} \label{ssec:mac-diss}

We now derive a dissipation estimate for the macroscopic variables $a$, $b$, and $c$. The microscopic estimates in Section \ref{ssec:micro-energy} control the dissipative component $(I-P)g$, but they do not directly provide control of $\nabla_x(a,b,c)$. The purpose of this subsection is to obtain the missing macroscopic dissipation from the equations satisfied by the coefficients of $Pg$.

Substituting the decomposition $g=Pg+(I-P)g$ into \eqref{eq:pert-glo}, we obtain
\begin{align}\label{eq:expand-s5}
    (\pa_t+p\cdot\nabla_x)Pg = \ell+\mathfrak h,
\end{align}
where
\[
\ell := \lt\{ -\pa_t - p\cdot\nabla_x + L \rt\}(I-P)g,
\quad
\mathfrak h := \Gamma(g).
\]
Using \eqref{eq:Pg-abc}, we compute
\begin{align}\label{eq:mac-l-exp-s5}
\begin{aligned}
(\pa_t+p\cdot\nabla_x)Pg &= \Bigg[ \pa_ta + \sum_{i=1}^3p_i(\pa_tb_i+\pa_{x_i}a) + \sum_{i=1}^3p_i^2(\pa_tc+\pa_{x_i}b_i) \cr
&\quad \quad + \sum_{1\le i<j\le3} p_ip_j(\pa_{x_i}b_j+\pa_{x_j}b_i) + \sum_{i=1}^3p_i|p|^2\pa_{x_i}c \Bigg]\sqrt{\mu_\hbar }.
\end{aligned}
\end{align}
 
To extract the macroscopic relations from \eqref{eq:expand-s5}, we introduce an appropriate family of coefficient functionals. Let
\[
\mathscr V := \operatorname{span} \lt\{ 1,\ p_i,\ p_i^2,\ p_ip_j,\ p_i|p|^2 : 1\le i\le3,\ 1\le i<j\le3 \rt\} \sqrt{\mu_\hbar }.
\]
The functions in the preceding family are linearly independent. We denote them collectively by $\{\psi_A\}_{A\in\mathscr A}$, where $\mathscr A$ is the corresponding finite index set. Let $\{\psi_A^*\}_{A\in\mathscr A} \subset \mathscr V$ be the dual basis with respect to the $L^2_p$-inner product, so that $\lt\lal \psi_A,\psi_B^* \rt\ral_{L^2_p} = \delta_{AB}$.
For a sufficiently regular function $F=F(t,x,p)$, we define its coefficient functionals by $\mathfrak C_A[F](t,x) := \lt\lal F(t,x,\cdot), \psi_A^* \rt\ral_{L^2_p}$. Since the functions $\psi_A^*$ are fixed smooth exponentially decaying functions of $p$, each coefficient functional is a finite linear combination of moments against smooth exponentially decaying weights.

The left-hand side of \eqref{eq:expand-s5} belongs to $\mathscr V$. Applying the coefficient functionals $\mathfrak C_A$ to both sides of \eqref{eq:expand-s5}, we obtain the relations between the macroscopic variables and the corresponding moments of $\ell$ and $\mathfrak h$.

For convenience, we write
\[
l_c,\quad l_i,\quad l_{ij},\quad l_{bi},\quad l_a
\]
for the coefficient functionals of $\ell$ associated with
\[
p_i|p|^2\sqrt{\mu_\hbar },
\quad
p_i^2\sqrt{\mu_\hbar },
\quad
p_ip_j\sqrt{\mu_\hbar },
\quad
p_i\sqrt{\mu_\hbar },
\quad
\sqrt{\mu_\hbar },
\]
respectively. Here $l_c = (l_{c,1},l_{c,2},l_{c,3})$. We define
\[
h_c,\quad h_i,\quad h_{ij},\quad h_{bi},\quad h_a
\]
analogously from $\mathfrak h$.

\begin{lemma}\label{lem:lh-s5}
The macroscopic variables satisfy
\begin{align}\label{eq:mac-rel-s5}
\begin{aligned}
    \nabla_x c &= l_c+h_c, \cr
    \pa_t c+\pa_{x_i}b_i
    &=  l_i+h_i, \quad i=1,2,3, \cr
    \pa_{x_i}b_j+\pa_{x_j}b_i &= l_{ij}+h_{ij},  \quad 1\le i<j\le3, \cr
    \pa_tb_i+\pa_{x_i}a &= l_{bi}+h_{bi},    \quad i=1,2,3, \cr
    \pa_t a &= l_a+h_a.
\end{aligned}
\end{align}
Moreover, every $l$-coefficient is a finite linear combination of moments of the form
\bq\label{eq:form-I-s5}
    \intr \lt\{ -\pa_t - p\cdot\nabla_x + L \rt\} (I-P)g\,\xi(p)\,\rdp,
\eq
where $\xi$ ranges over a fixed finite family of smooth exponentially decaying functions of $p$. Similarly, every $h$-coefficient is a finite linear combination of moments of the form
\bq\label{eq:form-h-s5}
    \intr  \Gamma(g)\xi(p)\,\rdp.
\eq
\end{lemma}

\begin{proof}
By \eqref{eq:mac-l-exp-s5}, the left-hand side of \eqref{eq:expand-s5} is an element of $\mathscr V$. Applying the dual coefficient functionals $\mathfrak C_A$ to both sides of \eqref{eq:expand-s5}, we obtain
\[
\mathfrak C_A \lt[ (\pa_t+p\cdot\nabla_x)Pg \rt] = \mathfrak C_A[\ell] + \mathfrak C_A[\mathfrak h]
\]
for every $A\in\mathscr A$. Comparing this identity with the explicit expansion \eqref{eq:mac-l-exp-s5} gives \eqref{eq:mac-rel-s5}. Since every dual basis function $\psi_A^*$ is a fixed smooth exponentially decaying function of $p$, the definitions of $\ell$ and $\mathfrak h$ imply \eqref{eq:form-I-s5} and \eqref{eq:form-h-s5}. This completes the proof.
\end{proof}

We define
\[
\|\tilde h\|_{L^2_x}^2 := \|h_c\|_{L^2_x}^2 + \sum_{i=1}^3 \|h_i\|_{L^2_x}^2 + \sum_{1\le i<j\le3} \|h_{ij}\|_{L^2_x}^2 + \sum_{i=1}^3 \|h_{bi}\|_{L^2_x}^2 + \|h_a\|_{L^2_x}^2.
\]
For derivatives, $\|\pa_x^\alpha\tilde h\|_{L^2_x}$ is defined analogously.

The next lemma provides the macroscopic dissipation estimate. It collects the estimates for the spatial derivatives of the macroscopic variables $a$, $b$, and $c$, which are obtained from the coefficient relations \eqref{eq:mac-rel-s5} and suitable temporal interaction terms. The proof keeps the key cancellations in the body of the paper, while the purely algebraic coefficient comparison may be expanded in Appendix \ref{app:mac-diss} if desired.

\begin{lemma} \label{lem:mac-reco-s5}
Let $s\ge4$, and let $g$ satisfy the bootstrap condition \eqref{eq:boots}. Then, for every multi-index $\alpha$ with $|\alpha|\le s-1$, there exists an interaction functional $\calI_\alpha(t)$ such that 
\bq\label{eq:mac-rec-s5}
\|\pa_x^\alpha\nabla_x(a,b,c)\|_{L^2_x}^2 \le -\frac{\rd}{\rdt}\calI_\alpha(t) + C \|\pa_x^\alpha\nabla_x(I-P)g\|_{L^2_{x,p}}^2 + C \|\pa_x^\alpha(I-P)g\|_{L^2_{x,p}}^2 + C \|\pa_x^\alpha\tilde h\|_{L^2_x}^2.
\eq
More precisely, $\calI_\alpha$ is given by
\begin{align}\label{eq:Ialpha-s5}
\begin{aligned}
\calI_\alpha(t)
&:= \sum_{m=1}^{N_a} c_{a,m} \inttr \pa_x^\alpha(I-P)g\, \xi_{a,m}(p)\cdot \pa_x^\alpha\nabla_xa \,\rdp\rdx  \cr
&\quad + \sum_{m=1}^{N_b} c_{b,m} \inttr \pa_x^\alpha(I-P)g\, \xi_{b,m}(p): \pa_x^\alpha\nabla_xb \,\rdp\rdx  \cr
&\quad + \sum_{m=1}^{N_c} c_{c,m} \inttr \pa_x^\alpha(I-P)g\, \xi_{c,m}(p)\cdot \pa_x^\alpha\nabla_xc \,\rdp\rdx  \cr
&\quad + C_* \intr \pa_x^\alpha b\cdot \pa_x^\alpha\nabla_xa \,\rdx.
\end{aligned}
\end{align}
Here $N_a,N_b,N_c\in\mN$, the constants  $c_{a,m}$, $c_{b,m}$, $c_{c,m}$, $C_*$ are fixed, and $\xi_{a,m}$, $\xi_{b,m}$, $\xi_{c,m}$ are fixed smooth exponentially decaying functions of $p$ determined by the coefficient functionals introduced above. All these quantities depend only on the equilibrium $\mu_\hbar $.
\end{lemma}

\begin{proof}
The detailed coefficient-level calculations are given in Appendix \ref{app:mac-diss}. We briefly explain the structure of the argument.

We first derive an estimate for $\nabla_x b$. The relations $\pa_tc+\pa_{x_i}b_i = l_i+h_i$ and $\pa_{x_i}b_j+\pa_{x_j}b_i = l_{ij}+h_{ij}$ yield an elliptic equation for each component $b_i$. As shown in
Lemma \ref{app:blh-section5}, we obtain
\bq\label{eq:b-reco-p-s5}
\|\pa_x^\alpha\nabla_xb\|_{L^2_x}^2 \le \calL_{\alpha,b} + C \|\pa_x^\alpha\tilde h\|_{L^2_x}^2,
\eq
where $\calL_{\alpha,b}$ denotes a finite sum of terms involving the $l_i$- and $l_{ij}$-coefficients paired with $\pa_x^\alpha\nabla_xb$.

The estimate of $\nabla_xc$ is more direct. From $\nabla_x c = l_c+h_c$, Lemma \ref{app:clh-section5} gives
\bq\label{eq:c-reco-p-s5}
\|\pa_x^\alpha\nabla_xc\|_{L^2_x}^2 \le \calL_{\alpha,c} + C \|\pa_x^\alpha\tilde h\|_{L^2_x}^2,
\eq
where $\calL_{\alpha,c}$ is the corresponding contribution involving $l_c$.

Finally, using $\pa_tb+\nabla_xa = l_b+h_b$, Lemma \ref{app:alh-section5} yields
\begin{align}\label{eq:a-reco-p-s5}
\begin{aligned}
\|\pa_x^\alpha\nabla_xa\|_{L^2_x}^2 &\le -\frac{\rd}{\rdt} \intr \pa_x^\alpha b\cdot \pa_x^\alpha\nabla_xa \,\rdx + C\calL_{\alpha,a} + C \|\pa_x^\alpha\nabla_xb\|_{L^2_x}^2 \cr
&\quad + C \|\pa_x^\alpha\nabla_x(I-P)g\|_{L^2_{x,p}}^2 + C \|\pa_x^\alpha\tilde h\|_{L^2_x}^2.
\end{aligned}
\end{align}

We now combine \eqref{eq:b-reco-p-s5}, \eqref{eq:c-reco-p-s5}, and \eqref{eq:a-reco-p-s5}. Multiplying the estimates for $b$ and $c$ by sufficiently large fixed constants, we absorb the term $C \|\pa_x^\alpha\nabla_xb\|_{L^2_x}^2$ appearing on the right-hand side of \eqref{eq:a-reco-p-s5}. Hence,
\bq\label{eq:mac-reco-be-ell}
\|\pa_x^\alpha\nabla_x(a,b,c)\|_{L^2_x}^2 \le -\frac{\rd}{\rdt} \lt\{ C_* \intr \pa_x^\alpha b\cdot \pa_x^\alpha\nabla_xa \,\rdx \rt\}
+ C\calL_\alpha + C \|\pa_x^\alpha\nabla_x(I-P)g\|_{L^2_{x,p}}^2 + C \|\pa_x^\alpha\tilde h\|_{L^2_x}^2,
\eq
where $C_*>0$ is a fixed constant and $\calL_\alpha := \calL_{\alpha,a} + \calL_{\alpha,b} + \calL_{\alpha,c}$.

By Lemma \ref{app:ell-terms-s5}, the total contribution of the $l$-coefficients satisfies
\bq\label{eq:ell-bound-main}
\calL_\alpha \le -\frac{\rd}{\rdt}\calJ_\alpha(t) + C \|\pa_x^\alpha\nabla_x(I-P)g\|_{L^2_{x,p}}^2 + C\|\pa_x^\alpha(I-P)g\|_{L^2_{x,p}}^2 + \eta \|\pa_x^\alpha\nabla_x(a,b,c)\|_{L^2_x}^2,
\eq
where
\begin{align*}
\calJ_\alpha(t) &= \sum_{m=1}^{N_a} c_{a,m} \inttr \pa_x^\alpha(I-P)g\, \xi_{a,m}(p)\cdot \pa_x^\alpha\nabla_xa \,\rdp\rdx  \cr
&\quad + \sum_{m=1}^{N_b} c_{b,m} \inttr \pa_x^\alpha(I-P)g\, \xi_{b,m}(p): \pa_x^\alpha\nabla_xb \,\rdp\rdx  \cr
&\quad + \sum_{m=1}^{N_c} c_{c,m} \inttr \pa_x^\alpha(I-P)g\, \xi_{c,m}(p)\cdot \pa_x^\alpha\nabla_xc \,\rdp\rdx .
\end{align*}
Substituting \eqref{eq:ell-bound-main} into \eqref{eq:mac-reco-be-ell} and choosing $\eta>0$ sufficiently small, we absorb the final term into the left-hand side.
Defining
\[
\calI_\alpha(t) := \calJ_\alpha(t) + C_* \intr \pa_x^\alpha b\cdot \pa_x^\alpha\nabla_xa \,\rdx,
\]
we obtain \eqref{eq:mac-rec-s5}. This completes the proof.
\end{proof}

We now estimate the nonlinear coefficient $\tilde h$.

\begin{lemma} 
Let $s\ge4$, and suppose that $g$ satisfies \eqref{eq:boots}. Then, for every multi-index $\alpha$
satisfying $|\alpha|\le s-1$, we have
\bq\label{eq:h-ineq-section5}
\|\pa_x^\alpha\tilde h\|_{L^2_x}^2 \le C\delta_{\rm in}^2 \lt( \|\pa_x^\alpha\nabla_x(a,b,c)\|_{L^2_x}^2 + \|\pa_x^\alpha\nabla_x(I-P)g\|_{L^2_{x,p}}^2 \rt).
\eq
\end{lemma}

\begin{proof}
By \eqref{eq:form-h-s5}, each coefficient of $\tilde h$ is a finite linear combination of moments of the form $\intr \Gamma(g)\xi(p)\,\rdp$, where $\xi$ belongs to a fixed finite family of smooth exponentially decaying functions of $p$.

Applying the same weighted flux estimate and low--high splitting in the $x$ variable as in the proof of Lemma \ref{Gamma 2}, we obtain, for $|\alpha|\le s-1$,
\[
\lt\| \intr  \pa_x^\alpha\Gamma(g)\xi(p)\,\rdp \rt\|_{L^2_x} \le C_\xi \|g\|_{H^s} \|\pa_x^\alpha\nabla_x g\|_{L^2_{x,p}}.
\]
Since the family of coefficient weights is finite, the corresponding constants can be absorbed into a uniform constant. Hence, $\|\pa_x^\alpha\tilde h\|_{L^2_x} \le C \|g\|_{H^s} \|\pa_x^\alpha\nabla_x g\|_{L^2_{x,p}}$. Using the bootstrap assumption \eqref{eq:boots} and the decomposition $g = Pg+(I-P)g$, together with \eqref{eq:Pg-equiv-der}, we have
\[
\|\pa_x^\alpha\nabla_xg\|_{L^2_{x,p}}^2 \le C \|\pa_x^\alpha\nabla_x(a,b,c)\|_{L^2_x}^2  + C \|\pa_x^\alpha\nabla_x(I-P)g\|_{L^2_{x,p}}^2.
\]
Combining the preceding estimates concludes the desired result. 
\end{proof}

Combining the previous estimates, we obtain the desired macroscopic
dissipation bound.

\begin{lemma} \label{lem:mac-diss-s5}
Let $s\ge4$, and suppose that $g$ satisfies \eqref{eq:boots}. Then, for $\delta_{\rm in}>0$ sufficiently small, there exists an interaction functional
\[
\calI(t) := \sum_{|\alpha|\le s-1} \calI_\alpha(t)
\]
such that
\[
\frac{\rd}{\rdt}\calI(t) + \frac12 \|\nabla_x(a,b,c)\|_{H^{s-1}_x}^2 \le C \sum_{|\alpha|\le s} \|\pa_x^\alpha(I-P)g\|_D^2.
\]
Moreover,
\bq\label{eq:I-low-ord}
|\calI(t)| \le C \|g(t)\|_{H^s}^2.
\eq
\end{lemma}

\begin{proof}
Summing \eqref{eq:mac-rec-s5} over
$|\alpha|\le s-1$, we obtain
\begin{align*} 
\|\nabla_x(a,b,c)\|_{H^{s-1}_x}^2 &\le - \frac{\rd}{\rdt} \calI(t) + C \sum_{|\alpha|\le s-1} \|\pa_x^\alpha\nabla_x(I-P)g\|_{L^2_{x,p}}^2 \cr
&\quad + C \sum_{|\alpha|\le s-1} \|\pa_x^\alpha(I-P)g\|_{L^2_{x,p}}^2 + C \sum_{|\alpha|\le s-1} \|\pa_x^\alpha\tilde h\|_{L^2_x}^2.
\end{align*}
By \eqref{eq:h-ineq-section5}, we get
\[
\sum_{|\alpha|\le s-1} \|\pa_x^\alpha\tilde h\|_{L^2_x}^2 \le C\delta_{\rm in}^2 \|\nabla_x(a,b,c)\|_{H^{s-1}_x}^2 + C\delta_{\rm in}^2 \sum_{|\alpha|\le s-1} \|\pa_x^\alpha\nabla_x(I-P)g\|_{L^2_{x,p}}^2.
\]
Moreover, since the $D$-norm controls the $L^2_p$-norm, we have
\[
\sum_{|\alpha|\le s-1} \lt(\|\pa_x^\alpha\nabla_x(I-P)g\|_{L^2_{x,p}}^2 + \|\pa_x^\alpha(I-P)g\|_{L^2_{x,p}}^2 \rt) \le C \sum_{|\alpha|\le s} \|\pa_x^\alpha(I-P)g\|_D^2.
\]
Consequently, this yields
\[
\|\nabla_x(a,b,c)\|_{H^{s-1}_x}^2 \le - \frac{\rd}{\rdt} \calI(t) + C \sum_{|\alpha|\le s} \|\pa_x^\alpha(I-P)g\|_D^2 + C\delta_{\rm in}^2 \|\nabla_x(a,b,c)\|_{H^{s-1}_x}^2.
\]
Choosing $\delta_{\rm in}>0$ sufficiently small so that $C\delta_{\rm in}^2 \le \frac12$, we absorb the last term on the right-hand side into the left-hand side and obtain
\[
\frac12 \|\nabla_x(a,b,c)\|_{H^{s-1}_x}^2 \le - \frac{\rd}{\rdt} \calI(t) + C \sum_{|\alpha|\le s} \|\pa_x^\alpha(I-P)g\|_D^2.
\]
 
It remains to verify \eqref{eq:I-low-ord}. By the explicit form of $\calI_\alpha$ in \eqref{eq:Ialpha-s5}, $\calI(t)$ is a finite linear combination of terms of the form
\[
\inttr \pa_x^\alpha(I-P)g\, \xi(p)\, \pa_x^\alpha\nabla_xq \,\rdp\rdx , \quad q\in\{a,b,c\}, \quad \text{and} \quad \intr \pa_x^\alpha b\cdot \pa_x^\alpha\nabla_xa \,\rdx,
\]
where $|\alpha|\le s-1$ and $\xi$ is a fixed smooth exponentially decaying function of $p$. By Cauchy's inequality, the exponential decay of $\xi$, and \eqref{eq:Pg-equiv-der}, we have
\[
\lt| \inttr \pa_x^\alpha(I-P)g\, \xi(p)\, \pa_x^\alpha\nabla_xq \,\rdp\rdx  \rt|  \le C \|\pa_x^\alpha(I-P)g\|_{L^2_{x,p}} \|\pa_x^\alpha\nabla_xq\|_{L^2_x}  \le C \|g\|_{H^s}^2,
\]
and
\[
\lt| \intr \pa_x^\alpha b\cdot \pa_x^\alpha\nabla_xa \,\rdx \rt| \le \|\pa_x^\alpha b\|_{L^2_x} \|\pa_x^\alpha\nabla_xa\|_{L^2_x} \le C \|g\|_{H^s}^2.
\]
Since the number of interaction terms is finite, summing the preceding bounds gives \eqref{eq:I-low-ord}. This completes the proof.
\end{proof}

%
%
%
%
%
%
%
%
%
%

\subsection{Closure of the global energy estimate}

We now combine the microscopic energy estimates and the macroscopic dissipation estimate to close the global a priori estimate. Define
\[
    \mcD_s(g) := \sum_{|\alpha|+|\beta|\le s} \|\pa_x^\alpha\pa_p^\beta(I-P)g\|_D^2 + \|\nabla_x(a,b,c)\|_{H^{s-1}_x}^2.
\]
The energy functional will be obtained by adding a sufficiently small multiple of the interaction functional constructed in Lemma \ref{lem:mac-diss-s5} to an appropriate weighted microscopic energy.

\begin{lemma} \label{lem:glob-ener-s5}
Let $s\ge4$. There exist constants $\delta_{\rm in}>0$, $\lambda>0$, $C_E>1$ such that the following holds. Let $g$ be a smooth solution to
\eqref{eq:pert-glo} on $[0,T]$ satisfying
\[
    \sup_{0\le t\le T} \|g(t)\|_{H^s} \le \delta_{\rm in}.
\]
Then there exists an energy functional $\mcE_s(g)(t)$ satisfying
\bq\label{eq:ener-equi-s5}
    C_E^{-1} \|g(t)\|_{H^s}^2 \le \mcE_s(g)(t) \le C_E \|g(t)\|_{H^s}^2
\eq
and
\[
\frac{\rd}{\rdt} \mcE_s(g)(t) +  \lambda \mcD_s(g)(t) \le 0
\]
for all $0\le t\le T$.
\end{lemma}

\begin{proof}
We first combine the microscopic estimates. For each $1\le k\le s$, define
\[
\mathfrak E_k(g) := \sum_{\substack{|\alpha|+|\beta|\le s\\|\beta|=k}} \|\pa_x^\alpha\pa_p^\beta(I-P)g\|_{L^2_{x,p}}^2,
\quad 
\mathfrak D_k(g) := \sum_{\substack{|\alpha|+|\beta|\le s\\|\beta|=k}} \|\pa_x^\alpha\pa_p^\beta(I-P)g\|_D^2.
\]
For the purely spatial derivatives, set
\[
\mathfrak E_0(g) := \|g\|_{L^2_p(H^s_x)}^2,
\quad
\mathfrak D_0(g) := \sum_{|\alpha|\le s} \|\pa_x^\alpha(I-P)g\|_D^2.
\]

Let $\e_*>0$ be a sufficiently small constant to be fixed below. Choose positive constants $1=a_0>a_1>\cdots>a_s>0$ recursively. The weights are chosen sufficiently small so that the lower $p$-order microscopic terms on the right-hand side of \eqref{eq:sum2} are absorbed by the dissipation terms at the
preceding levels. In addition, we choose them so that
\bq\label{eq:small-w-mac}
    \sum_{k=1}^s C_{s,k}a_k \le \frac{\e_*}{2},
\eq
where $C_{s,k}>0$ denotes the constant multiplying the macroscopic gradient terms at the $p$-derivative level $k$.

Define the weighted microscopic energy and dissipation by
\[
\mcE_s^{\rm mic}(g) := \mathfrak E_0(g) + \sum_{k=1}^s a_k\mathfrak E_k(g),
\quad
\mcD_s^{\rm mic}(g) := \mathfrak D_0(g) + \sum_{k=1}^s a_k\mathfrak D_k(g).
\]
Since the weights $a_1,\ldots,a_s$ are fixed positive constants and $P$ is a finite-dimensional projection, there exists a constant $C_s>1$ such that
\bq\label{eq:weight-mic-equi}
    C_s^{-1} \|g\|_{H^s}^2 \le \mcE_s^{\rm mic}(g) \le C_s  \|g\|_{H^s}^2.
\eq
Similarly,
\bq\label{eq:weight-diss-equi}
    C_s^{-1} \sum_{|\alpha|+|\beta|\le s} \|\pa_x^\alpha\pa_p^\beta(I-P)g\|_D^2 \le \mcD_s^{\rm mic}(g)
\eq
and
\[
    \mcD_s^{\rm mic}(g) \le C_s \sum_{|\alpha|+|\beta|\le s} \|\pa_x^\alpha\pa_p^\beta(I-P)g\|_D^2.
\]

We now multiply the estimate in Lemma \ref{priori 2} at the $p$-derivative level $k$ by $a_k$, sum over $1\le k\le s$, and add the spatial estimate in Lemma \ref{priori 1}. By the recursive choice of the weights, all lower $p$-order microscopic terms are absorbed into $\mcD_s^{\rm mic}(g)$. Moreover, by \eqref{eq:small-w-mac}, the macroscopic gradient terms arising from the mixed derivative estimates have a sufficiently small coefficient. Finally, choosing the bootstrap constant $\delta_{\rm in}>0$ sufficiently small, we also control the nonlinear macroscopic term arising
from Lemma \ref{priori 1}. Hence, we obtain
\bq\label{eq:comb-mic-s5}
\frac{\rd}{\rdt} \mcE_s^{\rm mic}(g) + c_1 \mcD_s^{\rm mic}(g) \le \e_* \|\nabla_x(a,b,c)\|_{H^{s-1}_x}^2
\eq
for some constant $c_1>0$.

Next, by Lemma \ref{lem:mac-diss-s5}, there exists an interaction functional $\calI(t)$ satisfying
\[
\frac{\rd}{\rdt} \calI(t) + \frac12  \|\nabla_x(a,b,c)\|_{H^{s-1}_x}^2 \le C_0 \sum_{|\alpha|\le s} \|\pa_x^\alpha(I-P)g\|_D^2.
\]
Since
\[
\sum_{|\alpha|\le s} \|\pa_x^\alpha(I-P)g\|_D^2 \le C \mcD_s^{\rm mic}(g),
\]
we have
\bq\label{eq:mac-diss-sh-we-s5}
\frac{\rd}{\rdt} \calI(t) + \frac12 \|\nabla_x(a,b,c)\|_{H^{s-1}_x}^2 \le C_0 \mcD_s^{\rm mic}(g).
\eq
Moreover,
\bq\label{eq:I-lower-s5}
    |\calI(t)| \le C_I \|g(t)\|_{H^s}^2
\eq
for some constant $C_I>0$.

Choose a fixed constant $\vartheta>0$ sufficiently small so that
\bq\label{eq:vartheta-cho}
    \vartheta C_0 \le \frac{c_1}{2} \quad \text{and} 
    \quad     
    \vartheta C_I \le \frac{1}{2C_s}.
\eq
After fixing $\vartheta$, choose the weights
$a_1,\ldots,a_s$ and the bootstrap constant $\delta_{\rm in}>0$ sufficiently small
so that
\bq\label{eq:epsilon-s-cho}
    \e_* \le \frac{\vartheta}{4}.
\eq

Multiplying \eqref{eq:mac-diss-sh-we-s5} by $\vartheta$ and adding the resulting inequality to \eqref{eq:comb-mic-s5}, we obtain
\[
\frac{\rd}{\rdt} \lt\{ \mcE_s^{\rm mic}(g) + \vartheta\calI(t) \rt\} + \lt( c_1-\vartheta C_0 \rt) \mcD_s^{\rm mic}(g) + \lt( \frac{\vartheta}{2} - \e_* \rt) \|\nabla_x(a,b,c)\|_{H^{s-1}_x}^2 \le
0.
\]
By \eqref{eq:vartheta-cho} and \eqref{eq:epsilon-s-cho}, we have
\[
c_1-\vartheta C_0 \ge \frac{c_1}{2} \quad \text{and} \quad \frac{\vartheta}{2}-\e_* \ge \frac{\vartheta}{4}.
\]
Therefore, defining
\[
\mcE_s(g)(t) := \mcE_s^{\rm mic}(g)(t) + \vartheta\calI(t),
\]
we obtain
\[
\frac{\rd}{\rdt} \mcE_s(g)(t) + \lambda \mcD_s(g)(t) \le 0
\]
for some constant $\lambda>0$, where we used \eqref{eq:weight-diss-equi}.

Finally, by \eqref{eq:weight-mic-equi}, \eqref{eq:I-lower-s5}, and \eqref{eq:vartheta-cho}, we have $\mcE_s(g)(t) \sim \|g(t)\|_{H^s}^2$. More precisely, after adjusting the equivalence constant if necessary,
\[
C_E^{-1} \|g(t)\|_{H^s}^2 \le \mcE_s(g)(t) \le C_E \|g(t)\|_{H^s}^2.
\]
This completes the proof.
\end{proof}

%
%
%
%
%
%
%
%
%
%

\section{Global existence and large-time behavior}\label{sec:gwp-lt}

In this section, we complete the proof of the global-in-time theory and derive the large-time behavior of the solution. The main ingredient is the global energy inequality established in Section \ref{sec:global-energy}. We first use this estimate to continue the local solution globally in time. We then establish the propagation of negative Sobolev norms and use an interpolation argument to obtain algebraic decay rates.

%
%
%
%
%
%
%
%
%
%

\subsection{Global existence}

We first prove the global existence part of Theorem \ref{ge}. Recall that the local solution constructed in Theorem \ref{LE} satisfies $g \in \mcC([0,T_*];H^s(\R^3 \times \R^3))$ and solves
\bq\label{eq:s6-pert}
    \pa_t g + p\cdot\nabla_xg = Lg+\Gamma(g).
\eq
The key point is that the a priori estimate in Lemma \ref{lem:glob-ener-s5} is uniform on any time interval on which the solution remains sufficiently small in $H^s$.

\begin{proof}[Proof of the global existence part of Theorem \ref{ge}]
Let $\delta_{\rm in}>0$ and $C_E>1$ be the constants in Lemma \ref{lem:glob-ener-s5}. We assume that the initial perturbation is sufficiently small so that
\bq\label{eq:glob-sma-ini}
    \|g_0\|_{H^s}^2 \le \frac{\delta_{\rm in}^2}{16C_E^2}.
\eq

By the local well-posedness theorem, there exists a local solution $g \in \mcC([0,T_0];H^s(\R^3 \times \R^3))$ for some $T_0>0$. We define the bootstrap lifespan by
\[
T^* := \sup \lt\{ T>0: \text{the solution exists on }[0,T] \text{ and} \sup_{0\le t\le T} \|g(t)\|_{H^s} \le \delta_{\rm in} \rt\}.
\]
By the local well-posedness theorem, the continuity of the solution in $H^s$, and \eqref{eq:glob-sma-ini}, we have $T^*>0$. Fix any $T<T^*$. Since the bootstrap assumption in
Lemma \ref{lem:glob-ener-s5} is satisfied on $[0,T]$, we have
\bq\label{eq:s6-ener-ineq}
    \frac{\rd}{\rdt}\mcE_s(g)(t) + \lambda \mcD_s(g)(t) \le 0.
\eq
Integrating \eqref{eq:s6-ener-ineq} over $[0,t]$, where
$0\le t\le T$, gives
\bq\label{eq:s6-ener-int}
\mcE_s(g)(t) + \lambda \int_0^t \mcD_s(g)(\tau) \,\rdta \le \mcE_s(g_0).
\eq
Using the energy equivalence \eqref{eq:ener-equi-s5}, we obtain
\[
\|g(t)\|_{H^s}^2 \le C_E \mcE_s(g)(t) \le C_E \mcE_s(g_0) \le C_E^2 \|g_0\|_{H^s}^2 \le \frac{\delta_{\rm in}^2}{16}.
\]
Since $T<T^*$ was arbitrary, it follows that
\bq\label{eq:s6-boots-imp}
\sup_{0\le t<T^*} \|g(t)\|_{H^s} \le \frac{\delta_{\rm in}}{4}.
\eq

We claim that $T^*=+\infty$. Suppose, for contradiction, that  $T^*<\infty$. By \eqref{eq:s6-boots-imp}, the $H^s$-norm of the solution remains uniformly bounded by $\delta_{\rm in}/4$ on $[0,T^*)$. The continuation criterion associated with the local well-posedness theory hence allows us to restart the solution at a time $t_0<T^*$ sufficiently close to $T^*$, with an existence interval whose length depends only on the uniform $H^s$-bound. This extends the solution beyond $T^*$. Moreover, by continuity in $H^s$, the bound $\|g(t)\|_{H^s} < \delta_{\rm in}$ continues to hold on a nontrivial interval beyond $T^*$. This contradicts the definition of $T^*$. Hence, we have $T^* = +\infty$.

Since \eqref{eq:s6-ener-int} holds on every finite time interval, we may let $t\to\infty$. Using the energy equivalence once again, we obtain
\bq\label{eq:glo-main-bound-s6}
\sup_{t\ge0} \|g(t)\|_{H^s}^2 + \int_0^\infty \mcD_s(g)(t) \,\rdt \le C  \|g_0\|_{H^s}^2.
\eq
In particular,
\[
\int_0^\infty \sum_{|\alpha|+|\beta|\le s} \|\pa_x^\alpha \pa_p^\beta(I-P)g(t)\|_D^2\,\rdt +  \int_0^\infty \|\nabla_x(a,b,c)(t)\|_{H^{s-1}_x}^2 \,\rdt \le C\|g_0\|_{H^s}^2.
\]

The uniqueness of the global solution follows from the local uniqueness result by continuation. The admissibility bounds are also preserved throughout the continuation argument: nonnegativity is propagated for all $t\ge0$, and, in the fermionic case $\kappa=-1$, so is the upper Pauli bound $f\le \frac1\hbar$. This completes the proof of the global existence part of Theorem \ref{ge}.
\end{proof}

%
%
%
%
%
%
%
%
%
%

\subsection{Large-time behavior}

We now prove the large-time behavior of the global solution. The argument is based on the propagation of negative Sobolev norms and the interpolation method of \cite{GW12}. Throughout this subsection, $g$ denotes the global solution constructed above. We assume in addition that $\Lambda_x^{-\tilde s}g_0 \in L^2(\R^3 \times \R^3)$ for some $0<\tilde s<\frac32$. For notational simplicity, we write $\Ls := \Lambda_x^{-\tilde s}$.

%
%
%
%
%
%
%
%
%
%

\subsubsection{Negative Sobolev estimates}

We first recall the interpolation and Hardy--Littlewood--Sobolev inequalities used below.  

\begin{lemma} 
The following estimates hold.

\begin{enumerate}
\item
Let $\tilde s\ge0$ and let $l\in\mN\cup\{0\}$. Then 
\bq\label{eq:NS-interp-s6}
\|\nabla_x^l f\|_{L^2_{x,p}} \le C \|\nabla_x^{l+1}f\|_{L^2_{x,p}}^{1-\theta} \|\Lambda^{-\tilde s}f\|_{L^2_{x,p}}^\theta,
    \quad
    \theta = \frac{1}{l+1+\tilde s}.
\eq

\item
Let $0<\tilde s<3$, and let $1<r_1<r_2<\infty$ satisfy
\[
\frac1{r_2} + \frac{\tilde s}{3} = \frac1{r_1}.
\]
Then
\bq\label{eq:Riesz-s6}
\|\Lambda^{-\tilde s}f\|_{L^{r_2}_x} \le C \|f\|_{L^{r_1}_x}.
\eq
The same estimate holds for $L^2_p$-valued functions.
\end{enumerate}
\end{lemma}

We next estimate the nonlinear term after applying the negative Sobolev operator. This is the only point at which the proof differs slightly from the nonlinear estimates in Section \ref{ssec:non-est}.

\begin{lemma}\label{LB Gamma}
Let $s\ge4$ and $0<\tilde s<\frac32$. Suppose that $\sup_{t\ge0} \|g(t)\|_{H^s} \le \delta_{\rm in}$ for sufficiently small $\delta_{\rm in}>0$. Then
\bq\label{eq:LB-Gamma}
 \lt| \lt\lal \Ls\Gamma(g), \Ls g \rt\ral_{L^2_{x,p}} \rt|  \le C \|\Ls(I-P)g\|_D \|g\|_{L^2_{x,p}} \lt( \|(I-P)g\|_{L^{3/\tilde s}_xD_p} + \|(a,b,c)\|_{L^{3/\tilde s}_x} \rt).
\eq
\end{lemma}

\begin{proof}
Since $\Ls$ acts only in the $x$ variable, it commutes with $P$. Moreover, the nonlinear collision term satisfies $P\Gamma(g) = 0$. Thus,
\[
\lt\lal \Ls\Gamma(g), \Ls g \rt\ral_{L^2_{x,p}} = \lt\lal \Ls\Gamma(g), \Ls(I-P)g \rt\ral_{L^2_{x,p}}.
\]

The argument is similar to the proof of the nonlinear estimates in Section \ref{ssec:non-est}. We present the computation for a representative term. Recall that
\[
\calT F := \frac1{\sqrt{\mu_\hbar }} \nabla_p\cdot \lt( \sqrt{\mu_\hbar }F \rt) \quad \text{and} \quad  \calA_-g := \nabla_pg - \frac12p\eta_\hbar g.
\]
Since $\nabla_p(g\sqrt{\mu_\hbar }) = \sqrt{\mu_\hbar }\calA_-g$, the representative nonlinear term can be written, up to a fixed multiplicative constant, as $\Gamma_1 = \calT \lt( \mathscr c(g)\calA_-g \rt)$, where
\[
\mathscr c(g) = \intr |p|^2\eta_\hbar g\sqrt{\mu_\hbar } \,\rdp.
\]
Using the weighted flux pairing \eqref{eq:wei-f-pair}, we obtain
\[
\lt| \lt\lal \Ls\Gamma_1, \Ls(I-P)g \rt\ral_{L^2_{x,p}} \rt|  \le C \|\Ls(I-P)g\|_D \lt\| \Ls \lt( \mathscr c(g)\calA_-g \rt) \rt\|_{L^2_{x,p}}.
\]
Let $1<q<2$ be determined by $\frac1q = \frac12 + \frac{\tilde s}{3}$. By the Hardy--Littlewood--Sobolev inequality \eqref{eq:Riesz-s6}, applied with $r_1=q$ and $r_2=2$, we have
\[
\lt\| \Ls \lt( \mathscr c(g)\calA_-g \rt) \rt\|_{L^2_{x,p}} \le C \lt\| \mathscr c(g)\calA_-g \rt\|_{L^q_xL^2_p}.
\]
The moment bound $|\mathscr c(g)| \le C \|g\|_{L^2_p}$ and H\"older's inequality in the $x$ variable yield
\begin{align*} 
\lt\| \mathscr c(g)\calA_-g \rt\|_{L^q_xL^2_p} &\le C \|\mathscr c(g)\|_{L^2_x} \|\calA_-g\|_{L^{3/\tilde s}_xL^2_p} \cr
&\le C \|g\|_{L^2_{x,p}} \lt( \|(I-P)g\|_{L^{3/\tilde s}_xD_p} + \|Pg\|_{L^{3/\tilde s}_xL^2_p} \rt) \cr
&\le C \|g\|_{L^2_{x,p}} \lt( \|(I-P)g\|_{L^{3/\tilde s}_xD_p} + \|(a,b,c)\|_{L^{3/\tilde s}_x} \rt).
\end{align*}
Here we used the decomposition $g=Pg+(I-P)g$, the estimate $\|\calA_-(I-P)g\|_{L^2_p} \le C|(I-P)g|_D$, and the finite-dimensional equivalence $\|\calA_-Pg\|_{L^2_p} + \|Pg\|_{L^2_p} \le C|(a,b,c)|$. Combining the above, we obtain the desired estimate for $\Gamma_1$.

The remaining terms in $\Gamma(g)$ are estimated in the same manner. More precisely, we use the weighted flux pairing \eqref{eq:wei-f-pair}, the exponential decay of the $p$-dependent coefficients, the moment bounds, and the estimates for the nonlinear macroscopic remainders $N_u$ and $N_\Theta$. Every term contains at least two factors depending on $g$, and after applying the macro--micro decomposition to the flux factor, the same upper bound as in \eqref{eq:LB-Gamma} follows. Combining the estimates for all components of $\Gamma(g)$ gives \eqref{eq:LB-Gamma}.
\end{proof}

We now prove the propagation of the negative Sobolev norm.

\begin{lemma}\label{LB bound}
Let $s\ge4$, let $0<\tilde s<\frac32$, and let $g$ be the global solution constructed above. Assume that $\Ls g_0 \in L^2_{x,p}$. Then the following differential inequalities hold. If
$0<\tilde s\le\frac12$, then
\bq\label{NS ineq 1}
\frac{\rd}{\rdt} \|\Ls g\|_{L^2_{x,p}}^2 + \lambda_0 \|\Ls(I-P)g\|_D^2 \le C\delta_{\rm in}^2 \mcD_s(g).
\eq
If $\frac12<\tilde s<\frac32$, then
\bq\label{NS ineq 2}
\frac{\rd}{\rdt} \|\Ls g\|_{L^2_{x,p}}^2 + \lambda_0 \|\Ls(I-P)g\|_D^2  \le C\delta_{\rm in}^2\mcD_s(g) + C \|g\|_{L^2_{x,p}}^2 \|(a,b,c)\|_{L^2_x}^{2\tilde s-1} \|\nabla_x(a,b,c)\|_{L^2_x}^{3-2\tilde s}.
\eq
In particular, if $0<\tilde s\le\frac12$, then
\bq\label{eq:neg-uni-b-h}
\sup_{t\ge0} \|\Ls g(t)\|_{L^2_{x,p}}^2 \le C \lt( \|\Ls g_0\|_{L^2_{x,p}}^2 + \|g_0\|_{H^s}^2 \rt).
\eq
\end{lemma}

\begin{proof}
Applying $\Ls$ to \eqref{eq:s6-pert} and taking the $L^2_{x,p}$-inner product with $\Ls g$, we obtain
\[
\frac12 \frac{\rd}{\rdt} \|\Ls g\|_{L^2_{x,p}}^2  = \lt\lal L(\Ls g), \Ls g \rt\ral_{L^2_{x,p}} + \lt\lal \Ls\Gamma(g), \Ls g \rt\ral_{L^2_{x,p}} =: K_1+K_2.
\]
Here we used the skew-symmetry of $p\cdot\nabla_x$ and the fact that $\Ls$ commutes with $L$, since $\Ls$ acts only on the $x$ variable whereas $L$ acts only on the $p$ variable.

By the coercivity of $L$, we have
\[
K_1 \le - \lambda_0 \|\Ls(I-P)g\|_D^2.
\]
On the other hand, Lemma \ref{LB Gamma} and Young's inequality yield
\begin{align}\label{eq:K2-g}
\begin{aligned}
|K_2| &\le C \|\Ls(I-P)g\|_D \|g\|_{L^2_{x,p}} \lt( \|(I-P)g\|_{L^{3/\tilde s}_xD_p} + \|(a,b,c)\|_{L^{3/\tilde s}_x} \rt) \cr
&\le \frac{\lambda_0}{2} \|\Ls(I-P)g\|_D^2 + C \|g\|_{L^2_{x,p}}^2 \lt( \|(I-P)g\|_{L^{3/\tilde s}_xD_p}^2 + \|(a,b,c)\|_{L^{3/\tilde s}_x}^2 \rt).
\end{aligned}
\end{align}

We estimate the last factor according to the size of $\tilde s$.

\medskip
\noindent
{\it Case 1: $0<\tilde s\le\frac12$.} In this case, $\frac3{\tilde s} \ge 6$. By the Gagliardo--Nirenberg interpolation inequality in the $x$ variable, applied to the $D_p$-valued function $(I-P)g$, we have
\bq\label{eq:GN-mic-low-s6}
\|(I-P)g\|_{L^{3/\tilde s}_xD_p} \le C \|\nabla_x(I-P)g\|_{L^2_xD_p}^{\frac12+\tilde s} \|\nabla_x^2(I-P)g\|_{L^2_xD_p}^{\frac12-\tilde s}.
\eq
Similarly,
\bq\label{eq:GN-mac-low-s6}
\|(a,b,c)\|_{L^{3/\tilde s}_x} \le C \|\nabla_x(a,b,c)\|_{L^2_x}^{\frac12+\tilde s} \|\nabla_x^2(a,b,c)\|_{L^2_x}^{\frac12-\tilde s}.
\eq
Since $s\ge4$, all terms on the right-hand sides of \eqref{eq:GN-mic-low-s6} and \eqref{eq:GN-mac-low-s6} are contained in $\mcD_s(g)$. Hence
\[
\|(I-P)g\|_{L^{3/\tilde s}_xD_p}^2 + \|(a,b,c)\|_{L^{3/\tilde s}_x}^2 \le C\mcD_s(g).
\]
Using the global smallness $\|g\|_{L^2_{x,p}} \le \|g\|_{H^s} \le \delta_{\rm in}$ in \eqref{eq:K2-g}, we get
\[
|K_2| \le \frac{\lambda_0}{2} \|\Ls(I-P)g\|_D^2 + C\delta_{\rm in}^2\mcD_s(g).
\]
Combining this with the estimate for $K_1$ gives \eqref{NS ineq 1}.

\medskip
\noindent
{\it Case 2: $\frac12<\tilde s<\frac32$.} In this case, $2 < \frac3{\tilde s} < 6$. For the microscopic contribution, the Gagliardo--Nirenberg interpolation inequality gives
\[
\|(I-P)g\|_{L^{3/\tilde s}_xD_p} \le C \|(I-P)g\|_{L^2_xD_p}^{\tilde s-\frac12} \|\nabla_x(I-P)g\|_{L^2_xD_p}^{\frac32-\tilde s}.
\]
Using the elementary inequality $X^{2\tilde s-1}Y^{3-2\tilde s} \le C(X^2+Y^2)$ for $X,Y\ge0$, and the smallness of $g$ in $H^s$, we obtain
\bq\label{eq:mic-h-Ds-s6}
\|g\|_{L^2_{x,p}}^2 \|(I-P)g\|_{L^{3/\tilde s}_xD_p}^2 \le C\delta_{\rm in}^2\mcD_s(g).
\eq

For the macroscopic contribution, we use
\[
\|(a,b,c)\|_{L^{3/\tilde s}_x} \le C \|(a,b,c)\|_{L^2_x}^{\tilde s-\frac12} \|\nabla_x(a,b,c)\|_{L^2_x}^{\frac32-\tilde s}.
\]
This implies
\bq\label{eq:mac-h-term-s6}
\|g\|_{L^2_{x,p}}^2 \|(a,b,c)\|_{L^{3/\tilde s}_x}^2 \le C \|g\|_{L^2_{x,p}}^2 \|(a,b,c)\|_{L^2_x}^{2\tilde s-1} \|\nabla_x(a,b,c)\|_{L^2_x}^{3-2\tilde s}.
\eq
Combining \eqref{eq:K2-g}, \eqref{eq:mic-h-Ds-s6}, and \eqref{eq:mac-h-term-s6}, and absorbing the microscopic coercive term, we obtain \eqref{NS ineq 2}.

It remains to prove the uniform bound \eqref{eq:neg-uni-b-h}. Assume that $0<\tilde s\le\frac12$. Integrating \eqref{NS ineq 1} over $[0,t]$ and using the global energy
estimate \eqref{eq:glo-main-bound-s6}, we obtain
\[
\|\Ls g(t)\|_{L^2_{x,p}}^2 \le \|\Ls g_0\|_{L^2_{x,p}}^2 + C\delta_{\rm in}^2 \int_0^t \mcD_s(g)(\tau) \,\rdta \le C \lt( \|\Ls g_0\|_{L^2_{x,p}}^2 + \|g_0\|_{H^s}^2 \rt).
\]
This concludes the desired bound \eqref{eq:neg-uni-b-h}.
\end{proof}

%
%
%
%
%
%
%
%
%
%

\subsubsection{Algebraic decay}

We now use the negative Sobolev estimates to derive algebraic decay. For each integer $0\le l\le s-1$, define the spatial dissipation functional
\begin{align}\label{eq:Dl-x-def}
\mcD_l^{\rm x}(t) := \sum_{l\le|\alpha|\le s} \|\pa_x^\alpha(I-P)g(t)\|_D^2 + \sum_{l\le|\alpha|\le s-1} \|\pa_x^\alpha\nabla_x(a,b,c)(t)\|_{L^2_x}^2.
\end{align}
We also define the modified spatial energy
\[
\mcE_l^{\rm x}(t) := \sum_{l\le|\alpha|\le s} \|\pa_x^\alpha g(t)\|_{L^2_{x,p}}^2 + \vartheta \sum_{l\le|\alpha|\le s-1} \calI_\alpha(t),
\]
where $\vartheta>0$ is the sufficiently small fixed constant used in the closure of the global energy estimate.

By \eqref{eq:I-low-ord}, after decreasing $\vartheta>0$ if necessary, we have
\bq\label{eq:El-x-equi}
\mcE_l^{\rm x}(t) \sim \sum_{l\le|\alpha|\le s} \|\pa_x^\alpha g(t)\|_{L^2_{x,p}}^2.
\eq

Repeating the spatial energy argument in Lemma \ref{priori 1}, summing over $l\le|\alpha|\le s$, and combining it with the macroscopic dissipation estimate \eqref{eq:mac-rec-s5}, summed over $l\le|\alpha|\le s-1$, we obtain
\bq\label{eq:El-x-diff}
\frac{\rd}{\rdt} \mcE_l^{\rm x}(t) + \lambda \mcD_l^{\rm x}(t) \le 0,  \quad  0\le l\le s-1.
\eq
Here and below, $\lambda>0$ denotes a generic constant independent of time.

Moreover, by the macro--micro decomposition $g = Pg+(I-P)g$, the finite-dimensional equivalence between $Pg$ and $(a,b,c)$, and the coercivity of the $D$-norm, we have
\[
C\mcD_l^{\rm x}(t) \ge  \mcE_{l+1}^{\rm x}(t), \quad 0\le l\le s-1,
\]
where, for $l=s-1$, we use the convention
\[
\mcE_s^{\rm x}(t) := \sum_{|\alpha|=s} \|\pa_x^\alpha g(t)\|_{L^2_{x,p}}^2.
\]
Indeed, the microscopic part of $\mcE_{l+1}^{\rm x}$ is controlled by the first sum in \eqref{eq:Dl-x-def}, whereas the macroscopic part is controlled by the second sum.

We first consider the case $0<\tilde s\le\frac12$. By Lemma \ref{LB bound}, we have
\bq\label{eq:NS-unif-low-s6}
\sup_{t\ge0} \|\Ls g(t)\|_{L^2_{x,p}}^2 \le C \lt( \|\Ls g_0\|_{L^2_{x,p}}^2 + \|g_0\|_{H^s}^2 \rt).
\eq
Set
\[
M_{-\tilde s} := \sup_{t\ge0} \|\Ls g(t)\|_{L^2_{x,p}}.
\]
For each $0\le l\le s-1$, the negative Sobolev interpolation inequality \eqref{eq:NS-interp-s6} gives
\bq\label{eq:low-ord-inter-s6}
\|\nabla_x^lg(t)\|_{L^2_{x,p}}^2 \le C \|\nabla_x^{l+1}g(t)\|_{L^2_{x,p}}^{ \frac{2(l+\tilde s)}{l+1+\tilde s} } M_{-\tilde s}^{\frac{2}{l+1+\tilde s}}.
\eq
Since every spatial derivative of order at least $l+1$ is controlled by $\mcD_l^{\rm x}$, the equivalence \eqref{eq:El-x-equi} and \eqref{eq:low-ord-inter-s6} imply
\[
\mcE_l^{\rm x}(t) \le C \mcD_l^{\rm x}(t) + C \lt( \mcD_l^{\rm x}(t) \rt)^{\frac{l+\tilde s}{l+1+\tilde s}} M_{-\tilde s}^{\frac{2}{l+1+\tilde s}}.
\]
Since the global solution remains uniformly small in $H^s$, we may absorb the first term on the right-hand side into the second one after adjusting the constant. Consequently,
\[
\mcD_l^{\rm x}(t) \ge C \lt( \mcE_l^{\rm x}(t) \rt)^{1+\frac1{l+\tilde s}}, \quad 0\le l\le s-1.
\]
Here, the constant may depend on $M_{-\tilde s}$.

Combining this and \eqref{eq:El-x-diff}, we obtain
\bq\label{eq:El-x-ode}
\frac{\rd}{\rdt} \mcE_l^{\rm x}(t) + C \lt( \mcE_l^{\rm x}(t) \rt)^{1+\frac1{l+\tilde s}} \le 0.
\eq
Solving this differential inequality yields
\bq\label{eq:decay-low-s6}
\mcE_l^{\rm x}(t) \le C (1+t)^{-(l+\tilde s)}, \quad 0<\tilde s\le\frac12, \quad 0\le l\le s-1.
\eq

We next extend the uniform negative Sobolev bound to the range $\frac12<\tilde s<\frac32$. As observed after Lemma \ref{LB bound}, the assumption $\Lambda_x^{-\tilde s}g_0 \in L^2_{x,p}$ implies $\Lambda_x^{-\frac12}g_0 \in L^2_{x,p}$. We may thus apply \eqref{eq:decay-low-s6} at the endpoint
$\tilde s=\frac12$. Taking $l=0$ and $l=1$, we obtain
\[
\|\nabla_x^k g(t)\|_{L^2_{x,p}} + \|\nabla_x^k (a,b,c)(t)\|_{L^2_x} \le C (1+t)^{-\frac{1+2k}{4}}, \quad k=0,1.
\]
Here we used the finite-dimensional equivalence between $Pg$ and $(a,b,c)$.

For $\frac12<\tilde s<\frac32$, inequality \eqref{NS ineq 2} gives
\begin{align}\label{eq:NS-high-b-pf-s6}
\begin{aligned}
\frac{\rd}{\rdt} \|\Ls g(t)\|_{L^2_{x,p}}^2 &\le C\delta_{\rm in}^2 \mcD_s(g)(t) + C \|g(t)\|_{L^2_{x,p}}^2 \|(a,b,c)(t)\|_{L^2_x}^{2\tilde s-1} \|\nabla_x(a,b,c)(t)\|_{L^2_x}^{3-2\tilde s} \cr
&\le C\delta_{\rm in}^2 \mcD_s(g)(t) + C (1+t)^{-(\frac52-\tilde s)}. 
\end{aligned}
\end{align}
Since $\frac12 < \tilde s < \frac32$, we get $\frac52-\tilde s > 1$. Thus, the second term on the right-hand side of \eqref{eq:NS-high-b-pf-s6} is integrable in time. The first
term is also integrable by \eqref{eq:glo-main-bound-s6}. Consequently,
\[
\sup_{t\ge0} \|\Ls g(t)\|_{L^2_{x,p}}^2 \le C \lt( \|\Ls g_0\|_{L^2_{x,p}}^2 + \|g_0\|_{H^s}^2 \rt), \quad \frac12<\tilde s<\frac32.
\]
Combining this and \eqref{eq:NS-unif-low-s6}, we obtain
\[
\sup_{t\ge0} \|\Ls g(t)\|_{L^2_{x,p}}^2 \le C \lt( \|\Ls g_0\|_{L^2_{x,p}}^2 + \|g_0\|_{H^s}^2 \rt)
\]
for every $0<\tilde s<\frac32$.

We can now repeat the interpolation argument \eqref{eq:low-ord-inter-s6}--\eqref{eq:El-x-ode} for the full range $0<\tilde s<\frac32$. This yields
\[
\mcE_l^{\rm x}(t) \le C (1+t)^{-(l+\tilde s)}, \quad 0<\tilde s<\frac32, \quad 0\le l\le s-1.
\]
Using \eqref{eq:El-x-equi}, we conclude that
\[
\lt( \sum_{l\le|\alpha|\le s} \|\pa_x^\alpha g(t)\|_{L^2_{x,p}}^2 \rt)^{\frac12} \le C (1+t)^{-\frac{l+\tilde s}{2}}, \quad 0<\tilde s<\frac32, \quad 0\le l\le s-1.
\]
This completes the proof of the large-time behavior part of Theorem
\ref{ge}.

%
%
%
%
%
%
%
%
%
%

\section*{Acknowledgments}
 The work of Y.-P. Choi and J.-H. Hyun was supported by NRF grant no. 2022R1A2C1002820 and no. RS-2024-00406821. The work of B.-H. Hwang was supported by the National Research Foundation of Korea(NRF) grant funded by the Korean government(MSIT) (No.RS-2026-25475225).

%
%
%
%
%
%
%
%
%
%
 
\appendix

%
%
%
%
%
%
%
%
%
%

\section{Details of the local well-posedness argument}\label{app:lwp}

In this appendix, we provide the details of the construction of the local solution stated in Theorem \ref{LE}. We first establish uniform estimates for an iterative approximation scheme and then prove the convergence of the approximate solutions. We finally show that the admissible bounds on the distribution function are propagated by the local solution.

%
%
%
%
%
%
%
%
%
%

\subsection{Construction of the local solution}

We begin with the iterative scheme \eqref{eq:local-iter}. For convenience, we recall that
\bq\label{app:iter}
\pa_t g^{n+1} = Bg^{n+1}+\Gamma(g^n), \quad g^{n+1}|_{t=0} = g_0.
\eq
Equivalently,
\[
g^{n+1}(t) = e^{tB}g_0 + \int_0^t e^{(t-\tau)B}\Gamma(g^n(\tau))\,\rdta.
\]
The scheme is first carried out for smooth approximations of the initial datum and the nonlinear source term. Since all the estimates below are uniform with respect to the approximation parameters, we suppress these parameters in the notation.

\begin{lemma}\label{lem:local-uni-b}
Let $s\ge4$. There exist sufficiently small constants $ \delta_{\rm in}>0$ and $T_*>0$ such that, if $\|g_0\|_{H^s}\le \delta_{\rm in}$, then the sequence $\{g^n\}_{n\ge0}$ defined by \eqref{app:iter}
satisfies
\[
g^n\in\mcI(s,T_*; \delta_{\rm in}) \quad \text{for all }n\ge0.
\]
\end{lemma}

\begin{proof}
We prove the estimate by induction. The proof is divided into two steps. We first derive the purely spatial $H^s_x$ estimate and then establish the full mixed $H^s$ estimate.

For notational convenience, we set
\[
\calD_n := \sum_{|\alpha|+|\beta|\le s} \|\pa_x^\alpha\pa_p^\beta(I-P)g^n\|_D^2.
\]
By the finite-dimensionality of $P$ and the exponential decay of the basis functions of $\mcN$, we have
\bq\label{app:full-D-con}
\|g^n\|_{D,s}^2 \le C\lt( \calD_n + \|g^n\|_{H^s}^2 \rt).
\eq
Indeed, for every $|\alpha|+|\beta|\le s$,
\[
\|\pa_x^\alpha\pa_p^\beta Pg^n\|_D \le C_\beta \|\pa_x^\alpha g^n\|_{L^2_{x,p}}.
\]

\medskip
\noindent
{\bf Step 1. The spatial $H^s_x$ estimate.} For $|\alpha|\le s$, we apply $\pa_x^\alpha$ to \eqref{app:iter} and take the $L^2_{x,p}$-inner product with $\pa_x^\alpha g^{n+1}$. Since $p\cdot\nabla_x$ is skew-symmetric in $L^2_{x,p}$, we obtain
\bq\label{app:Hx-energy-0}
\frac12\frac{\rd}{\rdt} \|\pa_x^\alpha g^{n+1}\|_{L^2_{x,p}}^2 = \lt\lal L(\pa_x^\alpha g^{n+1}), \pa_x^\alpha g^{n+1} \rt\ral_{L^2_{x,p}} + \lt\lal \pa_x^\alpha\Gamma(g^n), \pa_x^\alpha g^{n+1} \rt\ral_{L^2_{x,p}}.
\eq
By the coercivity estimate \eqref{eq:L-coercivity},
\bq\label{app:Hx-lin-bo}
\lt\lal L(\pa_x^\alpha g^{n+1}), \pa_x^\alpha g^{n+1} \rt\ral_{L^2_{x,p}} \le -\lambda_0 \|\pa_x^\alpha(I-P)g^{n+1}\|_D^2.
\eq

We next estimate the nonlinear term. Applying Lemma \ref{Gamma 1} with $\beta=0$ and $h=g^{n+1}$, using Young's inequality and \eqref{app:full-D-con}, we find
\begin{align*} 
 \lt| \lt\lal \pa_x^\alpha\Gamma(g^n), \pa_x^\alpha g^{n+1} \rt\ral_{L^2_{x,p}} \rt|  
&\le C\|g^n\|_{H^s} \lt( \|\pa_x^\alpha(I-P)g^{n+1}\|_D + \|\pa_x^\alpha g^{n+1}\|_{L^2_{x,p}} \rt) \|g^n\|_{D,s} \cr
&\le \frac{\lambda_0}{2} \|\pa_x^\alpha(I-P)g^{n+1}\|_D^2 + C \|\pa_x^\alpha g^{n+1}\|_{L^2_{x,p}}^2 + C \|g^n\|_{H^s}^2 \lt( \calD_n + \|g^n\|_{H^s}^2 \rt).
\end{align*}

Combining the above estimates, and summing over $|\alpha|\le s$, we obtain
\bq\label{app:Hx-energy}
 \frac{\rd}{\rdt} \|g^{n+1}\|_{L^2_p(H^s_x)}^2 + \lambda_0 \sum_{|\alpha|\le s} \|\pa_x^\alpha(I-P)g^{n+1}\|_D^2  \le C \|g^{n+1}\|_{L^2_p(H^s_x)}^2 + C \|g^n\|_{H^s}^2 \lt( \calD_n + \|g^n\|_{H^s}^2 \rt).
\eq

For the first iterate, $g^0=0$ and thus $\Gamma(g^0)=0$. Hence, it follows from \eqref{app:Hx-energy-0} and \eqref{app:Hx-lin-bo} that 
\[
\frac{\rd}{\rdt} \|g^1\|_{L^2_p(H^s_x)}^2 + 2\lambda_0 \sum_{|\alpha|\le s} \|\pa_x^\alpha(I-P)g^1\|_D^2 \le 0,
\]
and consequently,
\[
\sup_{0\le t\le T_*} \|g^1(t)\|_{L^2_p(H^s_x)}^2 + 2\lambda_0 \sum_{|\alpha|\le s} \int_0^{T_*} \|\pa_x^\alpha(I-P)g^1(\tau)\|_D^2\,\rdta  \le \|g_0\|_{H^s}^2 <  \delta_{\rm in}^2.
\]

Assume that $g^n\in\mcI(s,T_*; \delta_{\rm in})$. Then, we get
\[
\sup_{0\le t\le T_*} \|g^n(t)\|_{H^s} \le \sqrt{ \delta_{\rm in}} \quad \text{and} \quad \int_0^{T_*} \calD_n(\tau)\,\rdta \le  \delta_{\rm in}.
\]
Integrating \eqref{app:Hx-energy} over $[0,t]\subset[0,T_*]$ and applying Gr\"onwall's inequality, we obtain
\[
\sup_{0\le t\le T_*} \|g^{n+1}(t)\|_{L^2_p(H^s_x)}^2 + \lambda_0 \sum_{|\alpha|\le s} \int_0^{T_*} \|\pa_x^\alpha(I-P)g^{n+1}(\tau)\|_D^2\,\rdta  \le C e^{CT_*} \lt( \|g_0\|_{H^s}^2 +  \delta_{\rm in}^2 + T_* \delta_{\rm in}^2 \rt).
\]
Thus, after choosing $ \delta_{\rm in}>0$ and $T_*>0$ sufficiently small, the spatial part of the induction estimate is closed.

\medskip
\noindent
{\bf Step 2. The mixed $H^s$ estimate.} We now estimate derivatives involving at least one $p$-derivative. Let $|\alpha|+|\beta|\le s$ and $|\beta|\ge1$. Applying $\pa_x^\alpha\pa_p^\beta$ to \eqref{app:iter} and taking the $L^2_{x,p}$-inner product with $\pa_x^\alpha\pa_p^\beta g^{n+1}$, we obtain
\begin{align*}
\begin{aligned}
\frac12\frac{\rd}{\rdt} \|\pa_x^\alpha\pa_p^\beta g^{n+1}\|_{L^2_{x,p}}^2
&= \lt\lal \pa_p^\beta L(\pa_x^\alpha g^{n+1}), \pa_x^\alpha\pa_p^\beta g^{n+1} \rt\ral_{L^2_{x,p}} - \lt\lal [\pa_p^\beta,p\cdot\nabla_x] \pa_x^\alpha g^{n+1}, \pa_x^\alpha\pa_p^\beta g^{n+1} \rt\ral_{L^2_{x,p}} \cr
&\quad + \lt\lal \pa_x^\alpha\pa_p^\beta\Gamma(g^n), \pa_x^\alpha\pa_p^\beta g^{n+1} \rt\ral_{L^2_{x,p}}.
\end{aligned}
\end{align*}

We estimate the three terms on the right-hand side separately. First, by the $p$-derivative coercivity estimate \eqref{eq:L-p-der-coer}, the finite-dimensionality of $P$, and the fact that the $D$-norm controls the $L^2$-norm, we have
\begin{align*}
\begin{aligned}
\lt\lal \pa_p^\beta L(\pa_x^\alpha g^{n+1}), \pa_x^\alpha\pa_p^\beta g^{n+1} \rt\ral_{L^2_{x,p}}
& \le -\frac{\lambda_1}{2} \|\pa_x^\alpha\pa_p^\beta(I-P)g^{n+1}\|_D^2 + C_\beta \|\pa_x^\alpha g^{n+1}\|_{L^2_{x,p}}^2 \cr
&\quad + C_\beta \sum_{|\nu|\le|\beta|-1} \|\pa_x^\alpha\pa_p^\nu(I-P)g^{n+1}\|_D^2.
\end{aligned}
\end{align*}

We next consider the transport commutator. Since $p\cdot\nabla_x$ is linear in $p$, we have
\[
[\pa_p^\beta,p\cdot\nabla_x]h = \sum_{\substack{1\le j\le3\\ \beta_j\ge1}} \beta_j \pa_{x_j}\pa_p^{\beta-\mfe_j}h.
\]
Using the decomposition $h=(I-P)h+Ph$, the finite-dimensionality of $P$, and Young's inequality, we obtain, for
any $\eta>0$,
\begin{align*}
\begin{aligned}
\lt| \lt\lal [\pa_p^\beta,p\cdot\nabla_x] \pa_x^\alpha g^{n+1}, \pa_x^\alpha\pa_p^\beta g^{n+1} \rt\ral_{L^2_{x,p}} \rt|
&\le \eta \|\pa_x^\alpha\pa_p^\beta(I-P)g^{n+1}\|_D^2 + C_{\eta,\beta} \|\pa_x^\alpha g^{n+1}\|_{L^2_p(H^1_x)}^2 \cr
&\quad + C_{\eta,\beta} \sum_{\substack{1\le j\le3\\ \beta_j\ge1}} \|\pa_x^{\alpha+\mfe_j} \pa_p^{\beta-\mfe_j}(I-P)g^{n+1}\|_D^2.
\end{aligned}
\end{align*}
The last term contains strictly fewer $p$-derivatives and preserves the total number of derivatives. Indeed,  $|\alpha+\mfe_j|+|\beta-\mfe_j| = |\alpha|+|\beta| \le s$.

Finally, applying Lemma \ref{Gamma 1} with $g=g^n$ and $h=g^{n+1}$, we obtain
\[
\lt| \lt\lal \pa_x^\alpha\pa_p^\beta\Gamma(g^n), \pa_x^\alpha\pa_p^\beta g^{n+1} \rt\ral_{L^2_{x,p}} \rt| \le C \|g^n\|_{H^s} \lt( \|\pa_x^\alpha\pa_p^\beta(I-P)g^{n+1}\|_D + \|\pa_x^\alpha g^{n+1}\|_{L^2_{x,p}} \rt) \|g^n\|_{D,s}.
\]
By Young's inequality and \eqref{app:full-D-con}, we have
\[
 \lt| \lt\lal \pa_x^\alpha\pa_p^\beta\Gamma(g^n), \pa_x^\alpha\pa_p^\beta g^{n+1} \rt\ral_{L^2_{x,p}} \rt|  \le \frac{\lambda_1}{8} \|\pa_x^\alpha\pa_p^\beta(I-P)g^{n+1}\|_D^2 + C \|\pa_x^\alpha g^{n+1}\|_{L^2_{x,p}}^2 + C \|g^n\|_{H^s}^2 \lt( \calD_n + \|g^n\|_{H^s}^2 \rt).
\]

Choosing $\eta>0$ sufficiently small and combining the above estimates, we obtain
\begin{align}\label{app:mixed-diff}
\begin{aligned}
&\frac{\rd}{\rdt} \|\pa_x^\alpha\pa_p^\beta g^{n+1}\|_{L^2_{x,p}}^2 + \lambda_2 \|\pa_x^\alpha\pa_p^\beta(I-P)g^{n+1}\|_D^2 \cr
&\quad\le C_\beta \|\pa_x^\alpha g^{n+1}\|_{L^2_{p}(H^1_x)}^2 + C_\beta \sum_{|\nu|\le|\beta|-1} \|\pa_x^\alpha\pa_p^\nu(I-P)g^{n+1}\|_D^2 \cr
&\quad \quad + C_\beta \sum_{\substack{1\le j\le3\\ \beta_j\ge1}} \|\pa_x^{\alpha+\mfe_j} \pa_p^{\beta-\mfe_j}(I-P)g^{n+1}\|_D^2 + C \|g^n\|_{H^s}^2 \lt( \calD_n + \|g^n\|_{H^s}^2 \rt)
\end{aligned}
\end{align}
for some constant $\lambda_2>0$.

We now sum the preceding estimates by means of a triangular induction over the number of $p$-derivatives. For each $1\le k\le s$, define
\[
\calE_{n+1,k}(t) := \sum_{\substack{|\alpha|+|\beta|\le s\\|\beta|=k}} \|\pa_x^\alpha\pa_p^\beta g^{n+1}(t)\|_{L^2_{x,p}}^2 
\quad \text{and} \quad 
\calD_{n+1,k}(t) := \sum_{\substack{|\alpha|+|\beta|\le s\\|\beta|=k}} \|\pa_x^\alpha\pa_p^\beta (I-P)g^{n+1}(t)\|_D^2.
\]
For the spatial level, we set
\[
\calE_{n+1,0}(t) := \|g^{n+1}(t)\|_{L^2_p(H^s_x)}^2
\quad \text{and} \quad 
\calD_{n+1,0}(t) := \sum_{|\alpha|\le s} \|\pa_x^\alpha(I-P)g^{n+1}(t)\|_D^2.
\]
These quantities have already been estimated in Step 1.

We first note that the lower $p$-order terms on the right-hand side of \eqref{app:mixed-diff} are controlled by the dissipation levels strictly below $|\beta|$. Indeed, for $|\beta|=k\ge1$, we get $|\nu| \le |\beta|-1 = k-1$, and hence
\bq\label{app:lower-p-order-sum}
\sum_{\substack{|\alpha|+|\beta|\le s\\|\beta|=k}} \sum_{|\nu|\le|\beta|-1} \|\pa_x^\alpha\pa_p^\nu(I-P)g^{n+1}\|_D^2 \le C_s \sum_{\ell=0}^{k-1} \calD_{n+1,\ell}.
\eq
Here and below, $C_s>0$ accounts for the finite multiplicity arising from the summation over multi-indices.

The transport commutator terms have the same structure. Since $|\alpha+\mfe_j|+|\beta-\mfe_j| = |\alpha|+|\beta| \le s$ and $|\beta-\mfe_j| = |\beta|-1 = k-1$, we have
\bq\label{app:tran-low-l}
\sum_{\substack{|\alpha|+|\beta|\le s\\|\beta|=k}} \sum_{\substack{1\le j\le3\\\beta_j\ge1}} \|\pa_x^{\alpha+\mfe_j} \pa_p^{\beta-\mfe_j}(I-P)g^{n+1}\|_D^2 \le C_s \calD_{n+1,k-1}.
\eq

Summing \eqref{app:mixed-diff} over $|\alpha|+|\beta|\le s$ with $|\beta|=k$, and using \eqref{app:lower-p-order-sum} and \eqref{app:tran-low-l}, we obtain
\bq\label{app:mixed-lev-k}
\frac{\rd}{\rdt} \calE_{n+1,k} + \lambda_2 \calD_{n+1,k} \le C_{s,k} \calE_{n+1,0} + C_{s,k} \sum_{\ell=0}^{k-1} \calD_{n+1,\ell} + C_{s,k} \|g^n\|_{H^s}^2 \lt( \calD_n + \|g^n\|_{H^s}^2 \rt)
\eq
for every $1\le k\le s$.

We now choose positive constants $1=a_0>a_1>\cdots>a_s>0$ recursively so that the lower $p$-order dissipation terms arising at level $k$ are absorbed by the dissipation terms at the levels $0,\ldots,k-1$. More precisely, after fixing $a_0=1$, we choose $a_k>0$ sufficiently small relative to $a_0,\ldots,a_{k-1}$ so that
\[
a_k C_{s,k} \le \frac{\lambda_*}{2s} \min_{0\le\ell\le k-1}a_\ell \quad \text{for each }1\le k\le s,
\]
where $\lambda_* := \min\{\lambda_0,\lambda_2\}$. With this choice, the lower-order dissipation terms satisfy
\[
\sum_{k=1}^s a_k C_{s,k} \sum_{\ell=0}^{k-1} \calD_{n+1,\ell} \le \frac{\lambda_*}{2} \sum_{\ell=0}^{s-1} a_\ell \calD_{n+1,\ell}.
\]
Thus, after multiplying \eqref{app:mixed-lev-k} by $a_k$, summing over $1\le k\le s$, and adding the spatial estimate \eqref{app:Hx-energy}, all terms involving strictly fewer $p$-derivatives are absorbed by the dissipation terms at the preceding levels.

Define
\[
\calE_{n+1}(t) := \sum_{k=0}^s a_k\calE_{n+1,k}(t) = \sum_{|\alpha|+|\beta|\le s} a_{|\beta|} \|\pa_x^\alpha\pa_p^\beta g^{n+1}(t)\|_{L^2_{x,p}}^2
\]
and
\[
\widetilde{\mathcal D}_{n+1}(t) := \sum_{k=0}^s a_k\calD_{n+1,k}(t) = \sum_{|\alpha|+|\beta|\le s} a_{|\beta|} \|\pa_x^\alpha\pa_p^\beta (I-P)g^{n+1}(t)\|_D^2.
\]
We then obtain
\bq\label{app:full-energy}
\frac{\rd}{\rdt} \calE_{n+1}(t) + c_s \widetilde{\mathcal D}_{n+1}(t) \le C_s \calE_{n+1}(t) + C_s \|g^n(t)\|_{H^s}^2 \lt( \calD_n(t) + \|g^n(t)\|_{H^s}^2 \rt)
\eq
for some positive constants $c_s$ and $C_s$. Since the weights $a_0,\ldots,a_s$ are fixed positive constants, the weighted energy $\calE_{n+1}$ is equivalent to $\|g^{n+1}\|_{H^s}^2$,  and the weighted dissipation $\widetilde{\mathcal D}_{n+1}$ is equivalent to $\calD_{n+1}$.
 
Integrating \eqref{app:full-energy} over $[0,t]\subset[0,T_*]$ and applying
Gr\"onwall's inequality, we obtain
\[
\sup_{0\le t\le T_*} \|g^{n+1}(t)\|_{H^s}^2 + \int_0^{T_*} \calD_{n+1}(\tau)\,\rdta \le C_s e^{C_sT_*} \lt( \|g_0\|_{H^s}^2 +  \delta_{\rm in}^2 + T_* \delta_{\rm in}^2 \rt) \le C_s e^{C_sT_*} \lt(  2 \delta_{\rm in}^2  + T_* \delta_{\rm in}^2 \rt)
\]
due to $\|g_0\|_{H^s}^2 \le  \delta_{\rm in}^2$. Choosing $ \delta_{\rm in}>0$ and $T_*>0$ sufficiently small so that $C_s e^{C_sT_*}(2+T_*) \delta_{\rm in} \le 1$, we obtain $g^{n+1}\in\mcI(s,T_*; \delta_{\rm in})$.
Since the first iterate is controlled by the corresponding homogeneous linear estimate, the result follows by induction.
\end{proof}

\begin{lemma}\label{lem:local-cau}
Let $\{g^n\}_{n\ge0}$ be the approximate sequence defined by \eqref{app:iter}. Under the assumptions of Lemma \ref{lem:local-uni-b}, $\{g^n\}_{n\ge0}$ is a Cauchy sequence in $\mcC([0,T_*];L^2(\R^3 \times \R^3))$. Moreover, $\{(I-P)g^n\}_{n\ge0}$ is a Cauchy sequence in the corresponding microscopic dissipation space.
\end{lemma}

\begin{proof}
Set $w^{n+1} := g^{n+1}-g^n$. For $n\ge1$, subtracting the equations for $g^{n+1}$ and $g^n$ gives
\[
\pa_t w^{n+1} = Bw^{n+1} + \Gamma(g^n)-\Gamma(g^{n-1}), \quad w^{n+1}|_{t=0}=0.
\]
Taking the $L^2_{x,p}$-inner product with $w^{n+1}$ and using the skew-symmetry of the transport operator together with the coercivity estimate \eqref{eq:L-coercivity}, we obtain
\begin{align}\label{app:cauchy-0}
\frac12 \frac{\rd}{\rdt} \|w^{n+1}\|_{L^2_{x,p}}^2 + \lambda_0 \|(I-P)w^{n+1}\|_D^2 \le \lt| \lt\lal \Gamma(g^n)-\Gamma(g^{n-1}), w^{n+1} \rt\ral_{L^2_{x,p}} \rt|.
\end{align}

Since $\Gamma(g)\in\mcN^\perp$, we get $\Gamma(g^n)-\Gamma(g^{n-1}) \in \mcN^\perp$. Thus,
\[
\lt\lal \Gamma(g^n)-\Gamma(g^{n-1}), w^{n+1} \rt\ral_{L^2_{x,p}} = \lt\lal \Gamma(g^n)-\Gamma(g^{n-1}), (I-P)w^{n+1} \rt\ral_{L^2_{x,p}}.
\]

Applying Lemma \ref{Gamma 1.5} with $g=g^n$, $h=g^{n-1}$, $\phi=(I-P)w^{n+1}$, we obtain
\begin{align*} 
&\lt| \lt\lal \Gamma(g^n)-\Gamma(g^{n-1}), (I-P)w^{n+1} \rt\ral_{L^2_{x,p}} \rt| \cr
&\quad\le C \lt[ \lt( \|g^n\|_{H^s} + \|g^{n-1}\|_{H^s} \rt) \|w^n\|_D + \lt( \|g^n\|_{D,s} + \|g^{n-1}\|_{D,s} \rt) \|w^n\|_{L^2_{x,p}} \rt] \|(I-P)w^{n+1}\|_D.
\end{align*}

By the finite-dimensionality of $P$ and the exponential decay of the basis functions of $\mcN$, we have $\|w^n\|_D \le \|(I-P)w^n\|_D + C\|w^n\|_{L^2_{x,p}}$. Moreover, by \eqref{app:full-D-con}, $\|g^k\|_{D,s}^2 \le C \lt( \calD_k + \|g^k\|_{H^s}^2 \rt)$. Using Lemma \ref{lem:local-uni-b}, we also have
\[
\sup_{0\le t\le T_*} \|g^k(t)\|_{H^s} \le \sqrt{ \delta_{\rm in}} \quad (k\ge0).
\]

Combining the above estimates and applying Young's inequality, we obtain
\bq\label{app:cau-non-b}
\lt| \lt\lal \Gamma(g^n)-\Gamma(g^{n-1}), (I-P)w^{n+1} \rt\ral_{L^2_{x,p}} \rt| \le \frac{\lambda_0}{2} \|(I-P)w^{n+1}\|_D^2 + C \delta_{\rm in} \|(I-P)w^n\|_D^2 + \calA_n(t) \|w^n\|_{L^2_{x,p}}^2,
\eq
where $\calA_n(t) := C \lt(  \delta_{\rm in} + \calD_n(t) + \calD_{n-1}(t) \rt)$. Substituting \eqref{app:cau-non-b} into \eqref{app:cauchy-0} and absorbing the first term on the right-hand side, we obtain
\bq\label{app:cauchy-diff}
 \frac{\rd}{\rdt} \|w^{n+1}\|_{L^2_{x,p}}^2 + \lambda_0 \|(I-P)w^{n+1}\|_D^2  \le C \delta_{\rm in} \|(I-P)w^n\|_D^2 + \calA_n(t) \|w^n\|_{L^2_{x,p}}^2.
\eq

By Lemma \ref{lem:local-uni-b},
\[
\int_0^{T_*} \calD_k(\tau)\,\rdta \le  \delta_{\rm in} \quad (k\ge0),
\]
and this gives
\[
\int_0^{T_*} \calA_n(\tau)\,\rdta \le C \delta_{\rm in}(1+T_*).
\]
Integrating \eqref{app:cauchy-diff} over $[0,t]$ and taking the supremum over $0\le t\le T_*$, we find
\begin{align}\label{app:cauchy-ineq}
\begin{aligned}
&\sup_{0\le t\le T_*} \|w^{n+1}(t)\|_{L^2_{x,p}}^2 + \lambda_0 \int_0^{T_*} \|(I-P)w^{n+1}(\tau)\|_D^2\,\rdta \cr
&\quad\le C \delta_{\rm in} \int_0^{T_*} \|(I-P)w^n(\tau)\|_D^2\,\rdta + C \delta_{\rm in}(1+T_*) \sup_{0\le t\le T_*} \|w^n(t)\|_{L^2_{x,p}}^2.
\end{aligned}
\end{align}

Define
\[
X_n := \sup_{0\le t\le T_*} \|w^{n+1}(t)\|_{L^2_{x,p}}^2 \quad \text{and} \quad Y_n := \int_0^{T_*} \|(I-P)w^{n+1}(\tau)\|_D^2\,\rdta.
\]
Then \eqref{app:cauchy-ineq} gives
\[
X_n+\lambda_0Y_n \le C \delta_{\rm in}(1+T_*) \lt( X_{n-1}+Y_{n-1} \rt).
\]
Choosing $ \delta_{\rm in}>0$ and $T_*>0$ sufficiently small, we may assume that
\[
C \delta_{\rm in}(1+T_*) \max\lt\{1,\frac1{\lambda_0}\rt\} \le \frac12.
\]
This implies
\bq\label{app:cauchy-cont}
X_n+\lambda_0Y_n \le \frac12 \lt( X_{n-1}+\lambda_0Y_{n-1} \rt).
\eq
Iterating this, we obtain
\[
X_n+\lambda_0Y_n \le 2^{-n} \lt( X_0+\lambda_0Y_0 \rt).
\]
In particular,
\[
\sum_{n\ge0} X_n^{\frac12} < \infty,
\]
and thus, for $m>n$,
\[
\sup_{0\le t\le T_*} \|g^m(t)-g^n(t)\|_{L^2_{x,p}} \le \sum_{k=n}^{m-1} \sup_{0\le t\le T_*} \|w^{k+1}(t)\|_{L^2_{x,p}} = \sum_{k=n}^{m-1} X_k^{\frac12} \to 0 \quad \text{as }n\to\infty.
\]
Hence, $\{g^n\}_{n\ge0}$ is a Cauchy sequence in $\mcC([0,T_*];L^2_{x,p})$. Similarly, it follows from \eqref{app:cauchy-cont} that $\{(I-P)g^n\}_{n\ge0}$ is a Cauchy sequence in the corresponding microscopic dissipation space. This completes the proof.
\end{proof}

%
%
%
%
%
%
%
%
%
%

\subsection{Propagation of the admissible bounds} We next show that the admissible bounds on the initial distribution are propagated by the local solution. In particular, nonnegative initial data remain nonnegative. In the fermionic case, the upper Pauli bound is also preserved.

\begin{lemma} \label{lem:local-pos-pauli}
Let $g$ be the local solution obtained in Theorem \ref{LE}, and set $f:=\mcF_\hbar +g\sqrt{\mu_\hbar }$. If $f_0 := \mcF_\hbar +g_0\sqrt{\mu_\hbar } \ge0$ almost everywhere, then $f(t,x,p)\ge0$ for almost every $(x,p)\in\mR^3\times\mR^3$ and all $0\le t\le T_*$. Moreover, in the fermionic case $\kappa=-1$, if $f_0(x,p)\le\frac1{\hbar}$ almost everywhere, then $ f(t,x,p)\le\frac1{\hbar}$ for almost every $(x,p)\in\mR^3\times\mR^3$ and all $0\le t\le T_*$.
\end{lemma}

\begin{proof}
The computations below are justified by using smooth convex approximations of the relevant truncation functions and standard cutoff functions in $(x,p)$. We first derive the corresponding time-integrated estimates and then pass to the limits in the approximation and cutoff parameters.

Since $s\ge4$, the Sobolev embedding and \eqref{eq:local-b-main} give 
\[
\sup_{0\le t\le T_*} \|g(t)\|_{L^\infty_{x,p}} \le \sup_{0\le t\le T_*} \|g(t)\|_{H^s} \le \sqrt{ \delta_{\rm in}}.
\]
Thus, the macroscopic coefficients remain close to their equilibrium values. In particular, using
\[
\rho \Theta = M + \frac13 \intr |p|^2\eta_\hbar g\sqrt{\mu_\hbar }\,\rdp + N_\Theta
\]
and Lemma \ref{lem:Nu-NT-est}, we obtain
\[
|\rho \Theta-M| \le C \lt( \|g\|_{L^\infty_{x,p}} + \|g\|_{L^\infty_{x,p}}^2 \rt) \le C\sqrt{ \delta_{\rm in}}.
\]
Hence, after reducing $ \delta_{\rm in}>0$ if necessary,
\bq\label{app:rhoT-pos}
\rho \Theta \ge \frac{M}{2} > 0.
\eq
The same perturbative estimates imply
\bq\label{app:rho-bdd}
\|\rho\|_{L^\infty_x} \le C.
\eq

We first prove the nonnegativity of $f$. Recall that $f$ satisfies
\bq\label{app:non-fp-eq}
\pa_t f + p\cdot\nabla_x f = \rho \nabla_p\cdot \lt\{ \Theta\nabla_p f + (p-u)f(1+\hbar\kappa f) \rt\}.
\eq
Set $f_- := \max\{-f,0\}$. Since $f = \mcF_\hbar +\sqrt{\mu_\hbar }g$ and $\mcF_\hbar \ge0$, we have $f_- \le \sqrt{\mu_\hbar }|g|$. Indeed, the estimate is immediate on the set $\{f\ge0\}$. On the set
$\{f<0\}$, we have $g<0$ and 
\[
f_- = -\mcF_\hbar -\sqrt{\mu_\hbar }g \le -\sqrt{\mu_\hbar }g = \sqrt{\mu_\hbar }|g|. 
\]
Since $\mu_\hbar $ is bounded, it follows that $f_-\in L^2_{x,p}$. Moreover,
\bq\label{app:fm-L}
\|f_-\|_{L^\infty_{x,p}} \le C\sqrt{ \delta_{\rm in}}.
\eq

Let $\chi_\delta\in C^2(\mR)$ be a nonnegative convex function satisfying 
\[
\chi_\delta(z)=0
\quad
\text{for }z\ge0,
\quad
0\le\chi_\delta''(z)\le1, 
\quad \text{and} \quad 
\chi_\delta(z) \to \frac12(z_-)^2 \quad \text{as }\delta\to0.
\]
We may choose $\chi_\delta$ so that $\chi_\delta''(z) \to \mathbf 1_{\{z<0\}}$ as $\delta\to 0$ for a.e. $z \in \R$. Multiplying \eqref{app:non-fp-eq} by
$\chi_\delta'(f)$ and integrating by parts in $(x,p)$, we obtain
\bq\label{app:conv-neg-p}
 \frac{\rd}{\rdt} \inttr \chi_\delta(f) \,\rdx\rdp  + \inttr \rho \Theta \chi_\delta''(f) |\nabla_p f|^2 \,\rdx\rdp  = 3 \inttr \rho \Psi_\delta(f) \,\rdx\rdp ,
\eq
where
\[
\Psi_\delta(z) := \int_0^z s(1+\hbar\kappa s) \chi_\delta''(s) \,\rds.
\]
Here we used the fact that $\rho$ and $u$ are independent of $p$ and $\nabla_p\cdot(p-u)=3$.

We now integrate \eqref{app:conv-neg-p} over $[0,t]$. By \eqref{app:rhoT-pos} and the convexity of $\chi_\delta$, the diffusion term is nonnegative. Hence,
\bq\label{app:conv-neg-p-i}
\inttr \chi_\delta(f(t)) \,\rdx\rdp   \le \inttr \chi_\delta(f_0) \,\rdx\rdp  + 3 \int_0^t \inttr \rho \Psi_\delta(f) \,\rdx\rdp \rdta.
\eq
Since $f_0\ge0$, we have $\chi_\delta(f_0)=0$. Moreover, for every $z\in\mR$, $\Psi_\delta(z) \to \Psi_-(z)$ as $\delta \to 0$, where
\[
\Psi_-(z) :=
\begin{cases}
  \frac12z^2+\frac{\hbar\kappa}{3}z^3, & z<0, \\[2mm]
0, & z\ge0,
\end{cases}
\quad \text{equivalently, } 
\Psi_-(z) = \frac12z_-^2 - \frac{\hbar\kappa}{3}z_-^3.
\]
Since $0\le\chi_\delta''\le1$, we also have $|\Psi_\delta(f)| \le C( f_-^2+f_-^3)$, and by \eqref{app:fm-L}, $f_-^3 \le C\sqrt{ \delta_{\rm in}}\,f_-^2$. Together with \eqref{app:rho-bdd}, this gives $|\rho\Psi_\delta(f)| \le C f_-^2$. The right-hand side is integrable on $[0,T_*]\times\R^3 \times \R^3$. Thus, by dominated convergence and Fatou's lemma, we may let $\delta\to0$ in \eqref{app:conv-neg-p-i} to obtain
\[
\frac12 \|f_-(t)\|_{L^2_{x,p}}^2  \le 3 \int_0^t \inttr \rho \lt( \frac12f_-^2 - \frac{\hbar\kappa}{3}f_-^3 \rt) \,\rdx\rdp \rdta  \le C \int_0^t \|f_-(\tau)\|_{L^2_{x,p}}^2 \,\rdta.
\]
By Gr\"onwall's inequality, $f_-(t)=0$ for all $0\le t\le T_*$. Hence, we have $f(t,x,p)\ge0$ almost everywhere.

We next prove the fermionic upper bound. Assume that $\kappa=-1$. Since the fermionic equilibrium satisfies $0<\mcF_\hbar (p)\le\frac1{\hbar}$, we have $\lt( f-\frac1{\hbar} \rt)_+ \le \sqrt{\mu_\hbar }|g|$. In particular, $\lt( f-\frac1{\hbar} \rt)_+ \in L^2_{x,p}$.

We proceed as in the proof of the nonnegativity of $f$. Let $\chi_\delta^\hbar\in C^2(\mR)$ be a nonnegative convex function satisfying
\[
\chi_\delta^\hbar(z)=0 \quad \text{for }z\le\frac1{\hbar},
\quad
0\le (\chi_\delta^\hbar)''(z) \le1,
\quad \text{and} \quad 
\chi_\delta^\hbar(z) \to \frac12 \lt( z-\frac1{\hbar} \rt)_+^2 \quad \text{as }\delta\to0.
\]
Multiplying the fermionic equation by $(\chi_\delta^\hbar)'(f)$, integrating by parts in $(x,p)$, and then integrating in time over $[0,t]$, we obtain
\begin{align}\label{app:conv-pauli-i}
\begin{aligned}
&\inttr \chi_\delta^\hbar(f(t)) \,\rdx\rdp  + \int_0^t \inttr \rho \Theta (\chi_\delta^\hbar)''(f) |\nabla_p f|^2 \,\rdx\rdp \rdta \cr
&\quad = \inttr \chi_\delta^\hbar(f_0) \,\rdx\rdp  + 3 \int_0^t \inttr \rho \Psi_\delta^\hbar(f) \,\rdx\rdp \rdta,
\end{aligned}
\end{align}
where
\[
\Psi_\delta^\hbar(z) := \int_{\frac1\hbar}^z s(1-\hbar s) (\chi_\delta^\hbar)''(s) \,\rds.
\]
Since $f_0\le\frac1{\hbar}$, we get $\chi_\delta^\hbar(f_0)=0$. Moreover, on the support of $(\chi_\delta^\hbar)''$, we find $s\ge\frac1{\hbar}$, and thus $s(1-\hbar s)\le0$. This yields $\Psi_\delta^\hbar(z)\le0$ for every $z \in \R$. Since the nonnegativity of $f$ has already been established, we also obtain $\rho = \intr f\,\rdp \ge0$. Consequently, the right-hand side of
\eqref{app:conv-pauli-i} is nonpositive, whereas its left-hand
side is nonnegative. It follows that
\[
\inttr \chi_\delta^\hbar(f(t)) \,\rdx\rdp  = 0.
\]
Letting $\delta\to0$ and using Fatou's lemma, we obtain
\[
\lt\| \lt( f(t)-\frac1{\hbar} \rt)_+ \rt\|_{L^2_{x,p}}^2 = 0.
\]
Hence, we have $f(t,x,p)\le\frac1{\hbar}$ almost everywhere for all $0\le t\le T_*$. This completes the proof.
\end{proof}

%
%
%
%
%
%
%
%
%
%
\section{Details of the macroscopic dissipation estimate}
\label{app:mac-diss}

In this appendix, we provide the coefficient-level details used in the proof of Lemma \ref{lem:mac-reco-s5}. The argument is based on the finite-dimensional expansion of $(\pa_t+p\cdot\nabla_x)Pg=\ell+\mathfrak h$, where
\[
    \ell= \lt\{ -\pa_t-p\cdot\nabla_x+L \rt\}(I-P)g, \quad \mathfrak h=\Gamma(g).
\]
Using the coefficient relations \eqref{eq:mac-rel-s5}, we derive the dissipation estimates for the spatial derivatives of the macroscopic variables $a$, $b$, and $c$.

%
%
%
%
%
%
%
%
%
%

\subsection{Estimates for the macroscopic variables}

We first give the details for the estimate of $\nabla_x b$. Recall from Lemma \ref{lem:lh-s5} that
\bq\label{app:cb_rela}
\pa_t c+\pa_{x_i}b_i = l_i+h_i, \quad \pa_{x_i}b_j+\pa_{x_j}b_i = l_{ij}+h_{ij}, \quad i\neq j.
\eq

\begin{lemma}\label{app:blh-section5}
Let $|\alpha|\le s-1$. Then, for each $i=1,2,3$,
\[
\|\pa_x^\alpha\nabla_x b_i\|_{L^2_x}^2 \le C \sum_{j\neq i} \lt| \intr \pa_x^\alpha l_{ij}\cdot \pa_x^\alpha\nabla_x b_i\,\rdx \rt| + C \sum_{j=1}^3 \lt| \intr \pa_x^\alpha l_j\cdot \pa_x^\alpha\nabla_x b_i\,\rdx \rt| + C\|\pa_x^\alpha\tilde h\|_{L^2_x}^2.
\]
Consequently, using the representation \eqref{eq:form-I-s5}, the right-hand side can be written in the form used in \eqref{eq:mac-rec-s5}.
\end{lemma}

\begin{proof}
Fix $i\in\{1,2,3\}$. We compute
\[
 \Delta_x b_i = \pa_{x_i}^2 b_i + \sum_{j\neq i}\pa_{x_j}^2 b_i.
\]
For $j\neq i$, the relation \eqref{app:cb_rela} gives
\[
\pa_{x_j}^2 b_i = \pa_{x_j}(l_{ij}+h_{ij}) - \pa_{x_i}\pa_{x_j}b_j \quad \text{and} \quad \pa_{x_j}b_j =  l_j+h_j-\pa_t c.
\]
Thus, we obtain
\[
\pa_{x_j}^2 b_i = \pa_{x_j}(l_{ij}+h_{ij}) - \pa_{x_i}(l_j+h_j) + \pa_{x_i}\pa_t c.
\]
Summing over $j\neq i$, we obtain
\[
\Delta_x b_i = \pa_{x_i}^2 b_i + \sum_{j\neq i} \lt\{ \pa_{x_j}(l_{ij}+h_{ij}) - \pa_{x_i}(l_j+h_j) + \pa_{x_i}\pa_t c \rt\}.
\]
Using again \eqref{app:cb_rela}, we get
\[
\pa_{x_i}\pa_t c = \pa_{x_i}(l_i+h_i) - \pa_{x_i}^2b_i.
\]
This yields
\[
\Delta_x b_i = -\pa_{x_i}^2b_i + \sum_{j\neq i} \lt\{ \pa_{x_j}(l_{ij}+h_{ij}) - \pa_{x_i}(l_j+h_j) + \pa_{x_i}(l_i+h_i) \rt\}.
\]
Applying $\pa_x^\alpha$, multiplying by $\pa_x^\alpha b_i$, integrating in $x$, and integrating by parts, we obtain
\[
\|\pa_x^\alpha\nabla_x b_i\|_{L^2_x}^2  \le C \sum_{j \neq i} \lt| \intr \pa_x^\alpha l_{ij}\cdot \pa_x^\alpha\nabla_xb_i\,\rdx \rt| + C \sum_{j=1}^3 \lt| \intr \pa_x^\alpha l_j\cdot \pa_x^\alpha\nabla_xb_i\,\rdx \rt| + C\|\pa_x^\alpha\tilde h\|_{L^2_x} \|\pa_x^\alpha\nabla_xb_i\|_{L^2_x}.
\]
Applying Young's inequality to the above concludes the desired result.
\end{proof}

%
%
%
%
%
%
%
%
%
%
The estimate for $c$ is more direct. From Lemma \ref{lem:lh-s5}, $\nabla_x c=l_c+h_c$.

\begin{lemma}\label{app:clh-section5}
Let $|\alpha|\le s-1$. Then
\bq\label{app:c-estimate-section5}
\|\pa_x^\alpha\nabla_x c\|_{L^2_x}^2 \le C \intr \pa_x^\alpha l_c\cdot \pa_x^\alpha\nabla_xc\,\rdx + C\|\pa_x^\alpha\tilde h\|_{L^2_x}^2.
\eq
\end{lemma}

\begin{proof}
Applying $\pa_x^\alpha$ to $\nabla_x c=l_c+h_c$ and taking the $L^2_x$-inner product with $\pa_x^\alpha\nabla_x c$, we obtain
\[
\|\pa_x^\alpha\nabla_xc\|_{L^2_x}^2 = \intr \pa_x^\alpha l_c\cdot\pa_x^\alpha\nabla_xc\,\rdx + \intr \pa_x^\alpha h_c\cdot\pa_x^\alpha\nabla_xc\,\rdx.
\]
The second term is bounded by $C\|\pa_x^\alpha\tilde h\|_{L^2_x} \|\pa_x^\alpha\nabla_xc\|_{L^2_x}$, and Young's inequality gives \eqref{app:c-estimate-section5}.
\end{proof}

%
%
%
%
%
%
%
%
%
%

We finally estimate the gradient of $a$. From Lemma \ref{lem:lh-s5}, $\pa_t b_i+\pa_{x_i}a=l_{bi}+h_{bi}$. Taking the divergence gives
\[
\pa_t\nabla_x\cdot b+\Delta_x a = \nabla_x\cdot l_b+\nabla_x\cdot h_b,
\]
where $l_b=(l_{b1},l_{b2},l_{b3})$ and $h_b=(h_{b1},h_{b2},h_{b3})$.

\begin{lemma}\label{app:alh-section5}
Let $|\alpha|\le s-1$. Then, we have
\begin{align*} 
\|\pa_x^\alpha\nabla_x a\|_{L^2_x}^2
&\le -\frac{\rd}{\rdt} \intr \pa_x^\alpha b\cdot \pa_x^\alpha\nabla_xa\,\rdx + C \intr \pa_x^\alpha l_b\cdot \pa_x^\alpha\nabla_xa\,\rdx + C \|\pa_x^\alpha\nabla_x b\|_{L^2_x}^2 \cr
&\quad + C \|\pa_x^\alpha\nabla_x(I-P)g\|_{L^2_{x,p}}^2 + C\|\pa_x^\alpha\tilde h\|_{L^2_x}^2.
\end{align*}
\end{lemma}

\begin{proof}
Applying $\pa_x^\alpha$ to $\pa_t b+\nabla_x a=l_b+h_b$ and taking the $L^2_x$-inner product with $\pa_x^\alpha\nabla_xa$, we obtain
\[
\|\pa_x^\alpha\nabla_xa\|_{L^2_x}^2 = -\intr \pa_t\pa_x^\alpha b\cdot \pa_x^\alpha\nabla_xa\,\rdx + \intr \pa_x^\alpha l_b\cdot \pa_x^\alpha\nabla_xa\,\rdx + \intr \pa_x^\alpha h_b\cdot \pa_x^\alpha\nabla_xa\,\rdx.
\]
The time derivative term is rewritten as
\begin{align*}
-\intr \pa_t\pa_x^\alpha b\cdot \pa_x^\alpha\nabla_xa\,\rdx = -\frac{\rd}{\rdt} \intr \pa_x^\alpha b\cdot \pa_x^\alpha\nabla_xa\,\rdx + \intr \pa_x^\alpha b\cdot \pa_t\pa_x^\alpha\nabla_xa\,\rdx.
\end{align*}
Here, the last term can be estimated as 
\[
\lt|\intr \pa_x^\alpha b\cdot \pa_t\pa_x^\alpha\nabla_xa\,\rdx\rt| =  \lt| \intr \nabla_x\cdot\pa_x^\alpha b\, \pa_t\pa_x^\alpha a\,\rdx \rt| \le \eta\|\pa_x^\alpha\nabla_xb\|_{L^2_x}^2 + C_\eta\|\pa_x^\alpha\nabla_x(I-P)g\|_{L^2_{x,p}}^2,
\]
for any $\eta>0$, where we used $\|\pa_t\pa_x^\alpha a\|_{L^2_x} \le C\|\pa_x^\alpha\nabla_x(I-P)g\|_{L^2_{x,p}}$ due to \eqref{eq:abc_t}.
The $h_b$-term is controlled by $C\|\pa_x^\alpha\tilde h\|_{L^2_x} \|\pa_x^\alpha\nabla_xa\|_{L^2_x}$, and hence by Young's inequality. Combining these estimates concludes the desired result.
\end{proof}

%
%
%
%
%
%
%
%
%
%

\subsection{Estimate of the \texorpdfstring{$\ell$-terms}{ell-terms}}

We now estimate the terms involving the $l$-coefficients. By \eqref{eq:form-I-s5}, every $l$-coefficient is a finite linear combination of moments of
\[
\lt\{ -\pa_t-p\cdot\nabla_x+L \rt\}(I-P)g
\]
against fixed smooth exponentially decaying functions of $p$. The time-derivative parts produce the interaction functional, while the remaining terms are controlled by the microscopic norms.

Let $\calL_\alpha$ denote the signed finite linear combination of the $l$-coefficient contributions arising in Lemmas \ref{app:blh-section5}, \ref{app:clh-section5}, and \ref{app:alh-section5}, with the fixed coefficients inherited from the preceding elliptic estimates.

\begin{lemma}\label{app:ell-terms-s5}
Let $|\alpha|\le s-1$. Then, we have
\begin{align}\label{app:ell-est-s5}
\calL_\alpha \le -\frac{\rd}{\rdt}\calJ_\alpha(t) + C \|\pa_x^\alpha\nabla_x(I-P)g\|_{L^2_{x,p}}^2 + C \|\pa_x^\alpha(I-P)g\|_{L^2_{x,p}}^2 + \eta \|\pa_x^\alpha\nabla_x(a,b,c)\|_{L^2_x}^2,
\end{align}
where $\eta>0$ is arbitrary and $\calJ_\alpha$ is a finite linear combination of interaction terms of the form
\[
\inttr \pa_x^\alpha(I-P)g\, \xi(p)\, \pa_x^\alpha\nabla_xq \,\rdp\rdx , \quad q\in\{a,b,c\}.
\]
Here $\xi$ ranges over a fixed finite family of smooth exponentially decaying functions of $p$.
\end{lemma}

\begin{proof}
It suffices to estimate a representative signed term. Set $r := (I-P)g$, and let $\xi=\xi(p)$ be one of the fixed smooth exponentially decaying functions appearing in \eqref{eq:form-I-s5}. For
$q\in\{a,b,c\}$, consider
\[
\mathfrak L_{\alpha,q,\xi} := \inttr \lt\{ -\pa_t-p\cdot\nabla_x+L \rt\} \pa_x^\alpha r\, \xi(p)\, \pa_x^\alpha\nabla_xq \,\rdp\rdx .
\]
We preserve the sign of the time-derivative contribution and estimate only the remaining terms in absolute value.

\medskip
\noindent
{\it Estimate of the time-derivative term.} Note that
\[
- \inttr \pa_t\pa_x^\alpha r\, \xi(p)\,\pa_x^\alpha\nabla_xq \,\rdp\rdx  = - \frac{\rd}{\rdt} \inttr \pa_x^\alpha r\, \xi(p)\, \pa_x^\alpha\nabla_xq \,\rdp\rdx  + \inttr \pa_x^\alpha r\, \xi(p)\, \pa_t\pa_x^\alpha\nabla_xq \,\rdp\rdx .
\]
To estimate the last term without losing one spatial derivative, we
integrate by parts in $x$:
\[
\inttr \pa_x^\alpha r\, \xi(p)\, \pa_t\pa_x^\alpha\nabla_xq \,\rdp\rdx  = - \intr \nabla_x\cdot \lt( \intr  \pa_x^\alpha r\,\xi(p)\,\rdp \rt) \pa_t\pa_x^\alpha q \,\rdx.
\]
By the exponential decay of $\xi$,
\[
\lt\| \nabla_x\cdot \lt( \intr  \pa_x^\alpha r\,\xi(p)\,\rdp \rt) \rt\|_{L^2_x} \le C \|\pa_x^\alpha\nabla_xr\|_{L^2_{x,p}}.
\]

The balance laws \eqref{eq:abc_t} imply
\[
\|\pa_t\pa_x^\alpha(a,b,c)\|_{L^2_x} \le C \|\pa_x^\alpha\nabla_x(a,b,c)\|_{L^2_x} + C \|\pa_x^\alpha\nabla_xr\|_{L^2_{x,p}}.
\]
Indeed, the equation for $\pa_ta$ contains only a moment of $p\cdot\nabla_xr$, the equation for $\pa_tb$ contains $\nabla_x(a,c)$ and a moment of $p\cdot\nabla_xr$, and the equation for $\pa_tc$ contains $\nabla_x\cdot b$ and a moment of $p\cdot\nabla_xr$.

Combining these estimates and applying Young's inequality, we obtain
\[
\lt| \inttr \pa_x^\alpha r\, \xi(p)\, \pa_t\pa_x^\alpha\nabla_xq \,\rdp\rdx  \rt| \le C \|\pa_x^\alpha\nabla_xr\|_{L^2_{x,p}}^2 + \eta \|\pa_x^\alpha\nabla_x(a,b,c)\|_{L^2_x}^2.
\]

\medskip
\noindent
{\it Estimate of the transport term.} Using the exponential decay of $\xi$ and Cauchy's inequality, we have
\[
\lt| \inttr p\cdot\nabla_x\pa_x^\alpha r\, \xi(p)\, \pa_x^\alpha\nabla_xq \,\rdp\rdx  \rt| \le C \|\pa_x^\alpha\nabla_xr\|_{L^2_{x,p}} \|\pa_x^\alpha\nabla_xq\|_{L^2_x} \le C_\eta \|\pa_x^\alpha\nabla_xr\|_{L^2_{x,p}}^2 + \eta \|\pa_x^\alpha\nabla_xq\|_{L^2_x}^2.
\]

\medskip
\noindent
{\it Estimate of the $L$-term.} Since $L$ is self-adjoint and $\xi$ is a fixed smooth exponentially decaying function of $p$, we have
\[
\intr  L(\pa_x^\alpha r)\,\xi(p)\,\rdp = \intr  \pa_x^\alpha r\,L\xi(p)\,\rdp.
\]
Moreover, $L\xi \in L^2_p$, and hence
\[
\lt| \inttr L(\pa_x^\alpha r)\, \xi(p)\, \pa_x^\alpha\nabla_xq \,\rdp\rdx  \rt| \le C \|\pa_x^\alpha r\|_{L^2_{x,p}} \|\pa_x^\alpha\nabla_xq\|_{L^2_x} \le C_\eta \|\pa_x^\alpha r\|_{L^2_{x,p}}^2 + \eta \|\pa_x^\alpha\nabla_xq\|_{L^2_x}^2.
\]
Combining the above estimates and summing over the finitely many coefficient functionals and macroscopic variables, we obtain \eqref{app:ell-est-s5}. This completes the proof.
\end{proof}

%
%
%
%
%
%
%
%
%
%

\bibliographystyle{abbrv}
\bibliography{NQFP_ref}

%
%
%
%
%
%
%
%
%
%

\end{document}